\documentclass[11pt, oneside, titlepage]{book}
\usepackage[utf8]{inputenc}

\usepackage[a4paper]{geometry}
\usepackage[export]{adjustbox}
\usepackage{changepage}

\usepackage{xargs}                      
\usepackage[pdftex,dvipsnames]{xcolor}  
\usepackage[colorinlistoftodos,prependcaption,textsize=tiny, disable]{todonotes} 
\usepackage{xcolor}

\definecolor{beamerblue}{RGB}{51,51,178}
\definecolor{targetcolour}{RGB}{51,51,178}
\definecolor{sourcecolour}{RGB}{250, 65, 65}
\definecolor{mixed}{RGB}{255,0,0}
\definecolor{amaranth}{rgb}{0.9, 0.17, 0.31}
\definecolor{amber}{RGB}{255, 148,0}
\definecolor{amber2}{rgb}{1.0, 0.49, 0.0}

\usepackage{titlesec}

\titleformat{\chapter}[hang]
{\huge\bfseries}
{\thechapter}
{1pc}
{\huge}
\titlespacing{\chapter}{0pt}{0pt}{30pt}

\usepackage{setspace, graphics}
\usepackage[nottoc,notlot,notlof]{tocbibind} 

\usepackage{graphicx}
\usepackage{amsmath, amsfonts, amsthm, amssymb}
\usepackage{mathrsfs}
\usepackage{mathtools}
\usepackage[shortlabels]{enumitem}
\setlist[enumerate]{nosep}

\theoremstyle{plain}
\newtheorem{theorem}{Theorem}[chapter]
\newtheorem{corollary}{Corollary}[theorem]
\newtheorem{lemma}[theorem]{Lemma}
\newtheorem{proposition}[theorem]{Proposition}
\newtheorem*{theorem*}{Theorem}
\newtheorem{definition}[theorem]{Definition}
\theoremstyle{definition}
\newtheorem{example}[theorem]{Example}
\newtheorem{problem}{Problem}
\newtheorem{assumption}{Assumption}

\newtheorem{remark}[theorem]{Remark}

\newcommand{\R}{\mathbb{R}}
\newcommand{\N}{\mathbb{N}}
\newcommand{\K}{\mathcal{K}}
\newcommand{\C}{\mathcal{C}}
\newcommand{\X}{\mathcal{X}}
\newcommand{\Y}{\mathcal{Y}}
\newcommand{\F}{\mathcal{F}}

\newcommand{\di}{\mathrm{d}}
\newcommand{\E}{\mathcal{E}}
\DeclareMathOperator{\Int}{int}
\DeclareMathOperator{\ri}{ri}
\DeclareMathOperator{\aff}{aff}

\DeclareMathOperator{\conv}{conv}
\DeclareMathOperator{\id}{Id}
\DeclareMathOperator{\dom}{dom}
\DeclareMathOperator{\var}{Var}
\DeclareMathOperator{\supp}{Supp}
\DeclareMathOperator{\diam}{Diam}

\DeclareMathOperator*{\argmin}{arg\,\min}
\DeclareMathOperator*{\argmax}{arg\max}

\DeclareMathOperator{\lag}{Lag}
\usepackage{derivative}
\usepackage[shortlabels]{enumitem}
\setlist[enumerate]{nosep}

\def\XXint#1#2#3{{\setbox0=\hbox{$#1{#2#3}{\int}$ }
\vcenter{\hbox{$#2#3$ }}\kern-.6\wd0}}

\usepackage{tikz}
\usetikzlibrary{shapes.misc}
\usetikzlibrary{shapes.misc}
\usetikzlibrary{arrows.meta,calc}
\usetikzlibrary{decorations.pathreplacing}
\usetikzlibrary{positioning}
\tikzset{cross/.style={cross out, draw=black, minimum size=2*(#1-\pgflinewidth), inner sep=0pt, outer sep=0pt},
cross/.default={1pt}}

\newcounter{step}

\usepackage{nomentbl}

\usepackage{hyperref}

\begin{document}

\frontmatter
    
\thispagestyle{empty}
\begin{center}
\begin{minipage}{\linewidth}
    \centering
    \includegraphics[width=0.8\linewidth]{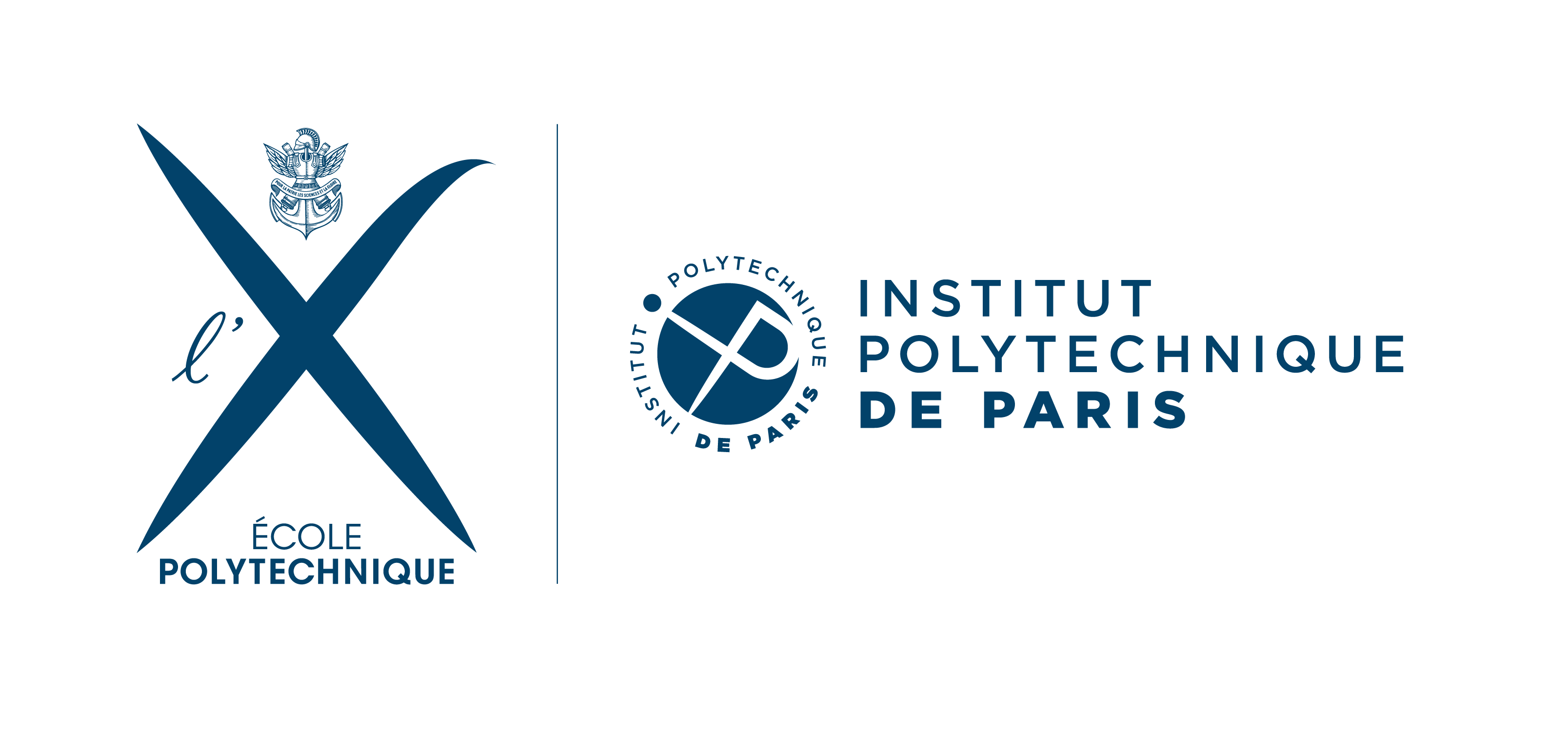}
    \vspace{3cm}


\Large{IP Paris M2 Optimisation Internship, completed at the Centre de Mathématiques Appliquées de l'École Polytechnique}\\
    
\rule{\linewidth}{0.3mm}\\
\vspace{0.3cm}
{\textbf{\LARGE Quantitative Stability in Discrete Optimal Transport}\par}
\rule{\linewidth}{0.3mm}
\vspace{0cm}

\textit{Author:}
\hfill
\textit{Directors:}\\
William Ford 
\hfill
Michael Goldman et Cyril Letrouit\\

\vspace{6cm}
{\Large September 18, 2025}
\end{minipage}
\end{center}

\clearpage
\vfill

\newpage
\vspace*{0.5\textheight}
\textit{To my grandma Josie, the kindest woman I have ever known.}
\newpage

{\textbf{\huge Acknowledgements}}\\

\textit{Tout d'abord, je présente mes remerciements à mes encadrants, Michael Goldman et Cyril Letrouit. Ça a été une plaisir de travailler avec vous pendant cette année, et j'ai hâte de continuer ce voyage avec vous durant mon doctorat.}

\textit{Thanks to my parents, for their love, support and counsel, and who always pushed me to achieve my best. I do not comment on who gave me my ``maths gene". Thank you to Toby, with whom I first really fell in love with mathematics through our countless discussions.}

\textit{I also want to thank all my friends who made my first year in Paris during this M2 so enjoyable. In particular, I would like to thank Aryna, Enzo, Ester, Francesca, Guido, Jacob, Jacopo, Jules, Michele, Pablo, Petra, Raphaël, Timur, Thomas, and Yegor. I probably would have done a lot more maths but had much less fun if it weren't for all of you.}

\vfill

\vspace*{\fill}

\begin{spacing}{-0.3}
    \parskip 0.7mm
    \tableofcontents
\end{spacing}

\mainmatter

\chapter{Introduction}
\label{ch: intro}

\section{Overview}

The optimal transport problem provides a rigorous framework for comparing probability measures, and has become a central tool in modern analysis and probability, as well as having found wide applications ranging from fluid mechanics and economics to data science and imaging. This report is concerned with the stability properties of optimal transport. While qualitative stability is well understood, quantitative estimates remain a subject of active research. Such estimates are crucial in applied contexts where data are inevitably discretised or noisy. The structure of the report is as follows:

In Chapter \ref{ch: intro}, we provide a brief summary of the optimal transport theory, then provide some motivation for the quantitative stability problem. Notably, we emphasise the necessity of studying the stability of \textit{optimal plans} rather than maps. We formulate the main problem, providing some preliminary results in this direction. 

In Chapter \ref{ch: ac stability}, we review some of the current literature on quantitative stability with a source fixed absolutely continuous, as well as providing some extensions of a certain result to a larger class of measures. 

In Chapter \ref{ch: discrete transport problem}, we shift towards the main theme of this report. Our approach is to begin to study the fully discrete setting, where the transport problem admits a linear programming formulation with a rich combinatorial structure.

In Chapter \ref{ch: pertubations}, we provide a (partial) analysis of quantitative stability properties of discrete optimal plans under perturbations.

In Chapter \ref{ch: uniqueness of dual problem}, we use the intuition given by the discrete transport problem to provide a sharper understanding of optimal transport in the general (non-discrete) case, significantly refining known uniqueness criteria for Kantorovich potentials (optimisers of the dual problem).

\section{Background on optimal transport}

Let $\rho$ and $\mu$ be two probability measures supported inside $\X, \Y \subset \R^d$. Let $c : \X \times \Y \to \R_+$  be lower semi-continuous and integrable in the sense that there exist $a \in L^1(\rho), b \in L^1(\mu)$ such that $c(x, y) \leq a(x) + b(y)$ pointwise. The optimal transport problem is the minimisation problem
\begin{equation}
\label{intro: c OT problem}
    \min_{\gamma \in \Pi(\rho, \mu)} \int_{\X \times \Y} c(x,y) \di \gamma(x, y),
\end{equation}
where $\Pi(\rho, \mu)$ is the set of probabilities over $\R^d \times \R^d$ with first marginal $\rho$ and second marginal $\mu$. Under our hypotheses on $c$, the existence of solutions to \eqref{intro: c OT problem} is classical, see \cite[Theorem 1.3]{villani2021topics}. As a linear optimisation problem with convex constraints, problem \eqref{intro: c OT problem} admits the following dual formulation
\begin{equation}
\label{intro: c dual problem}
    \max_{\substack{\phi, \psi \\ \phi \oplus \psi \leq c}} \int_\X \phi(x) \di \rho(x) + \int_\Y \psi(y) \di \mu(y).
\end{equation}
Here the supremum is taken over all $\phi \in L^1(\X, \rho)$ and $\psi \in L^1(\Y, \mu)$ such that 
\begin{equation}
    \label{intro: dual constraint}
    \phi(x) + \psi(y) \leq c(x, y) \text{  for all  } (x,y) \in \X \times \Y.
\end{equation}
By a density argument (See Lusin's theorem \cite[Theorem 2.24]{rudin1987real} and its corollary), the value of \eqref{intro: c dual problem} is equal to if we took the supremum only over $\phi \in C_b(\X)$ and $\psi \in C_b(\Y)$. Under our hypotheses on $c$ we have existence of optimisers for \eqref{intro: c dual problem} in $L^1(\rho), L^1(\mu)$ (but not necessarily $C_b$), see \cite[Theorem 5.10]{villani2008optimal}. Given $\phi, \psi$ satisfying \eqref{intro: dual constraint}, we define their $c$-transforms as 
\begin{equation}
\label{intro: c transf definition}
    \phi^c(y) := \inf_{x \in \X} c(x, y) - \phi(x) ; \quad\quad \psi^c(x):= \inf_{y \in \Y} c(x, y) - \psi(y).
\end{equation}
Functions which can be written as the $c$-transform of some function are referred to as $c$-\textit{concave}.
Take any $(\phi, \psi)$ admissible for \eqref{intro: c dual problem}, then
\begin{equation}
\label{intro:pointwise-ctransf-inequality}
\phi(x) \leq \psi^c(x) \quad \forall x \in \X
\end{equation}
pointwise as a consequence of \eqref{intro: dual constraint}. Hence, the pair $(\psi^c, \psi)$ always performs at least as well while still respecting the constraint. Applying another $c$-transform yields $(\psi^c, (\psi^c)^c)$ which also performs no worse. However, this process cannot be iterated indefinitely, as $((\psi^c)^c)^c = \psi^c$. Now take any $(\phi, \psi)$ optimal for \eqref{intro: c dual problem}. Taking a $c$ transform cannot increase the value of the objective in \eqref{intro: c dual problem}, thus
\begin{equation}
\label{intro:equality of c trans optims}
    \int_\X \phi(x) \di \rho(x) = \int_\X \psi^c(x) \di \rho(x).
\end{equation}
Combining \eqref{intro:pointwise-ctransf-inequality} and \eqref{intro:equality of c trans optims}, we deduce $\phi = \psi^c$  $\rho$-a.e. and hence without loss of generality, we can always choose representatives of any optimal potentials which are $c$-concave. The dual problem \eqref{intro: c dual problem} is also often reformulated as the following so-called \textit{semi-dual} problem
\begin{equation}
\label{intro: c semidual problem}
    \sup_{\psi \in C_b(\Y)} \int_\X \psi^c(x) \di \rho(x) + \int_\Y \psi(y) \di \mu(y).
\end{equation}

Convex duality (said to be ``Kantorovich Duality" in this case) tells us that the values of the primal and dual problems \eqref{intro: c OT problem} and \eqref{intro: c dual problem} are equal. Furthermore, for any $\gamma \in \Pi(\rho, \mu)$ optimal for \eqref{intro: c OT problem} and any pair $(\phi, \phi^c)$ optimal for \eqref{intro: c dual problem},
\begin{equation}
    \label{intro: compatibility condition}
    \supp \gamma \subset \left\{ (x, y) \in \X \times \Y : \phi(x) + \phi^c(y) = c(x, y) \right\}.
\end{equation}
The set on the right-hand side is referred to as the graph of the $c$-superdifferential of $\phi$, denoted $\partial^c \phi$, with a symmetric definition given for $\psi$. The $c$-superdifferential can equivalently be characterised by $y \in \partial^c\phi(x)$ if and only if
\begin{equation}
\label{eq: c superdiff}
    c(x, y) - \phi(x) \leq c(z, y) - \phi(z) \quad \forall z \in \X.
\end{equation}
Furthermore, if for some admissible $\gamma$ and $(\phi, \psi)$, \eqref{intro: compatibility condition} holds, both are optimal for their respective problems.

A set $W \subset \X \times \Y$ is said to be $c$-cyclically monotone if for every finite sequence of points $(x_1, y_1),... (x_n, y_n) \in W$ and any permutation $\sigma$ of $\{1,..,n\}$,
\begin{equation}
\label{intro: c cm definition statment}
    \sum_{i = 1}^n c(x_i, y_i) \leq \sum_{i = 1} c(x_i, y_{\sigma(i)})
\end{equation}
When the cost is continuous $\gamma \in \Pi(\rho, \mu)$ is optimal for \eqref{intro: c OT problem} if and only if $\supp \gamma$ is $c$-cyclically monotone, see \cite[Corollary 1]{schachermayer2009characterization}.

In the specific case $c(x, y) = \| x - y\|^p$ for some $p>0$ and $\X = \Y = \R^d$, the $p$-th root of the value of the optimal transport problem between $\rho$ and $\mu$ defines the $p$-Wasserstein distance
\begin{equation*}
    W_p(\rho, \mu) := \left( \inf_{\gamma \in \Pi(\rho, \mu)} \int_{\R^d \times \R^d} \|x-y\|^p \di \gamma(x, y)\right)^{1/p}.
\end{equation*}
When $p \geq1$, $W_p$ defines a metric on the space of probability measures with bounded $p$th moment. When $p = 2$, observing that
\begin{equation}
   \int_{\X \times \Y} \|x-y\|^2 \di \gamma =  \int_{\X} \|x\|^2 \di \rho + \int_{\X} \|y\|^2 \di \mu - 2 \int_{\X \times \Y} \langle x | y \rangle \di \gamma.  
\end{equation}
The first two terms are independent of the choice of $\gamma \in \Pi(\rho, \mu)$ and so \eqref{intro: c OT problem} is equivalent to maximising the correlation $\int \langle x | y \rangle \di \gamma$, or using the cost $c(x, y) = - \langle x | y \rangle$. For this problem, Kantorovich duality reads
\begin{equation}
\label{intro: max correlation and its dual}
    \max_{\gamma \in \Pi(\rho, \mu)} \int_{\X \times \Y} \langle x | y \rangle \di \gamma = \min_{\psi \in \C^0(\Y)} \int_\X \psi^* \di \rho + \int_\Y \psi\di \mu,
\end{equation}
where here the $c(x, y)= - \langle x | y \rangle $ transform acts as convex conjugation, defined for $\psi \in \C^0(\Y)$ by
\begin{equation}
\label{eq: legendre transform defn}
    \psi^*(x) = \sup_{y \in \Y} \langle x | y \rangle - \psi(y).
\end{equation}
This is equivalent to extending $\psi$ by $+ \infty$ outside of $\Y$, then taking the usual Legendre transform with supremum over all of $\R^d$. Dual potentials $\varphi$ for the cost $c(x, y) = \frac{1}{2}\|x-y\|^2$ and dual potentials $\phi$ for $c(x, y) = -\langle x, y \rangle$ are thus in bijection, related by
\begin{equation*}
    \phi(x) = \frac{1}{2}\|x\|^2 - \varphi(x).
\end{equation*}
In the case that $\rho$ is absolutely continuous and both measures have bounded second moment, a theorem of Y. Brenier \cite{brenier1991polar} asserts that the optimal transport plan $\gamma \in \Pi(\rho, \mu)$ is unique and induced by the map $T = \nabla \psi^*$, where $\psi$ is a solution to the dual problem on the right hand side of \eqref{intro: max correlation and its dual}. Here, induced means that $\gamma = (\id, T)_\#\rho$. Such an optimal $\psi^*$ is henceforth referred to as a Brenier potential. If we assume further that  $\supp\rho$ is the closure of a connected open set and $\supp \mu$ is bounded, the potentials $\psi, \psi^*$ are respectively $\mu$-a.e./$\rho$-a.e. uniquely defined up to a constant\cite[Proposition 7.18]{santambrogio2015optimal}. We present sharper conditions for uniqueness in Chapter \ref{ch: uniqueness of dual problem}.

\section{Motivation for quantitative stability}

Stability in optimal transport is important for applications, as well as being significant from a purely mathematical viewpoint. \textit{Qualitative} stability in optimal transport is well understood, established as a consequence of compactness arguments. In particular, from \cite[Section 1.6.4]{santambrogio2015optimal}
\begin{theorem}
\label{intro: qualitative stability}
    Let $\X$ and $\Y$ be compact metric spaces, let $\rho_n \in \mathcal{P}(\X)$ and $\mu_n \in \mathcal{P}(\Y)$ be sequences of probability measures converging weakly (in duality with continuous bounded functions) to limits  $ \rho_n \rightharpoonup \rho$ and $\mu_n \rightharpoonup \mu$. Then, up to a subsequence
    \begin{itemize}
        \item Optimal plans $\gamma_n$ between $\rho_n$ and $\mu_n$ converge weakly to an optimal plan between $\rho$ and $\mu$.
        \item If they exist, for $\rho_n \equiv \rho$ fixed, optimal maps $T_n$ from $\rho$ to $\mu_n$ converge strongly in $L^2(\rho)$ to the optimal map $T$ between $\rho$ and $\mu$.
        \item Optimal dual potentials between $\rho_n$ and $\mu_n$ converge uniformly to an optimal pair for $\rho$ and $\mu$.
        \item $W_p(\rho_n, \mu_n) \to W_p(\rho, \mu)$.
    \end{itemize}
\end{theorem}

The question of \textit{quantitative} stability in optimal transport has been the subject of numerous recent works. The case where one measure is absolutely continuous has received a lot of interest recently, for example \cite{gigli2011holder, delalande2023quantitative, letrouit2024gluing, kitagawa2025stability}. Here, the source measure $\rho$ is fixed as an absolutely continuous measure with support $\X \subset \R^d$, and the target measure $\mu$ is perturbed. The authors derive results of the form
\begin{equation}
\label{intro: example map stability statement}
    \|T_{\rho \to \mu_0} - T_{\rho \to \mu_1} \|_{L^2(\rho)} \leq C_{\rho, \Y} W_1(\mu_0, \mu_1)^{q} \text{ for all } \mu_0, \mu_1 \in \mathcal{P}(\Y),
\end{equation}
for some $q \leq 1$. Here $T_{\rho \to \mu_i}$ represents the unique quadratic optimal transport map from $\rho$ to $\mu_i$. The best exponent currently proven here is $q=1/6$ in \cite{delalande2023quantitative}. The following example illustrates that we cannot hope for better than $q = 1/2$ in general.

\begin{example}\textnormal{\cite[Lemma 5.2]{delalandethesis}}
\label{ex: delalande disc example}
    Let $\rho(x) = \frac{1}{\pi}\chi_{B_1}(x)$ as the uniform probability measure on the unit disc in $\R^2$. Let $x_\theta = (\sin \theta, \cos \theta)$, then define $\mu_\theta = \frac{1}{2}\delta_{x_\theta}+ \frac{1}{2}\delta_{x_{\theta+ \pi}}$. In this case, the optimal transport between $\rho$ and $\mu_\theta$ is given by
\begin{equation*}
    T_{\rho \to \mu_\theta}(x) = \begin{cases}
        x_\theta & \text{ if } \langle x | x_\theta \rangle \geq 0\\
        x_{\theta + \pi} & \text{ if } \langle x | x_{\theta} \rangle < 0
    \end{cases}
\end{equation*}
    see Figure \ref{fig:delalande disc example}.
    \begin{figure}[ht]
        \centering
        \includegraphics[width=0.3\linewidth]{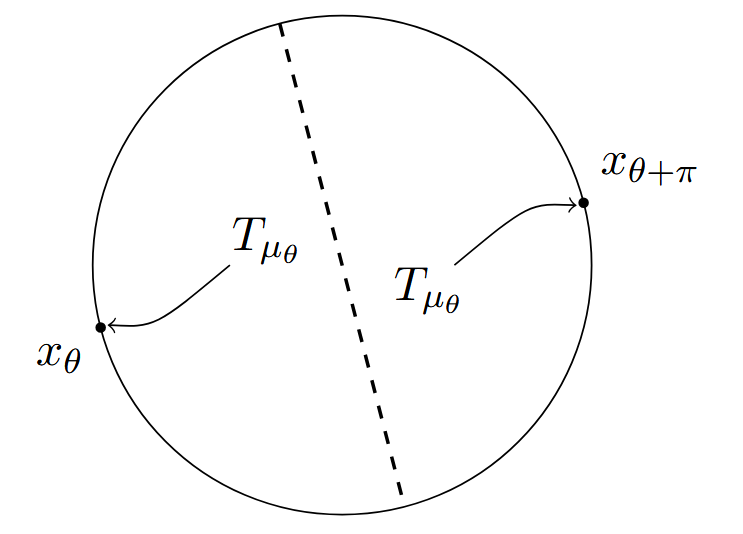}
        \caption{The optimal map $T_{\rho \to \mu_\theta}$. Taken from \cite[Figure 1]{letrouit2025lectures}.}
        \label{fig:delalande disc example}
    \end{figure}
    We can calculate explicitly the quantities involved. For small $\theta$, 
    \begin{equation*}
        W_2(\mu_0, \mu_\theta) = 2\sin(\theta/2) \sim |\theta|
    \end{equation*}
    for any $p \geq 1$. Also,
    \begin{equation*}
        \|T_{\rho \to \mu_{\theta}} - T_{\rho \to \mu_{0}} \|_{L^2(\rho)}^2 = \left(1-\frac{|\theta|}{\pi}\right) |2\sin(\theta/2)|^2 + \frac{|\theta|}{\pi} |2 \sin((\pi-\theta)/2)|^2 \sim \frac{4|\theta|}{\pi}.
    \end{equation*}
    Hence, we cannot hope for better than $q=1/2$ in \eqref{intro: example map stability statement}.
\end{example}

Current proof techniques actually do not establish \eqref{intro: example map stability statement} directly. Instead, they first establish a similar statement for optimisers of the dual problem \eqref{intro: c dual problem}, of the form
\begin{equation}
\label{intro: example potential stability statement}
    \|\phi_{\rho \to \mu_0} - \phi_{\rho \to \mu_1} \|_{L^2(\rho)} \leq C_{\rho, \Y} W_1(\mu_0, \mu_1)^{q} \text{ for all } \mu_0, \mu_1 \in \mathcal{P}(\Y).
\end{equation}
Here $\phi_{\rho \to \mu_i}$ denotes the Brenier potential between $\rho$ and $\mu_i$, satisfying ($\int \phi \di \rho = 0)$. This constraint is imposed to guarantee the uniqueness of potentials (they are unique up to a constant here without this additional constraint). We give a slightly more detailed overview of this theory in Chapter \ref{ch: ac stability}, for an in depth exposition we refer the reader to \cite{letrouit2025lectures}.

\section{Stability of transport plans}

Stability beyond when one measure is absolutely continuous has received comparatively little study. \cite{li2021quantitative} treats stability in both marginals in a neighbourhood of measures whose optimal map between them is Lipschitz, slightly extending the result of \cite{gigli2011holder}. Such regularity assumptions on optimal maps are hard to verify in applications. Moreover, they are often not true - Brenier's theorem does not apply, so in many cases we may not even have the existence of optimal maps, let alone regular ones. It is usually more reasonable to have some regularity assumptions on one of the measures.

\begin{problem}
\label{problem: general plan stability}
Let $\rho$ satisfy some regularity assumptions, for example, of the form
\begin{enumerate}[label = (\roman*)]
    \item \label{intro: enum smooth hypoth} \textit{Continuous}: $\rho$ has density comparable to Lebesgue measure of a bounded compact set.

    \item \label{intro: enum rigid hypoth} \textit{Discrete}: $\rho$ is a fully discrete measure, given by the projection of an absolutely continuous measure onto a uniform grid mesh placed over its support.
    
    \item \label{intro: enum rough hypoth} \textit{Rough}: $\rho$ is comparable to Lebesgue measure ``up to resolution $\varepsilon$" in the sense of \cite[Assumption 1]{jabin2023h}, in that there exist constants $0<\lambda \leq \Lambda$ such that for any ball $B_r \subset \supp \rho$ with $r \geq \varepsilon$, we have
    \begin{equation*}
        \lambda|B_r| \leq \rho(B_r) \leq \Lambda|B_r|.
    \end{equation*}
\end{enumerate}
For what constants $C$ and exponents $q$ are the optimal transport plans stable in the following sense: There exists a ``small" error $\mathcal{E}$ depending on \ref{intro: enum smooth hypoth} the datum of $\rho$ \ref{intro: enum rigid hypoth} the mesh grid size or \ref{intro: enum rough hypoth} the resolution $\varepsilon$, such that
\begin{equation}
\label{intro: eq target plan statement}
    W_p(\gamma_{\rho \to \mu_0}, \gamma_{\rho \to \mu_1}) \leq C W_p(\mu_0, \mu_1)^q + \mathcal{E}?
\end{equation}
Note that on the left-hand side, $W_p$ is taken with respect to the Euclidean distance on the product space $\R^d \times \R^d$, while those on the right-hand side are for the Euclidean distance on $\R^d$. We do not distinguish between these cases in notation, as it is clear from the context. 
\end{problem}

In case \ref{intro: enum rough hypoth}, $\mathcal{E}$ should depend on $\varepsilon$ in a way which vanishes as $\varepsilon \to 0$. Understanding the sharp behaviour in this case should simultaneously explain \ref{intro: enum rigid hypoth}, and this in turn should induce a result for \ref{intro: enum smooth hypoth} as a consequence of qualitative stability arguments. It is also possible that a thorough study of \ref{intro: enum rigid hypoth} could lead to \ref{intro: enum rough hypoth} by qualitative limits. The advantage of studying the discrete formulation is that the Laguerre cells induced on the target domain by choosing a dual potential for $\rho$ have a very nice structure generally here, as will be further discussed in Chapter \ref{ch: discrete transport problem}. One can also perturb both marginals at the same time, with an additional $W_p(\rho_0, \rho_1)$ term on the right-hand side of \eqref{intro: eq target plan statement}. 

We make three immediate observations regarding the stability of optimal plans. The first is that a form of ``reverse stability" of plans always holds.
\begin{proposition}
\label{prop:reverse stability}
    Let $p \in (1, \infty) $. Let $\rho_0$, $\rho_1$, $\mu_0$ and $\mu_1$ be probabilities on $\R^d$ with finite $p$-moments. Let $\gamma_i \in \Pi(\rho_i, \mu_i)$ be optimal transport plans between their marginals for the $p$-cost. Then
    \begin{equation*}
        W_p^p(\gamma_0, \gamma_1) \geq c_p(W_p^p(\rho_0, \rho_1) + W_p^p(\mu_0, \mu_1)).
    \end{equation*}
where $c_p = \min(1, 2^{p/2-1})$.
\end{proposition}
\begin{proof}
    Set $\eta \in \Pi(\gamma_0, \gamma_1) \subset\Pi(\rho_0, \mu_0, \rho_1, \mu_1)$ as an optimal $p$-cost transport on $\R^d \times \R^d$ between $\gamma_0$ and $\gamma_1$. The projections onto first and third or second and fourth coordinates give couplings between $ \rho_0$ and $\rho_1$ or $\mu_0$ and $\mu_1$ respectively. These are candidates for the respective transport problems, so
    \begin{align*}
        W_p^p(\gamma_0, \gamma_1) &= \int_{\R^{2d} \times \R^{2d}} \|(x_0, y_0) - (x_1, y_1)\|^p \di \eta(x_0, y_0, x_1, y_1)\\
        &= \int_{\R^{2d} \times \R^{2d}} \left( \sum_{i=1}^d (x_{0, i} - x_{1,i})^2 + \sum_{i=1}^d (y_{0, i} - y_{1,i})^2 \right)^{p/2} \di \eta(x_0, y_0, x_1, y_1)\\
        &\geq c_p \int_{\R^{2d}} \|x_0 - x_1\|^p \di \pi_{13\#}\eta + c_p\int_{\R^{2d}} \|y_0 - y_1\|^p \di \pi_{24\#}\eta \geq  W_p^p(\rho_0, \rho_1) + W_p^p(\mu_0, \mu_1).
    \end{align*}
where we used the inequality $(a+b)^{p/2} \geq c_p (a^{p/2} + b^{p/2})$ with $c_p = \min(1, 2^{\frac{p}{2}-1})$.
\end{proof}
The second observation is that ``plans are always at least as stable as maps". Fix $\rho$ and two targets $\mu_0, \mu_1$. Assume the existence of maps $T_i$ for the $p$-transport between $\rho$ and $\mu_i$ and set $\gamma_{\rho \to \mu_i} = (\id, T_{\rho \to \mu_i})_\# \rho \in \Pi(\rho, \mu_i)$ as the optimal plans induced by them. Then
\begin{equation}
\label{eq: plans better than maps}
    W_p(\gamma_{\rho \to \mu_0}, \gamma_{\rho \to \mu_1}) \leq \| T_0 - T_1 \|_{L^p(\rho)}.
\end{equation}
To see this, it suffices to observe that $(\id, T_{\rho \to \mu_0}, \id, T_{\rho \to \mu_1})_{\#}\rho$ is an admissible coupling between $\gamma_{\rho \to \mu_0}$ and $ \gamma_{\rho \to \mu_1}$, and its transport cost is precisely $\| T_0 - T_1 \|_{L^p(\rho)}^p$. Consequently, all the results of \cite{delalande2023quantitative, letrouit2024gluing} pass directly to plans, so that for the continuous case \ref{intro: enum smooth hypoth}, $\mathcal{E} = 0$ with $q=1/6$, under suitable assumptions on $\rho$ and the target support.

The third observation is related to the following natural question: Is the worst-case behaviour of plans any better than that of maps? Consider again Example \ref{ex: delalande disc example}, except now we are interested in the Wasserstein distance between the optimal plans $\gamma_{\rho \to \mu_{\theta}} = (\id, T_{\rho \to \mu_\theta})_{\#}\rho$. Consider the following coupling $\eta \in \Pi(\gamma_{\rho \to \mu_0}, \gamma_{\rho \to \mu_1})$, defined by 
\begin{equation*}
    \eta = (\id, T_{\rho \to \mu_0}, R_\theta, R_\theta T_{\rho \to \mu_0}) \in \Pi(\gamma_{\rho\to \mu_0}, \gamma_{\rho\to \mu_\theta}),
\end{equation*}
where $R_\theta$ is the matrix representing rotation anticlockwise by $\theta$ in $\R^2$. This is a valid coupling since $(R_{\theta}, R_\theta)_\#\gamma_{\rho\to \mu_0} = \gamma_{\rho\to \mu_\theta}$. The coupling $\eta$ gives an upper bound on the $W_2$ distance between the two plans. In particular,
\begin{align*}
    W_2^2(\gamma_{\rho \to \mu_{0}}, \gamma_{\rho \to \mu_{\theta}}) \leq& \int_{\R^{2} \times \R^{2}} \|x_0-x_1\|^2 + \|y_0 - y_1\|^2 \di \eta(x_0, y_0, x_1, y_1)\\
    =&
    \int_{\R^{2}} \| x- R_\theta x\|^2 + \| T_{\rho \to \mu_0}(x) - R_\theta T_{\rho \to\mu_0}(x)\|^2 \di \rho(x)\\
    =& \frac{1}{\pi}\int_0^1 \int_{0}^{2\pi} \Bigl( 2r\sin(\theta/2)\Bigr)^2 r \di \phi \di r + \Bigl(2\sin(\theta/2)\Bigr)^2\\
    =& \left(\frac{1}{2} + 1\right)\Bigl( 2\sin(\theta/2)\Bigr)^2 = \frac{3}{2}W_2^2(\mu_0,\mu_\theta)
\end{align*}
since $\| T_{\rho \to \mu_0}(x) - R_\theta T_{\rho \to\mu_0}(x)\| = 2\sin(\theta/2)$ $\rho$-a.e.
Hence, in this case, the plans exhibit Lipschitz behaviour, better than the Hölder $1/2$ of the maps. The argument below demonstrates that plans cannot be Lipschitz in general, but leaves the door open to exponents in $(1/2, 1)$.

\begin{lemma}
\label{lem: plan converges to LOT distance}
For fixed $\rho$ and two targets $\mu_0, \mu_1$, all compactly supported. Set $\gamma_i \in \Pi(\rho, \mu_i)$ as the optimal plans induced by optimal maps $T_i$. Let 
\begin{equation*}
    c_\varepsilon(x_0, y_0, x_1, y_1) = \|x_0 - x_1\|^2 + \varepsilon \|y_0 - y_1\|^2,
\end{equation*}
and denote
\begin{equation}
\label{eq: epsilon cost between plans}
    \tau_{c_\varepsilon}(\gamma_0, \gamma_1) = \inf_{\eta \in \Pi(\gamma_0, \gamma_1)} \int_{\R^4} c_\varepsilon(x_0, y_0, x_1, y_1) \di \eta(x_0, y_0, x_1, y_1)
\end{equation}
as the transport cost between plans with respect to $c_{\varepsilon}$. Then
\begin{equation*}
    \frac{1}{\varepsilon} \tau_{c_\varepsilon}(\gamma_0, \gamma_1) \to \| T_0 - T_1 \|^2_{L^2(\rho)},
\end{equation*}
and optimal plans $\eta_\varepsilon$ for \eqref{eq: epsilon cost between plans} converge weakly to $(\id, T_0, \id, T_1)_\#\rho$.
\end{lemma}
\begin{proof}
    This is simply a consequence of the stability of plans with respect to uniform convergence of costs. We present here a self-contained proof of this fact. Consider optimisers $\eta_\varepsilon \in \Pi(\gamma_0, \gamma_1)$ for the $c_\varepsilon$ problem. By compactness, we can assume these converge weakly up to a subsequence to some $\eta_0 \in \Pi(\gamma_0, \gamma_1)$. Then strong-weak convergence gives
    \begin{equation*}
        \int_{\R^{2d}} c_\varepsilon \di \eta_\varepsilon \to \int_{\R^{2d}} c_0 \di \eta_0.
    \end{equation*}
    But we also have
    \begin{equation*}
        |\tau_{c_0}(\gamma_0, \gamma_1) - \tau_{c_{\varepsilon}}(\gamma_0, \gamma_1) | \leq \|c_0 - c_\varepsilon\|_\infty,
    \end{equation*}
    and hence $\int c_0 \di \eta_0 = \tau_{c_0}(\gamma_0, \gamma_1)$. In other words, the limit $\eta_0$ is optimal. But this optimiser must precisely be
    \begin{equation*}
        \eta_0 = (\id, T_0, \id, T_1)_{\#}\rho
    \end{equation*}
    since it has zero cost for $c_0 = \| x_0 - x_1\|^2$, and this is the only way to glue through the identity (since we have existence of maps). Hence
    \begin{equation*}
        \int_{\R^{4d}} \|y_0 - y_1\|^2 \di \eta_\varepsilon(x_0, y_0, x_1, y_1) \leq \frac{1}{\varepsilon} \tau_{c_\varepsilon}(\gamma_0, \gamma_1) \leq \| T_0 - T_1 \|^2_{L^2(\rho)},
    \end{equation*}
    where the first inequality is by monotonicity of the cost $\| y_0 - y_1\|^2 \leq \frac{1}{\varepsilon}c_\varepsilon(x_0, y_0, x_1, y_1) $ and the fact that $\eta_\varepsilon$ is optimal for $c_\varepsilon$, and the second inequality is a consequence of the plan $(\id, T_0, \id, T_1)_\#\rho$ as a competitor for the $c_\varepsilon$ problem. Passing to the limit $\varepsilon \to 0$, the left hand side converges to $\| T_0 - T_1 \|_{L^2(\rho)}^2$ by weak convergence, and hence so does $\frac{1}{\varepsilon} \tau_{c_\varepsilon}(\gamma_0, \gamma_1)$ by squeezing.
\end{proof}

\begin{proposition}[Optimal plans cannot be Lipschitz]
    Fix $\rho$ a probability density uniform on a compact set with non-empty interior in $\R^d$. Then, given a compact set $\Y$ in $\R^d$ with non empty interior, there \textbf{does not exist} $C_{\rho, \Y} >0$ such that for all $\mu_0, \mu_1$ probabilities on $\Y$,
    \begin{equation}
    \label{eq: assumption plans are lip}
        W_2(\gamma_{\rho \to \mu_0}, \gamma_{\rho \to \mu_1}) \leq C W_2(\mu_0, \mu_1).
    \end{equation}
    where here $\gamma_{\rho \to \mu_i} = (\id, T_{\rho \to \mu_i})_\#\rho$ are the plans induced by the optimal maps from $\rho \to \mu_i$ for quadratic cost.
\end{proposition}
\begin{proof}
    Since $\Y$ has non-empty interior, it contains a closed ball. By invariance of optimisers to the quadratic OT problem under dilations and translations, we may assume $\Y = \overline{B}(0, 1) =: B_1$. Assume that \eqref{eq: assumption plans are lip} holds on $B_1$ with constant $C> 0$. Given $\mu_0, \mu_1 \in \mathcal{P}(B_1)$, we define the contractions
    \begin{equation*}
        \mu_i^\varepsilon := p_{\varepsilon\#} \mu_i \; \text{ where } p_\varepsilon(y) = \varepsilon y.
    \end{equation*}
    We have $\mu_i^\varepsilon \in \mathcal{P}(B_1)$ and hence
    \begin{equation}
    \label{eq: lip for the rho eps plans}
        W_2(\gamma_{\rho \to \mu_0^\varepsilon}, \gamma_{\rho \to \mu_1^\varepsilon}) \leq C W_2(\mu_0^\varepsilon, \mu_1^\varepsilon).
    \end{equation}
    A change of variables shows that
    \begin{equation*}
        W_2(\mu_0^\varepsilon, \mu_1^\varepsilon) = \varepsilon W_2(\mu_0, \mu_1)
    \end{equation*}
    since $\pi \in \Pi(\mu_0, \mu_1)$ if and only if $(\varepsilon y_0, \varepsilon y_1)_\#\pi \in \Pi(\mu_0^\varepsilon, \mu_1^\varepsilon)$.
    Furthermore,
    \begin{equation*}
        \eta \in \Pi(\gamma_{\rho \to \mu_0}, \gamma_{\rho \to \mu_1}) \iff (x_0, \varepsilon y_0, x_1, \varepsilon y_1)_\#\eta \in \Pi(\gamma_{\rho \to \mu_0^\varepsilon}, \gamma_{\rho \to \mu_1^\varepsilon})
    \end{equation*}
    and so
    \begin{align*}
         W_2^2(\gamma_{\rho \to \mu_0^\varepsilon}, \gamma_{\rho \to \mu_1^\varepsilon}) =& \inf_{\eta \in \Pi(\gamma_{\rho \to \mu_0}, \gamma_{\rho \to \mu_1})} \int \|x_0 - x_1\|^2 + \|\varepsilon y_0 - \varepsilon y_1 \|^2 \di \eta(x_0, y_0, x_1, y_1)\\
         =& \tau_{c_{\varepsilon^2}}(\gamma_{\rho \to \mu_0}, \gamma_{\rho \to \mu_1}).
    \end{align*}
    Consequently, \eqref{eq: lip for the rho eps plans} becomes
    \begin{equation*}
        \tau_{c_{\varepsilon^2}}(\gamma_{\rho \to \mu_0}, \gamma_{\rho \to \mu_1}) \leq C^2 \varepsilon^2 W_2^2(\mu_0, \mu_1).
    \end{equation*}
    Dividing by $\varepsilon^2$ and taking $\varepsilon \to 0$ using Lemma \ref{lem: plan converges to LOT distance} followed by a square root implies
    \begin{equation*}
        \|T_{\rho \to \mu_0} - T_{\rho \to \mu_1}\|_{L^2(\rho)} \leq C W_2(\mu_0, \mu_1) \text{ for all } \mu_0, \mu_1 \in \mathcal{P}(B_1).
    \end{equation*}
    But this is known to be impossible: Example \ref{ex: delalande disc example} contradicts this. Thus, there can be no $C_{\rho, \Y} >0$ such that \eqref{eq: assumption plans are lip} holds.
\end{proof}
\begin{remark}
     If one could quantify the convergence in Lemma \ref{lem: plan converges to LOT distance}, then the cost $c_\varepsilon$ could potentially be used to invert inequalities of the form \eqref{eq: plans better than maps} and, given a stability result about plans, infer back to stability about the maps.
\end{remark}

The principal aim of this report is to begin to treat quantitative stability beyond the case where one measure is fixed absolutely continuous. This is important for applications, as one might often work with a source that is a discretisation of an absolutely continuous measure, using a very fine mesh. To the best of our knowledge, there are no known techniques to treat stability with this type of source measure. Our approach is, in leaving behind absolutely continuous measures, to go to the other extreme. We study the quantitative stability properties of the fully discrete transport problem, aiming to recover statements for general target measures as the result of a qualitative limit. In the fully discrete case, the optimal transport problem becomes a linear program, which gives us another lens through which to understand the stability.
\chapter{Stability with one measure fixed absolutely continuous}
\label{ch: ac stability}

In the first half of this chapter, we review some of the literature on the quantitative stability of the quadratic optimal transport problem with one measure fixed and absolutely continuous. The first result of this form without regularity hypotheses on the maps was attained in \cite{berman2021convergence}. In \cite{delalande2023quantitative}, techniques were developed to attain stability exponents in \eqref{intro: example map stability statement} that are independent of the dimension $d$. These have been further used to extend the same results to a larger class of source measures in \cite{letrouit2024gluing}, and to Riemannian manifolds in \cite{kitagawa2025stability}. In general, these techniques have been designed to treat the quadratic cost, with some extensions to $p$-costs in \cite{mischler2024quantitative}. For a more complete account of this theory, see \cite{letrouit2025lectures}.

In the second half of the chapter, we provide an extension of a certain strong convexity inequality to a larger class of measures, using techniques developed in \cite{letrouit2024gluing}. This extends the results on the quantitative stability of Wasserstein barycentres in \cite{carlier2024barycentres} to a much larger class of absolutely continuous measures.

\section{The Kantorovich functional}

Fix $\rho$ to be an absolutely continuous probability measure with support $\X$, which is the closure of a connected open bounded set in $\R^d$. Fix also $\Y \subset \R^d$ compact. To study the stability of Brenier potentials, we consider the semi-dual problem
\begin{equation}
\label{eq: semidual quadratic}
    \min_{\psi \in C_b(\Y)} \int_\X \psi^*(x) \di \rho(x) + \int \psi(y) \di \mu(y),
\end{equation}
where $\psi^*$ represents the Legendre transform defined in \eqref{eq: legendre transform defn}.
This problem is a convex optimisation problem over $C_b(\Y)$, and hence (informally) stability of minimisers is equivalent to strong convexity of the objective \eqref{eq: semidual quadratic}. The second term is linear, so we can focus on the first term, which motivates the following definition. The Kantorovich functional is defined as
\begin{equation*}
    \K_\rho : C_b(\Y) \to \R \; \quad \psi \mapsto \int_\X \psi^*(x) \di \rho(x),
\end{equation*}
where the Legendre transform is taken as in \eqref{eq: legendre transform defn}. First introduced in \cite{kitagawa2019convergence} to study the convergence of algorithms for semi-discrete optimal transport, the Kantorovich functional is the central tool for the stability results presented in this chapter. The following proposition is a minor generalisation of \cite[Lemma 1.8]{delalandethesis}, providing a characterisation of the subdifferential of $\K_\rho$.
\begin{proposition}[Characterisation of the Subdifferential]
\label{prop: subdiff characterisation}
$\K_\rho$ is a convex, 1-Lipschitz functional on $C_b(\Y)$ equipped with $\| \cdot \|_\infty$ norm. The subdifferential of $\K_\rho$ is given by
\begin{equation}
\label{eq: K rho subdiff char by projection}
    \partial \K_\rho(\psi) = \left\{ -  p_{Y\#}\gamma \;|\;  \gamma \in \mathcal{P}(\X \times \Y) \text{ with } p_{X\#}\gamma = \rho \text{ and }\supp \gamma \subset \partial \psi \right\}.
\end{equation}
where $p_Y: \X \times \Y \to \Y$ denotes the canonical projection onto $\Y$ and
$\partial \psi\subset \X \times \Y$ denotes the graph on the subdifferential of $\psi^*$.
\end{proposition}
\begin{proof}
    The subdifferential is a priori valued in the topological dual $C_b(\Y)^*$, the space of finite signed regular Borel measures. Since the Legendre transform is a 1-Lipschitz operator on $C_b(\Y)$, and it is pointwise convex, $\K_\rho$ is a 1-Lipschitz convex functional on $C_b(\Y)$. Hence, by subdifferential rules for convex functions (see \cite[Corollary 2.1]{attouch1986duality}),
\begin{equation}
\label{eq: equivalence in subdiff of k rho in argmin}
    -\mu \in \partial \K_\rho(\psi) \iff \psi \in \argmin \int_\X \psi^* \di \rho + \int_\Y \psi \di \mu.
\end{equation}
We now show that any such $\mu$ should be a positive probability measure. Assume first $\langle 1 | \mu \rangle = m \neq 1.$ Then for any $\psi \in C_b(\Y)$, replacing with $\psi +  \lambda$ for some $\lambda \in \R$ gives
\begin{equation*}
    \int_\X (\psi+ \lambda)^* \di \rho + \int_\Y \psi + \lambda \di \mu = \int_\X \psi^*  \di \rho + \int_\Y \psi \di \mu + \lambda(m -1)
\end{equation*}
and so taking $\lambda \to \pm \infty$ depending on the sign of $m-1$ gives that no argmin can exist. Thus, any element of $\partial \mathcal{K}_\rho$ should have total mass $-1$. Similarly, assume there exists some measurable $A \subset \Y$ for which $\langle \chi_A | \mu \rangle = m < 0$. As mentioned in the introduction, by a density argument (see \cite[Theorem 2.24]{rudin1987real}) \eqref{eq: semidual quadratic} has the same value as if we instead took the infimum over all $\psi \in L^1(\mu)$. Then, if $\psi \in C_b(\Y)$ is optimal, it is also optimal over all $L^1(\mu)$. Considering $\psi + \lambda \chi_A$ for $\lambda>0$, by monotonicity of the Legendre transform
\begin{equation*}
    \int_\X (\psi+ \lambda \chi_A)^* \di \rho + \int_\Y \psi + \lambda \chi_A \di \mu \leq \int_\X \psi^*  \di \rho + \int_\Y \psi \di \mu + \lambda m
\end{equation*}
from which again $\psi$ cannot belong to the argmin. Hence, any subdifferential element is a \textit{negative}  measure of mass $-1$.

Now knowing that the only measures in $\partial\K_\rho$ are negative probability measures, by \eqref{eq: equivalence in subdiff of k rho in argmin}, combined with the compatibility condition \eqref{intro: compatibility condition}, we arrive at \eqref{eq: K rho subdiff char by projection}.
\end{proof}

\begin{remark}
    In general, $\K_\rho$ is never strongly convex, as it is affine along the path $\psi + t$ for $t \in \R$. Paths corresponding to the same dual potential up to a constant. To avoid this, we fix the constants of the potentials so that the Brenier potential is mean free ( $\int \psi^* \di \rho = 0$). Since  $\supp \rho$ is the closure of a connected open set, \cite[Proposition 7.18]{santambrogio2015optimal} tells us that mean free potentials are unique in this setting.
\end{remark}
We now turn to strong convexity estimates for $\K_\rho$, which in turn will imply quantitative stability of Brenier potentials.
\begin{theorem}\textnormal{\cite[Theorem 2.1]{delalande2023quantitative}\cite[Theorem 2.1]{letrouit2024gluing}}
\label{thrm: strong convexity weak version}
    Let $Q \subset \R^d$ be a compact convex set with non-empty interior, and let $\rho$ be a probability density over $Q$ satisfying $M_\rho \geq \rho \geq m_\rho >0$. Then for all $\psi_0, \psi_1 \in C_b(\Y)$ we have
    \begin{equation*}
        \var_\rho(\psi_1^* - \psi_0^*) \leq C_\psi \langle \psi^*_1 - \psi_0^* | (\nabla \psi_0^*)_\#\rho - (\nabla \psi_1^*)_\#\rho \rangle
    \end{equation*}
    where $C_\psi = e \frac{M_\rho}{m_\rho} (M_\phi - m_\phi)$ for $m_\phi \leq \psi_i^* \leq M_\phi$.
\end{theorem}

Here, we make explicit the dependence of the above constant on $\psi$, in terms of the oscillation of the Brenier potentials $(M_\phi - m_\phi)$. When both $\X$ and $\Y$ are bounded, we can uniformly bound this, and also when $\X$ is bounded and all target measures satisfy a uniform $p$th moment bound for $p>d$, as demonstrated in \cite{berman2021convergence, delalande2023quantitative}. The above strong convexity directly implies the following quantitative stability estimates for both maps and potentials.
\begin{corollary} \textnormal{\cite{delalande2023quantitative}}
\label{corr: Potential stability source convex}
    Let $\X \subset \R^d$ be a compact convex set with non-empty interior, and let $\rho$ be a probability density over $\X$ satisfying $M_\rho \geq \rho \geq m_\rho >0$. Then for any $\Y \subset \R^d$ compact, there exists $C_{\rho, \Y} > 0$ such that
    \begin{equation}
        \| \psi^*_1 - \psi^*_0\|_{L^2(\rho)} \leq C W_1(\mu_0, \mu_1)^{1/2} \quad \forall \mu_0, \mu_1 \in \mathcal{P}(\Y),
    \end{equation}
    where $\psi_i$ are optimal for \eqref{eq: semidual quadratic} between $\rho$ and $\mu_i$.
\end{corollary}
\begin{proof}
    We observe that in the preceding theorem, the constant $C$ depended on $\psi$ only by the oscillation $M_\phi - m_\phi$. By definition of the Legendre transform restricted to $\Y$,
    \begin{equation*}
        \psi^*(x) = \sup_{y \in \Y} \langle x | y \rangle - \psi(y)
    \end{equation*}
    is a supremum of $R_\Y$ Lipschitz functions $x \mapsto \langle x | y \rangle - \psi(y)$ where $R_\Y$ is the radius of the smallest ball containing $\Y$. Hence, $\psi^*$ is $R_\Y$ Lipschitz, and its oscillation over $\X$ is bounded by
    \begin{equation*}
        M_\phi - m_\phi \leq (\diam \X) R_\Y.
    \end{equation*}
    Furthermore, due to the invariance of the quadratic optimal transport problem to translations, translating $\Y$ allows us to replace $R_\Y$ with $\diam \Y$. Thus $C_\psi$ in \eqref{eq: stronger inequality convex support} is uniformly bounded in $\psi$, depending only on $\rho$ and $\Y$.

    Secondly, by Brenier's theorem, $(\nabla \psi_i^*)_{\#}\rho = \mu_i$ so that applying Theorem \ref{thrm: strong convexity weak version} gives
    \begin{equation*}
        \| \psi^*_1 - \psi_0^*\|_{L^2(\rho)}^2 = \var_\rho(\psi^*_1 - \psi_0^*) \leq C \langle \psi_0 - \psi_1 | \mu_1 - \mu_0 \rangle,
    \end{equation*}
    where the first equality is due to our choice of mean free potentials. Finally, since $\psi_1^* - \psi_0^*$ is $2 \diam \Y$ Lipschitz continuous, Kantorovich Rubinstein duality gives
    \begin{equation*}
        \| \psi^*_1 - \psi_0^*\|_{L^2(\rho)} \leq C W_1(\mu_0, \mu_1)^{1/2}
    \end{equation*}
    as required.
\end{proof}
\begin{corollary} \textnormal{\cite{delalande2023quantitative}}
\label{corr: map stability source convex}
    Let $\X \subset \R^d$ be a compact convex set with non-empty interior and $\mathcal{H}^{d-1}$ rectifiable boundary, and let $\rho$ be a probability density over $\X$ satisfying $M_\rho \geq \rho \geq m_\rho >0$. Then for any $\Y \subset \R^d$ compact, there exists $C_{\rho, \Y} > 0$ such that
    \begin{equation}
        \| T_1 - T_0\|_{L^2(\rho)} \leq C W_1(\mu_0, \mu_1)^{\frac{1}{6}} \quad \forall \mu_0, \mu_1 \in \mathcal{P}(\Y),
    \end{equation}  
    were $T_i$  is the unique optimal transport map between $\rho$ and $\mu_i$.
\end{corollary}
\begin{proof}
    One combines \cite[Proposition 4.1]{delalande2023quantitative} with Corollary \ref{corr: Potential stability source convex}, using again that the Lipschitz constants of $\psi_0^*$ and $\psi_1^*$ are uniformly bounded.
\end{proof}

\section{Glueing strong convexity inequalities}

The proof of Theorem \ref{thrm: strong convexity weak version} allows for a slightly stronger inequality presented below, which has found other usage, for example, in studying the stability of Wasserstein Barycentres in \cite{carlier2024barycentres} and for the optimal matching problem in \cite{huesmann2024asymptotics}. In particular, we have
\begin{theorem}\textnormal{\cite[Appendix B]{carlier2024barycentres}, \cite{delalande2023quantitative}}
\label{theorem: strong convexity convex support}
    Under the same hypotheses as Theorem \ref{thrm: strong convexity weak version}, for any $\Y \subset \R^d$ compact there exists $C_{\rho, \Y} >0$ such that for all $\psi_0, \psi_1 \in C_b(\Y)$,
\begin{equation}
    \label{eq: stronger inequality convex support}
    \var_\rho(\psi_1^* - \psi_0^*) \leq C \bigg(\K_\rho(\psi_1) -\K_\rho(\psi_0) + \langle \psi_1 - \psi_0 | (\nabla\psi_0^*)_\# \rho \rangle \bigg).
\end{equation}

\end{theorem}
In \cite{letrouit2024gluing, kitagawa2025stability}, the authors treat stability for a much larger class of absolutely continuous source measures $\rho$. One decomposes the source domain into cubes, upon each of which we can apply Theorem \ref{theorem: strong convexity convex support}. Then, under regularity assumptions on $\rho$, it is shown that the sum of stability errors locally over each cube can be combined to give a stability result globally across the domain. These techniques were applied to deduce the stability of maps and potentials for more general source measures. However, the techniques here did not directly establish \eqref{eq: stronger inequality convex support} for the same larger class of $\rho$. Using the same glueing techniques, we obtain a strong convexity inequality \eqref{eq: stronger inequality convex support} for John domains, which, to our knowledge, is new in this generality. We first review the glueing theory used in \cite{letrouit2024gluing}.
\begin{definition}[Boman Chain condition]\cite{boman1982lp}
A probability measure $\rho$ on an open set $\X \subset \R^d$ satisfies the Boman chain condition with parameters $A, B, C > 1$ if there exists a covering $\F$ of $\X$ of open cubes $Q \in \F$ such that
\begin{itemize}
    \item for $\sigma = \min(10/9, (d+1)/d)$ and for any $x \in \R^d$,
    \begin{equation}
        \label{eq: uniform upper bound on bowman cube overlap}
        \sum_{Q \in \F} \chi_{\sigma Q}(x) \leq A \chi_{\X}(x).
    \end{equation}
    \item For some fixed cube $Q_0$ in $\F$, called the central cube, and for every $Q \in \F$, there exists a chain $Q_0, Q_1,..., Q_N= Q$ of distinct cubes from $\F$ such that for any $j \in \{0,...,N-1\}$,
    \begin{equation}
        Q \subset B Q_j
    \end{equation}
    Where $BQ_j$ is the cube with the same centre as $Q_j$ and side length multiplied by $B$.
    \item Consecutive cubes of the above chain rho-verlap quantitatively: for any $j \in \{0,...,N-1\}$,
    \begin{equation}
        \rho(Q_j \cap Q_{j+1}) \geq C^{-1} \max(\rho(Q_j), \rho(Q_{j+1})).
    \end{equation}
\end{itemize} 
\end{definition}
Note that the length $N$ of the chain can depend on $Q \subset \F$ and need not be uniformly bounded over $Q \in \F$. The following lemma is the crux of the argument, providing a way to combine local stability on each cube of a Boman decomposition into a global stability result.
\begin{lemma}\textnormal{\cite[Lemma 3.3]{letrouit2024gluing}}
\label{lemma-gluingvars-letrouit}
Let $\rho$ be a probability density over a domain $\X$ satisfying the Boman chain condition for some covering $\mathcal{F}$ and some $A, B, C > 1$. Assume moreover that there exists $D >0$ such that
\begin{equation}
\label{scale regularity assumption}
    \forall Q \in \F, \quad \rho(5B\sqrt{d}) \leq D \rho(Q).
\end{equation}
Then settting $\Tilde\rho_Q = \frac{1}{\rho(Q)}\rho|_Q$, for any continuous function $f$ on $\X$, it holds
\begin{equation*}
    \var_\rho(f) \leq 200 A^2 C D^3 \sum_{Q \in \F} \rho(Q) \var_{\Tilde{\rho}_Q}(f).
\end{equation*}
\end{lemma}
\begin{theorem}[Strong convexity for more general source]
\label{thrm: general strong convexity estimate}
Let $\rho$ be a probability density over a domain $\X \subset \R^d$ satisfying the Boman chain condition for some covering $\F$ and $A, B, C > 1$. Assume moreover that \eqref{scale regularity assumption} holds for some $D > 0$ (w.l.o.g. we assume $D > 1$) and there exists $E>0$ such that
\begin{equation}
\label{eq: uniform comparability of densities over cubes}
    \sup_{Q \in \F} \frac{M_{\rho_Q}}{m_{\rho_{Q}}} \leq E < +\infty.
\end{equation}
Then for any $\Y \subset \R^d$ compact, there exits $C_{\rho, \Y}>0$ such that for all $\psi_0, \psi_1 \in C_b(\Y)$,
\begin{equation}
\label{eq: strong inequality}
    \var_\rho(\psi_1^* - \psi_0^*) \leq C \bigg(\K_\rho(\psi_1) - \K_\rho(\psi_0) + \langle \psi_1 - \psi_0 | (\nabla\psi_0^*)_\# \rho\ \rangle \bigg).
\end{equation}
\end{theorem}
\begin{proof}
For each $Q \in \F$ set $\Tilde{\rho}_Q = \frac{1}{\rho(Q)} \rho_{|Q}$, then Theorem \ref{eq: stronger inequality convex support} tells us
\begin{equation}
\label{eq local strong variance on cubes}
    \var_{\Tilde{\rho}_Q}(\psi_1^* - \psi_0^*) \leq C \bigg(\K_{\Tilde{\rho}_Q}(\psi_1) - \K_{\Tilde{\rho}_Q}(\psi_0) + \langle \psi_1 - \psi_0 | (\nabla\psi_0^*)_\#\Tilde{\rho}_Q \rangle \bigg).
\end{equation}
where $C_1 = e(d+1) 2^{d+1} E^2 \diam(\Y) \diam(\X)$.
Combining Lemma \ref{lemma-gluingvars-letrouit} with \eqref{eq local strong variance on cubes}, we get
\begin{align*}
    \var_\rho(\psi_1^* - \psi_0^*) \leq& 200 A^2 C D^3 \sum_{Q \in \F} \rho(Q) \var_{\Tilde{\rho}_Q}(\psi_1^* - \psi_0^*)\\
    \leq& 200 A^2 C D^3 C_1 \sum_{Q \in \F} \bigg(\K_{\rho_{|Q}}(\psi_1) - \K_{\rho_{|Q}}(\psi_0) + \langle \psi_1 - \psi_0 | (\nabla\psi_0^*)_\#\rho_{|Q} \rangle\bigg).
\end{align*}
Now define a partition $\F'$ of $\X$ such that $x, x'$ belong to the same element $P \in \F'$ if and only if they belong to the same elements of $\F$. Proposition \ref{prop: subdiff characterisation} applied to $\K_{\rho_{|P}}$ tells us that
\begin{equation}
\K_{\rho_{|P}}(\psi_1) - \K_{\rho_{|P}}(\psi_0) + \langle \psi_1 - \psi_0 | (\nabla\psi_0^*)_\# \rho_{|P} \rangle \geq 0 \quad \forall P \in \F'. 
\end{equation}
It follows that
\begin{align*}
    \var_\rho(\psi_1^* - \psi_0^*) \leq& 200 A^3 C D^3 C_1 \sum_{P \in \F'} (\K_{\rho_{|P}}(\psi_1) - \K_{\rho_{|P}}(\psi_0) + \langle \psi_1 - \psi_0 | (\nabla\psi_0^*)_\# \rho_{|P})\\
    =& 200 A^3 C D^3 C_1 \bigg(\K_\rho(\psi_1) - \K_\rho(\psi_0) + \langle \psi_1 - \psi_0 | (\nabla\psi_0^*)_\# \rho\ \rangle \bigg).
\end{align*}
Where the first inequality is due to the chain condition of no more than $A$ cubes simultaneously overlapping by \eqref{eq: uniform upper bound on bowman cube overlap} and the second equality is due to $P$ being a partition of $\X$.
\end{proof}

In particular, since the uniform measure on a John domain satisfies the Bowman chain condition (see \cite[Lemma 2.1]{boman1982lp}), we have the following.
\begin{theorem}[Strong convexity of the Kantorovich function on John domains]
    \label{thrm: John domain stronger inequality}

    Let $\X\subset \R^d$ be a John domain with non-empty interior. Let $\rho$ be a probability density on $\X$ bounded above and below by positive constants. Then for every compact set $\Y$ there exists a constant $C_{\rho, \Y} >0$ such that for all $\psi_0, \psi_1 \in C_b(\Y)$,
    \begin{equation*}
        \var_\rho(\psi_1^* - \psi_0^*) \leq C \bigg(\K_\rho(\psi_1) -\K_\rho(\psi_0) + \langle \psi_1 - \psi_0 | (\nabla\psi_0^*)_\# \rho \rangle \bigg).
    \end{equation*}
\end{theorem}
\chapter{The discrete transport problem}
\label{ch: discrete transport problem}
We now shift to the main focus of this report, the fully discrete optimal transport problem. Such problems have a natural interpretation as a linear programming problem, which we will exploit heavily.

\section{Primal and dual polyhedra}

In the special case of fully discrete optimal transport, both dual and primal problems are linear programming problems. Specifically, fix weights $\alpha \in \R^M$, $\beta \in \R^N$ strictly positive probability vectors. Given $X = (x_1, ..., x_M) \in (\R^d)^M$ and $Y = (y_1,...,y_N) \in (\R^d)^N$, we set
\begin{equation}
\label{eq: rho x and mu y definition}
   \rho_X = \sum_{i =1}^M \alpha_i \delta_{x_i} \quad \text{ and } \quad \mu_Y = \sum_{j=1}^N \beta_j\delta_{y_j}.
\end{equation}
We define a cost matrix $C(X, Y) \in \R^{M \times N}$ by $C(X, Y)_{ij} = c(x_i, y_j)$, and also the convex polyhedron of couplings
\begin{equation*}
    \Pi(\alpha, \beta) := \left\{ \hat\gamma \in \R^{M \times N} : \hat\gamma_{ij} \geq 0, \sum_{i = 1}^M \hat\gamma_{ij} = \beta_j \text{ and } \sum_{j = 1}^N \hat\gamma_{ij} = \alpha_i \text{ for } i=1,...,M; j=1,...,N \right\}.
\end{equation*}
Note here that we use $\hat \gamma$ to refer to couplings living in the polyhedron $\Pi(\alpha, \beta) \subset \R^{M \times N}$, while we will reserve $\gamma$ for plans as probability measures on $\R^d \times \R^d$. The primal transport problem \eqref{intro: c OT problem} between $\rho_X$ and $\mu_Y$ is equivalent to
\begin{equation}
\label{eq: C(X, Y) linear programming problem}
    \inf_{\hat\gamma \in \Pi(\alpha, \beta)} \langle C(X, Y) | \hat\gamma \rangle.
\end{equation}
Given $C \in \R^{M \times N}$, we define the dual polyhedron
\begin{equation*}
    D(C) : = \left\{ (\phi, \psi) \in \R^{M + N} : \phi_i + \psi_j \leq C_{ij} \text{ for all } i=1,...M, j = 1,...N \right\}.
\end{equation*}
Then the dual transport problem \eqref{intro: c dual problem} between $\rho_X$ and $\mu_Y$ is equivalent to
\begin{equation}
\label{eq: C(X, Y) dual linear programming problem}
    \sup_{\phi, \psi \in D(C(X, Y))} \langle \phi | \alpha \rangle + \langle \psi | \beta \rangle.
\end{equation}
Uniqueness of solutions in linear programming occurs when the cost vector to be maximised (here $-C$ for the primal or $(\alpha, \beta)$ for the dual) does not point towards a face of the polyhedron to which it is orthogonal. The unique solution is then one of the vertices (extreme points) of the polyhedron. This is illustrated in Figure \ref{fig:uniqueness linear program}. Thus, in the discrete case, uniqueness of the primal and dual problems is understood by characterising the extreme points of the polyhedra $\Pi(\alpha, \beta)$ and $D(C)$.
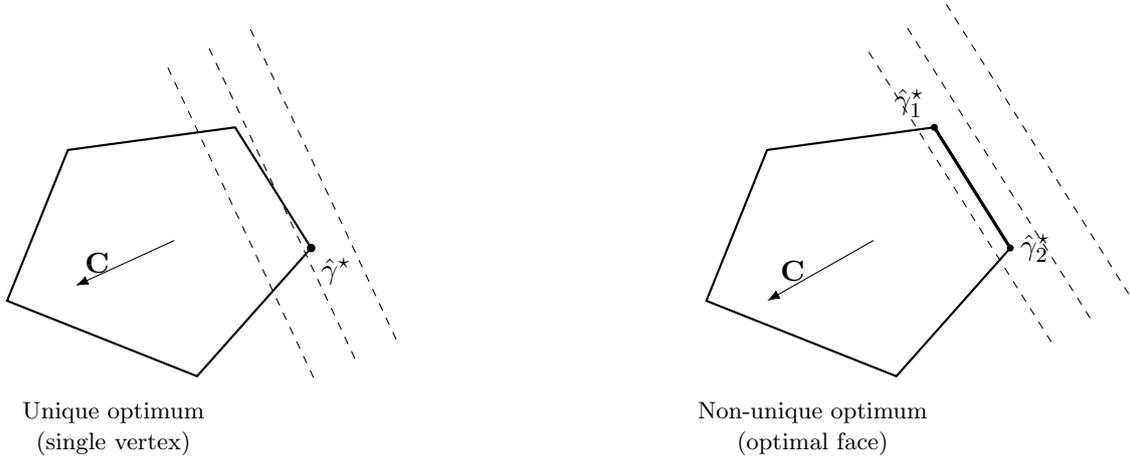
\begin{figure}[ht]
    \centering
\begin{tikzpicture}[scale=1, >=Latex]

\def\A{(-2.0,-0.5)}
\def\B{(-1.2,1.5)}
\def\C{(1.0,1.8)}
\def\D{(2.0,0.2)}
\def\E{(0.5,-1.5)}

\begin{scope}[shift={(-4.6,0)}]
  \draw[thick] \A -- \B -- \C -- \D -- \E -- cycle;

  \coordinate (cStart1) at (0.2,0.3);
  \coordinate (cEnd1)   at ($(cStart1)+(-1.3,-0.6)$);
  \draw[->] (cStart1) -- (cEnd1);
  \node[left=2pt] at ($(cStart1)!0.5!(cEnd1)$) {$\mathbf C$};

  \begin{scope}[rotate=295]
    \draw[dashed] (-2.3, 1.2) -- (2.3, 1.2);
    \draw[dashed] (-2.3, 1.8) -- (2.3, 1.8);
    \draw[dashed] (-2.3, 2.4) -- (2.3, 2.4);
  \end{scope}

  \fill \D circle (1.6pt) node[below right] {$\hat\gamma^\star$};
  
  \node[align=center] at (-0.6,-2.2) {\footnotesize Unique optimum\\[-2pt]\footnotesize (single vertex)};
\end{scope}

\begin{scope}[shift={(4.6,0)}]
  \draw[thick] \A -- \B -- \C -- \D -- \E -- cycle;

  \coordinate (cStart2) at (0.2,0.3);
  \coordinate (cEnd2)   at ($(cStart2)+(-1.4,-0.8)$);
  \draw[->] (cStart2) -- (cEnd2);
  \node[left=2pt] at ($(cStart2)!0.5!(cEnd2)$) {$\mathbf C$};

  \begin{scope}[rotate=302]
    \draw[dashed] (-2.3, 1.6) -- (2.3, 1.6);
    \draw[dashed] (-2.3, 2.2) -- (2.3, 2.2);
    \draw[dashed] (-2.3, 2.8) -- (2.3, 2.8);
  \end{scope}

  \draw[line width=1.2pt] \C -- \D;
  \fill \C circle (1.3pt) node[above left] {$\hat{\gamma}_1^\star$};
  \fill \D circle (1.3pt) node[right] {$\hat{\gamma}_2^\star$};
  
  \node[align=center] at (-0.6,-2.2) {\footnotesize Non-unique optimum\\[-2pt]\footnotesize (optimal face)};
\end{scope}

\end{tikzpicture}
    \caption{Uniqueness vs non-uniqueness of the minimisers of $\langle \hat{\gamma} |C \rangle$ depending on the direction of $-C$.}
    \label{fig:uniqueness linear program}
\end{figure}
 
The following lemma, due to \cite{klee1967facets}, characterises the extreme points of the primal transport polyhedron $\Pi(\alpha, \beta) \subset \R^{M \times N}$. This serves as a generalisation of the Birkhoff theorem for bi-stochastic matrices, which corresponds to the case $N = M$ and $\alpha, \beta$ uniform probability vectors. The second half of the proof presented here can be found in \cite[Chapter 8]{brualdi2006combinatorial}, the first half is formulated with the results of Chapter \ref{ch: pertubations} in mind.
\begin{lemma} \textnormal{\cite{klee1967facets}}
    \label{lemma: characterisation of extreme points of primal polyhedra}
    Given $\hat\gamma \in \Pi(\alpha, \beta)$, define the bipartite graph $G(\hat\gamma) = (V, E(\hat\gamma))$ of its support, with $V = (x_i)_{i=1}^M \cup (y_j)_{j=1}^N$ and
    \begin{equation}
        (x_i, y_j) \in E(\hat\gamma) \iff \hat\gamma_{ij} >0.
    \end{equation}
    Then $\hat\gamma \in \Pi(\alpha, \beta)$ is an extreme point if and only if $G(\hat\gamma)$ contains no cycles (it is a disjoint union of trees, often referred to as a forest).
\end{lemma}
\begin{proof}
    Assume $\hat\gamma$ is such that $G(\hat\gamma)$ has a cycle. Since no two $x$'s or $y$'s are connected, this is of the form $x_{i_1}, y_{j_1},..., x_{i_n}, y_{j_n}, x_{i_1}$, such that $\hat\gamma_{i_k j_k} >0$ and $\hat\gamma_{i_1 j_n} >0$ and $\hat\gamma_{i_{k+1} j_k} >0$. Here we choose $n \in \N$ to minimise the length of the chain
    \begin{equation}
    \label{eq: the chain fleetwood mac}
        i_1, j_1, ..., i_n, j_n, i_1,
    \end{equation}
    so the cycle is simple (it does not retrace itself). Set
    \begin{equation*}
        \varepsilon : = \min\left(\hat\gamma_{i_1 j_n}, \min_{k=1}^n \hat\gamma_{i_k j_k}, \min_{k=1}^{n-1} \hat\gamma_{i_{k+1} j_k}\right).
    \end{equation*}
    Then $\hat\gamma = \frac{1}{2}\hat\gamma_\varepsilon + \frac{1}{2} \hat\gamma_{-\varepsilon}$ where
    \begin{equation*}
        (\hat\gamma_{\varepsilon})_{ij} := \begin{cases}
            \hat\gamma_{ij} + \varepsilon & i,j \text{ appears in the chain \eqref{eq: the chain fleetwood mac},}\\
            \hat\gamma_{ij} - \varepsilon & j,i \text{ appears in the chain \eqref{eq: the chain fleetwood mac},} \\
            \hat\gamma_{ij} & \text{otherwise}.
        \end{cases}
    \end{equation*}
    By construction, $\hat\gamma_\varepsilon, \hat\gamma_{-\varepsilon} \in \Pi(\alpha, \beta)$. Hence $\hat\gamma$ is not an extreme point, since $\hat\gamma_{\varepsilon}$ and $\hat\gamma_{-\varepsilon}$ are distinct.

    Now assume the induced bipartite graph by $\hat\gamma \in \Pi(\alpha, \beta)$ is a forest. We show that there exists no other $\hat\gamma_0 \in \Pi(\alpha, \beta)$ whose induced graph is a subgraph of that of $\hat\gamma$ (recall that we are considering our graph as an unweighted graph, only the pattern is important here). Since $\alpha$ and $\beta$ are strictly positive probability vectors, each vertex has degree at least $1$. Since $G(\hat\gamma)$ is a forest, there exists a vertex of degree 1 on the graph, without loss of generality, let it be $y_1$, and let $(x_1, y_1)$ be the single edge connecting to $y_1$. So that the right amount of mass arrives at $y_1$, it must be that $\hat\gamma_{11} = \beta_1$. Repeatedly applying this observation to degree $1$ vertices on the tree as we remove them, we see that a tree graph uniquely specifies the element of $\Pi(\alpha, \beta)$. Thus, there can be no other $\hat\gamma_0 \in \Pi(\alpha, \beta)$ whose graph is a subgraph of $G(\hat\gamma)$. Consequently, if $\hat\gamma = \frac{1}{2}\hat\gamma^{-} + \frac{1}{2} \hat\gamma^{+}$, then $\hat \gamma = \hat\gamma^- = \hat\gamma^+$ since $G(\hat\gamma^-), G(\hat\gamma^+)$ are subgraphs of $G(\hat\gamma)$. Hence $\hat\gamma$ is an extreme point of $\Pi(\alpha, \beta)$.
\end{proof}

In \cite{balinski1984hirsch, balinski1984faces, balinski1983signatures}, extreme points for the dual transport polyhedron $D(C) \cap \{\phi_1=0\}$ are also characterised in terms of a bipartite graph with node sets corresponding to the support points $X$ and $Y$. (Here we fix the first coordinate arbitrarily, as otherwise there are no extreme points, with the lowest-dimensional faces being lines.)
\begin{lemma} \textnormal{\cite[Lemma 1]{balinski1984faces}}
\label{lemma: balinski dual extreme vertices characterisation}
    Given $(\phi, \psi) \in D(C) \cap \{\phi_1=0\}$, define the bipartite graph $G(\phi, \psi) = (V, E(\phi, \psi))$ with $V = (x_i)_{i=1}^M \cup (y_j)_{j=1}^N$ and
    \begin{equation}
        (x_i, y_j) \in E(\phi, \psi) \iff \phi_i + \psi_j = C_{ij}.
    \end{equation}
    Then $(\phi, \psi)$ is an extreme point of $D(C) \cap \{\phi_1=0\}$ if and only if $G(\phi, \psi)$ is connected.
\end{lemma}
\begin{proof}
Assume first that $(\phi, \psi)$ are such that $G(\phi, \psi)$ is connected. Then there exists a path from each vertex starting from $x_1$ (which represents $\phi_1=0$). By definition of $G(\phi, \psi)$, iterating along any path in the graph $1= i_1, j_1, i_2, j_2,...j_n$ in $G(\phi, \psi)$ we must have that $\phi_{i_{k+1}} = C_{i_{k+1}j_k} - \psi_{j_k}$, and $\psi_{j_{k}} = C_{i_kj_k} - \phi_{i_k}$. Hence, for a connected graph, the values of $(\phi, \psi)$ are uniquely determined. Let $(\phi^-, \psi^-), (\phi^+, \psi^+)$ be such that $(\phi, \psi) = \frac{1}{2}(\phi^-, \psi^-) + \frac{1}{2}(\phi^+, \psi^+)$. By definition of $D(C)$, for each $i, j$
\begin{equation*}
    \phi^-_i + \psi^-_j \leq C_{ij} \quad \text{and} \quad \phi^+_i + \psi^+_j \leq C_{ij}.
\end{equation*}
Hence $(x_i, y_j) \in E(\phi, \psi)$ if and only if it is in both $E(\phi^-, \psi^-)$ and $E(\phi^+, \psi^+)$; in other words $G(\phi, \psi) = G(\phi^-, \psi^-) \cap G(\phi^+, \psi^+)$. Consequently, if $G(\phi, \psi)$ is connected, both $G(\phi^-, \psi^-)$ and $G(\phi^+, \psi^+)$ are, and the same arcs are sufficient to connect them. It follows from the unique determinacy property of connected graphs shown above that $(\phi, \psi) = (\phi^-, \psi^-) =(\phi^+, \psi^+)$, and hence $(\phi, \psi)$ is an extreme point of $D(C)\cap\{\phi_1=0\}$.

Assume now that $(\phi, \psi)$ are such that $G(\phi, \psi)$ are disconnected. The above argument establishes that on each connected component of $G(\phi, \psi)$, the corresponding coordinates of $\phi$ and $\psi$ are fixed up to a constant. We will show that, pairwise, between different connected components of $G(\phi, \psi)$, there is some flexibility (we can add a small constant to all the coordinates in one component, while keeping the other components unchanged). This maximum amount of flexibility will define a sort of distance between two components. Since the graph is finite, choosing some component $D$ not containing $\phi_1$, and taking the minimum distance between all other components, we will be able to perturb all the values on this component slightly in each direction whilst respecting the constraints of $D(C)\cap\{\phi_1=0\}$. This will allow us to write $(\phi, \psi)$ as a non-trivial convex combination, proving that they do not constitute an extreme point.

Forgetting briefly the constraint $\{\phi_1=0\}$, decompose $G(\phi, \psi) = \sqcup_{k=1}^r G_{k}$, where $G_k = (V_k, E_k)$ are the connected components of $G(\phi, \psi)$. Consider two components $G_0, G_1$. We consider all the edges missing from $G(\phi, \psi)$ that could join $G_0$ to $G_1$, and see how much slack they have in their constraint
\begin{equation*}
\label{eq: slack cost constraint}
    C_{ij} - \phi_i -\psi_j >0.
\end{equation*}
First, we consider all the vertices starting from the $x$ side of $G_0$ to the $y$ side of $G_1$. For $k_0 \neq k_1$, set
\begin{equation}
    \label{eq: eps + definition}
    \varepsilon_{\phi, \psi}(G_{k}, G_{l}):= \min_{\substack{x_{i} \in V_{k}\\y_{j} \in V_{l}}} C_{ij} - \phi_{i} - \psi_{j} >0,
\end{equation}
(Note this is not at all symmetric in $G_{k_0}, G_{k_1}$) where the strict inequality comes from the fact that we are taking two vertices from different components, so none of these edges can be in the graph. Then for each coordinate $(i_0, j_1)$ in \eqref{eq: eps + definition}, adding $\varepsilon_{\phi, \psi}(G_{0}, G_{1})$ to each $\psi_{j_1}$ coordinate and subtracting it from each $\phi_{i_0}$ coordinate preserves all constraints within each component, as well as $(G_0, x)$ to $(G_1, y)$ edge constraints. Symmetrically, for each coordinate $(i_1, j_0)$ in, subtracting $\varepsilon_{\phi, \psi}(G_{1}, G_{0})$ from each $\psi_{j_1}$ coordinate and adding it to each $\phi_{i_0}$ coordinate preserves all constraints within each component, as well as $(G_0, y)$ to $(G_1, x)$ edge constraints. Returning now to our setting $(\phi, \psi) \in D(C) \cap \{\phi_1=0\}$ with disconnected graph choose any connected component $G_l$ not containing $\phi_1$, then set
\begin{equation*}
    \varepsilon^+_{\phi, \psi} = \min_{\substack{k=1,..,r\\ k \neq l}} \varepsilon_{\phi, \psi}(G_{k}, G_{l}) \quad  \text{and} \quad \varepsilon^-_{\phi, \psi} = \min_{\substack{k=1,..,r\\ k \neq l}} \varepsilon_{\phi, \psi}(G_k, G_l).
\end{equation*}
The above computations show that, keeping all the values of $(\phi, \psi)$ fixed on every component apart from $G_l$, we can add $\varepsilon^+$ to each $\psi$ and subtract from each $\phi$ value while staying inside $D(C) \cap \{\phi_1=0\}$. Similarly, we can subtract $\varepsilon^-$ from each $\psi$ and add it to each $\phi$. Hence, defining $(\phi^+, \psi^+)$ as the potential obtained by adding $\varepsilon:= \min(\varepsilon^+, \varepsilon^-)$ to each $\psi$ coordinate in $G_l$ and subtracting from each $\phi$ coordinate in $G_l$, and defining $(\phi^-, \psi^-)$ in the opposite manner with respect to addition and subtraction by $\varepsilon$, we have
\begin{equation*}
    (\phi, \psi) = \frac{1}{2} (\phi^+, \psi^+) + \frac{1}{2}(\phi^-, \psi^-).
\end{equation*}
Hence, we have written $(\phi, \psi)$ as a non-trivial linear combination of two elements of $D(C)\cap\{\phi_1=0\}$, and so it is not an extreme point.
\end{proof}

\begin{remark}
    The second half of the above proof involves quantitative calculations that are essentially identical to those of \cite[Theorem 3.9]{acciaio2025characterization}. We will return to this connection in later sections.
\end{remark}

In \cite{acciaio2025characterization}, the authors also study the uniqueness of solutions to the fully discrete dual transport problem. They also characterise uniqueness in terms of similar bipartite graphs as in Lemma \ref{lemma: balinski dual extreme vertices characterisation}. Here, they arrive at what is (usually) the same graph, but instead, it is constructed implicitly using the set of primal optimisers, combined with the compatibility condition \eqref{intro: compatibility condition}.
\begin{proposition} \textnormal{\cite[Proposition 3.5 (ii)/Corollary 3.14]{acciaio2025characterization}}
\label{prop: acciaio discrete dual uniqueness}
    For $\hat\gamma \in \Pi(\alpha, \beta)$, define the bipartite graph as in Lemma \ref{lemma: characterisation of extreme points of primal polyhedra}. Define the superposition of optimal graphs $G_\Gamma = (V, E_\Gamma)$ where
    \begin{equation*}
        E_\Gamma : = \bigcup_{\substack{\hat\gamma \in \Pi(\alpha, \beta)\\ \hat\gamma \text{ optimal for \eqref{eq: C(X, Y) dual linear programming problem}}}} E(\hat\gamma).
    \end{equation*}
    Then the solutions of \eqref{eq: C(X, Y) dual linear programming problem} are unique up to a constant if and only if $G_\Gamma$ is connected.
\end{proposition}

Condition \eqref{intro: compatibility condition} forces that for any dual optimisers $(\phi, \psi)$, $G_\Gamma \subset G(\phi, \psi)$. \cite[Proposition 3.6]{acciaio2025characterization} establishes the inverse direction $G_\Gamma = G(\phi, \psi)$ when dual potentials are unique up to a constant. Effectively, since in this case $G(\phi, \psi)$ is connected, if there were some edge in $G(\phi, \psi) \setminus G_\Gamma$, it must belong to a cycle of $G(\phi, \psi)$. The reason for this is that the tree-structured sections of $G(\phi, \psi)$ force all optimal plans supported there to have the same weights, so all edges of $G(\phi, \phi)$ not contained in a cycle must also lie in $G_\Gamma$. But if there is some edge lying in a cycle of $G(\phi, \psi)$ we can always take another optimal plan and permute some of the mass round the cycle as in the proof of Lemma \ref{lemma: characterisation of extreme points of primal polyhedra}, to create an optimal plan giving mass to this edge, and hence the edge is in $G_\Gamma$. Thus, in this case $G_\Gamma = G(\phi, \psi)$. In general, $G_\Gamma$ can be a strict subset of $G(\phi, \psi)$. To construct such a case, take any primal and dual problem for which $G(\phi, \psi)$ has at least two connected components, then add one edge to the graph joining the two components, by reducing a $C_{ij}$ connecting them until it is equal to $\phi_i + \psi_j$. Then no new optimal plans are created so $G_\Gamma$ remains unchanged, but $G(\phi, \psi)$ just gained an edge.

We note that in \cite{balinski1984faces}, the authors distinguish between two types of vertices of dual transport polyhedra, referred to as \textit{non-degenerate} when $G(\phi, \psi)$ is a spanning tree, and \textit{degenerate} when $G(\phi, \psi)$ contains a cycle. The degenerate case occurs precisely when we have non-uniqueness of the primal problem; this fact can be seen as a consequence of Lemma \ref{lemma: characterisation of extreme points of primal polyhedra}. 

\textbf{To synopsise the discrete case:} uniqueness of the primal problem occurs if and only if the corresponding graph $G_\Gamma$ has no cycles, and uniqueness of the dual if and only if $G_\Gamma$ is connected.

\section{The sets of optimal plans and potentials}

We now view the discrete transport problem again through continuous rather than linear programming language. As discussed, in general, we do not have uniqueness, so we are interested in the solution set of minimisers. We define
\begin{equation*}
        \Gamma_c : (\R^d)^M \times (\R^d)^N \rightrightarrows {\mathcal{P}(\R^{2d})}; 
 \end{equation*}
\begin{equation}
\label{eq: discrete transport problem X and Y}
    \Gamma_c(X, Y) : = \argmin_{\gamma \in \Pi(\rho_X, \mu_Y)} \int_{\R^d \times \R^d} c(x, y) \di \gamma(x, y),
\end{equation}
where $\rho_X$ and $\mu_Y$ are defined as in \eqref{eq: rho x and mu y definition}. We denote
\begin{equation*}
    \supp \Gamma(X, Y) = \bigcup_{\gamma \in \Gamma(X, Y)} \supp \gamma,
\end{equation*}
sometimes viewing $\supp \Gamma$ as a graph and sometimes as a set, as it should be clear from the context to which we refer. We know from the theory of the previous section that the optimal set should correspond to a (potentially zero-dimensional) face of the polyhedron $\Pi(\alpha, \beta)$. 
Consider sets of points of the form
\begin{equation*}
    P = \{(x_i, y_i)\}_{i=1}^n \subset \supp \Gamma;
\end{equation*}
which are \textit{distinct} in that
\begin{equation}
\label{eq: distinct points}
  x_i \neq x_j \quad\text{ and }\quad y_i \neq y_j \quad \text{ for all } i \neq j.  
\end{equation}
Under these hypotheses,
\begin{equation}
\label{eq: distinct shift gives new set}
    \{(x_{i+1}, y_i)\}_{i=1}^n \cap P = \emptyset,
\end{equation}
where $x_{n+1} := x_1$. Given such a $P$, we define the cyclical monotonicity gap
\begin{equation*}
    \Delta_P = \sum_{i=1}^n c(x_{i+1}, y_i) - c(x_i, y_i).
\end{equation*}
 We can then say the following.
\begin{proposition}
\label{prop: uniqueness iff strict cm}
    Let $X \in (\R^d)^M$ and $Y \in (\R^d)^N$ and consider the transport problem \eqref{eq: discrete transport problem X and Y} for some $c: \R^d \times \R^d \to \R$. Then
    \begin{enumerate}[label=(\roman*)]
        \item \label{enum: weak inequality} $\gamma \in \Gamma(X, Y)$ if and only if $\Delta_P \geq 0$ for all distinct (in the sense of \eqref{eq: distinct points}) $P$ in $\supp \gamma$.
        \item \label{enum: strict gives uniqueness} Problem \eqref{eq: discrete transport problem X and Y} has uniqueness of solutions if and only if there exists $\gamma \in \Pi(\rho_X, \mu_Y)$, such that $\Delta_P > 0$ for all distinct (in the sense of \eqref{eq: distinct points}) $P \subset \supp \gamma$ of length $\geq 2$ (and this $\gamma$ is the unique solution).
    \end{enumerate}
\end{proposition}
\begin{proof}
    Claim \ref{enum: weak inequality} is effectively a rephrasing of the classical cyclical monotonicity in the discrete case, as stated in the introduction \eqref{intro: c cm definition statment}. We know optimality is equivalent to testing $\Delta_P \geq 0$ for all $P=\{(x_i, y_i)\}_{i=1}^n \subset \supp \gamma$ without distinctness assumptions. We write the cycle
    \begin{equation*}
        x_1, y_1, x_2, y_2, ..., x_n, y_n, x_1,
    \end{equation*}
    and decompose into cycles $\{P_k\}_{k=1}^m$ of maximal length in which no $x$ or $y$ point is equal to another in the cycle. Then each $P_k = \{(x_i, y_i)\}_{i=1}^{n_k}$  is distinct, and $\Delta_P = \sum_{k=1}^m \Delta_{P_k}$, so that it is sufficient to test only distinct $P$ for optimality as required.

    Claim \ref{enum: strict gives uniqueness} can be seen as follows. Let $\gamma \in \Gamma(X, Y)$, assume that there exists $P = (x_i, y_i)_{i=1}^n \subset \supp \gamma$ distinct in the sense of \eqref{eq: distinct points}, with $\Delta_P = 0$. Define $\Tilde{\gamma}$ by
    \begin{equation*}
        \tilde\gamma := \gamma + \varepsilon \sum_{i=1}^n \left(\delta_{x_{i+1}, y_i} - \delta_{x_i, y_i}\right),
    \end{equation*}
    where $\varepsilon = \min_{i=1,..n} \gamma((x_i, y_i)) >0$. We have $\tilde\gamma \in \Pi(\rho_X, \mu_Y)$ by construction. By \eqref{eq: distinct shift gives new set}, $\gamma$ and $\tilde \gamma$ are distinct transport plans, and by $\Delta_P = 0$, $\int c \di \gamma = \int c \di \tilde\gamma$, so $\tilde \gamma$ is also an optimal transport plan for $c$ cost, and the problem does not have a unique solution.

    Conversely, assume that we have uniqueness, take $\gamma_0, \gamma_1 \in \Pi(\rho_X, \mu_Y)$ as distinct optimal plans. Then the measure $\gamma_0 - \gamma_1$ is a non-zero signed zero mass discrete probability measure on $\R^d \times \R^d$ with zero first and second marginals. Consider the bipartite graph of $\supp (\gamma_0 - \gamma_1)$ defined as $(V, E)$ where
    \begin{equation*}
        V = p_X(\supp (\gamma_0 - \gamma_1)) \cup p_Y(\supp (\gamma_0 - \gamma_1)),
    \end{equation*}
    and $E = E^+ \sqcup E^-$ where
    \begin{equation*}
        E^+ = \left\{ (x, y) : (\gamma_0 - \gamma_1)(x, y) > 0 \right\} \quad \text{and} \quad E^- = \left\{ (x, y) : (\gamma_0 - \gamma_1)(x, y) < 0 \right\}.
    \end{equation*}
    We have $E^+ \subset \supp \gamma_0$ and $E^- \subset \supp \gamma_1$. Every vertex has at least one positive and one negative edge leaving from it (since $\gamma_0 - \gamma_1$ has zero marginals). And the graph is simple, with at most one edge joining any two points, either positive or negative. We will construct a simple cycle in the graph. Start from some $x_{i_1}$, there exists a positive edge to some $y_{j_1}$. Then there exists some negative edge to some $x_{i_2}$ in the graph, necessarily $i_2 \neq i_1$ since the graph is simple and the signs of the edges are different. We repeat this process, always choosing to visit a new vertex, until eventually we are forced to revisit a vertex for the first time. Let $y_{j_k}$ the last distinct point before we subsequently return to vertex $x_{i_m}$ for some $m<k$. (We cannot have $i_m = i_k$ since the signs of their edges in the graph must be different.) Up to removing the points before $i_m$ and relabelling, it follows that there exists a nontrivial simple cycle of the form
    \begin{equation}
        \lefteqn{\overbrace{\phantom{x_1,\, y_1}}^{\in E^+}} x_1,\, \lefteqn{\underbrace{\phantom{y_1,\, x_2}}_{\in E^-}} y_1,\, \lefteqn{\overbrace{\phantom{x_2,\, y_2}}^{\in E^+}} x_2,\, y_2 \cdots \lefteqn{\overbrace{\phantom{x_n,\, y_n}}^{\in E^+}} x_n,\, \underbrace{y_n,\, x_1}_{\in E^-}.
    \end{equation}
    Take $P = (x_i, y_i)_{i=1}^n$, then since the cycle is simple, $P$ is distinct in the sense of \eqref{eq: distinct points}. Moreover, $P \subset \supp \gamma_0$ since these edges are all in $E^+$. For any optimal potential pair $\phi, \psi \in \R^{M \times N}$ for \eqref{eq: C(X, Y) dual linear programming problem}, the compatibility condition with both $\gamma_0$ and $\gamma_1$ implies $\phi_i + \psi_i = c(x_i, y_i)$ and $\phi_{i+1} + \psi_i = c(x_{i+1}, y_i)$, and consequently
    \begin{equation*}
        \Delta_P=\sum_{i=1}^n c(x_{i+1}, y_i) - c(x_i, y_i) = \sum_{i=1}^n (\phi_{i+1} + \psi_i - \phi_i - \psi_i) = 0.
    \end{equation*}
    Hence, we have shown that non-uniqueness implies that for any optimal $\gamma_0$, there exists a distinct $P \subset \gamma_0$ with $\Delta_P = 0$. The characterisation \eqref{enum: strict gives uniqueness} follows.
\end{proof}

We define similar optimal sets for the Kantorovich potentials. We will use the semi-dual formulation \eqref{intro: c semidual problem}, which takes a special formulation as soon as one of the measures is discrete. In general, we are only interested in the values of the potentials on the support points $X$ and $Y$,  and so we can reduce to a finite-dimensional problem. We define
\begin{equation*}
        \Phi_c : (\R^d)^M \times (\R^d)^N \rightrightarrows \R^M, \quad \quad \Psi_c : (\R^d)^M \times (\R^d)^N \rightrightarrows \R^N; 
 \end{equation*}
\begin{equation}
\label{discrete: c semi discrete phi vect}
    \Phi_c(X, Y) := \argmax_{\phi \in \R^M}  \langle \phi | \alpha \rangle_{\R^M} + \int_\X \phi^c(y | X) \di \mu_Y(y),
\end{equation}
\begin{equation}
\label{discrete: c semi discrete psi vect}
    \Psi_c(X, Y) := \argmax_{\psi \in \R^N} \int_\X \psi^c(x |Y) \di \rho_{X}(x) + \langle \psi | \beta\rangle_{\R^N},
\end{equation}
where the $c$-transforms are taken with respect to points in the support of the target $X$ and $Y$:
\begin{equation*}
    \psi^c(x |Y) := \inf_{j=1,...N} c(x, y_j) - \psi_j, \quad \text{ and } \quad \phi^c(y |X) := \inf_{i=1,...M} c(x_i, y) - \psi_i.
\end{equation*}
Problems \eqref{discrete: c semi discrete phi vect} and \eqref{discrete: c semi discrete psi vect} are referred to as the \textit{semi-discrete} transport problem, which has been the subject of numerous recent works, see \cite{kitagawa2019convergence,levy2015numerical, gallouet2018lagrangian,merigot2011multiscale}.
 Due to symmetry, we can define both $\Phi$ and $\Psi$ as above. It is not clear which is more useful, or if both are necessary (recall sometimes we want to fix $X$ and perturb $Y$ so our perturbations are not necessarily symmetric). To avoid developing a dual theory unnecessarily, we will concentrate on problem \eqref{discrete: c semi discrete phi vect}. Symmetric calculations apply to \eqref{discrete: c semi discrete psi vect}.
 
 On one side of \eqref{discrete: c semi discrete phi vect}, we have a vector in $\R^M$, corresponding to the dual in the linear programming formulation, and on the other side we have $\phi^c \in C_b(\Y)$, corresponding to the measure theoretic formulation. In this way, \eqref{discrete: c semi discrete phi vect} will serve as a bridge between these two formulations, allowing us to pass stability results from linear programming to functional analysis language. Then, qualitative stability, Theorem \ref{intro: qualitative stability}, should allow us to take a limit of discrete measures to say statements about more general optimal transport problems beyond the discrete case. For ease of exposition, we fix the cost to be quadratic, $c(x, y) = -\langle x | y \rangle$, and state the generalisations to general $c$ after. Here we simply denote $\Gamma$ and $\Phi$ for the sets of optimisers. In this case, \eqref{discrete: c semi discrete phi vect} becomes
\begin{equation}
\label{eq: semi discrete problem}
    \inf_{\phi \in \R^N} \langle \phi | \alpha\rangle_{\R^M} + \int_{\R^d} \phi^*(y) \di \mu(y).
\end{equation}
with Legendre transform $\psi^* : \R^d \to \R$ defined by
    \begin{equation}
    \label{eq: discrete leg transf sup}
        \phi^*(y |X) : = \sup_{i = 1,..., M} \langle y | x_i \rangle - \phi_i.
    \end{equation}
To keep notation light, where clear, we omit the dependence on $X$ in the Legendre transform. The function $\phi^* : \R^d \to \R$ is a finite supremum of hyperplanes, with one hyperplane $y \mapsto \langle y | x_i \rangle - \phi_i$ of gradient of $x_i$  corresponding to each point in the support of $\rho_X$. Thus $\phi^*$ is piecewise flat, and the region for which the hyperplane $i$ is active in the sup \eqref{eq: discrete leg transf sup} is called the $i$th Laguerre cell of $\phi \in \R^M$, denoted
\begin{equation*}
    \lag_i(\phi | X) = \left\{ y \in \R^d :  \langle y | x_i \rangle - \phi_i \geq \langle y | x_j \rangle - \phi_j \text{ for all } j = 1,..., M \right\} = \left\{ y \in \R^d : x_i \in \partial \phi^*(y) \right\}.
\end{equation*}
These cells have pairwise disjoint interiors, providing a decomposition of $\R^d$. This definition \textit{does not} depend on $\mu$, only on the points $X \in (\R^d)^M$. As \eqref{eq: semi discrete problem} is a convex minimisation problem, for each $X, Y \in \R^N$, $\Phi(X, Y)$ is a convex set. We will give an explicit characterisation of $\Phi$ as a finite intersection of certain half-spaces. This will allow us to study how $\Phi$ changes as we perturb $X$ and $Y$. 
\begin{lemma}
\label{lem: dual vect opt iff all Si}
For any vector of target positions $Y \in (\R^d)^N$, set
\begin{equation}
    S_i(X,Y) : = \{  y \in \{y_j\}_{j=1}^N: \text{ there exists } \gamma \in \Gamma(X,Y) \text{ such that } (x_i, y) \in \supp \gamma\}.
\end{equation}
Then
\begin{equation*}
    \Phi(X,Y) = \left\{ \phi \in \R^N : S_i(X,Y) \subset \lag_i(\phi | X) \;\text{ for all }\; i = 1,...,M \right\}.
\end{equation*}
\end{lemma}
\begin{proof}
By definition $y \in \lag_i(\phi |X)$ if and only if $x_i \in \partial \phi^*(y)$. Hence $S_i(X,Y) \subset \lag_i(\phi | X)$ for all $i=1,...,M$ if and only if for any $\gamma \in \Pi(\rho_X, \mu_Y)$ optimal and any $(x_i, y) \in \supp \gamma$, we have $x_i \in \partial \phi^*(y)$. But this is precisely saying that $\supp \gamma \subset \partial \phi^*$. By the compatibility condition \eqref{intro: compatibility condition}, this is equivalent to $\phi \in \Phi(X,Y)$.
\end{proof}

\begin{proposition}[Characterisation of the set of dual potentials]
\label{prop: dual set characterisation}
Let $\rho_X$ and $\mu_Y$ be discrete probability measures on $\R^d$ for some $X \in (\R^d)^M$ and $Y \in (\R^d)^N$. Then $\Phi(X, Y)$ is a convex polyhedron in $\R^M$ characterised by a finite intersection of half spaces, one for each pairing of distinct points $x, x' \in \{x_i\}_{i=1}^M$. Explicitly,
\begin{equation}
\label{eq: characterisation of set of dual potentials as intersection of hyperplanes}
    \Phi(X,Y) = \bigcap_{i, j} H_{ij}(X,Y),
\end{equation}
where
\begin{equation}
\label{eq: Hijx hyperplane definition}
    H_{ij}(X,Y): = \left\{ \phi \in \R^M : \langle \phi | e_i - e_j \rangle \leq h_{ij}(X,Y) \right\},
\end{equation}
$e_i$ are basis elements of $\R^M$, and
\begin{equation*}
    h_{ij} : (\R^d)^M \times (\R^d)^N \to \R; \quad X, Y \mapsto \inf_{y \in S_i(X,Y)} \langle x_i - x_j | y \rangle.
\end{equation*}
\end{proposition}

\begin{proof}
We will show that
\begin{equation*}
    \phi \in \bigcap_{i, j} H_{ij} \iff S_i \subset \lag_i(\phi) \text{ for all } i,
\end{equation*}
which implies the result by Lemma \ref{lem: dual vect opt iff all Si}. We have $S_i \subset \lag_i(\phi)$ if and only if for all $y \in S_i$ and all $j=1,...M$,
\begin{equation}
\label{eq: point in cell iff psi hyperplane}
    \langle x_i | y \rangle - \phi_i \geq \langle x_j | y \rangle - \phi_j.
\end{equation}
This is equivalent to asking that for all $j =1,...,M$,
\begin{equation*}
    \langle \phi | e_i - e_j \rangle_{\R^M} \leq \inf_{y \in S_i}\langle x_i - x_j | y \rangle_{\R^d} = h_{ij}.
\end{equation*}
Hence $\phi \in H_{ij}$ for all $j \neq i$, and so
\begin{equation*}
    \phi \in \bigcap_{j=1}^M H_{ij} \iff S_i \subset \lag_i(\phi).
\end{equation*}
Finally, we want $S_i \subset \lag_i(\phi)$ for any $i=1,..,M$. Hence, we take the intersection over $i$, giving the characterisation \eqref{eq: characterisation of set of dual potentials as intersection of hyperplanes}.
\end{proof}

\begin{remark}
We have two parallel half-spaces with gradient $e_i - e_j$ in $\R^M$. Their intersection $\{ \phi: \phi_i - \phi_j \in [-h_{ji}, h_{ij}]\}$ is all the admissible intervals of the relative vertical displacements between the hyperplane associated to $x_i$ and that associated to $x_j$, which preserve $S_i \subset \lag_i(\phi)$ and $S_j \subset \lag_j(\phi)$. Each of the hyperplanes defining $H_{ij}$ is orthogonal to the vector $(1, 1, ..., 1)$, and so the optimal set is symmetrical in this direction. This is consistent with the fact that potentials are only ever defined up to a constant. What matters are the relative heights of the hyperplanes defining $\phi^*$, as these decide the cell boundaries. If two hyperplanes are both translated vertically by the same amount, their intersection set remains the same, so if their cells are adjacent, the boundary is preserved.

The compatibility condition \eqref{intro: compatibility condition} forces that at $y \in \supp \mu_Y$ the subgradient of any optimal $\phi^*: \Y \to \R$ must contain all $x$ such that $(x, y) \in \supp \gamma$ for any $\gamma \in \Gamma(X, Y)$. We then have the liberty to choose any convex function whose subdifferential contains these values (it can always also contain more values too; \eqref{intro: compatibility condition} demands containment, not equality). If $\mu_Y$ is a very fine mesh, then the subgradients of $\phi^*$ are specified on a very dense set, so there is less flexibility in choice. If $\rho_X$ is a very fine mesh compared to $\mu_Y$, then for $\phi^{**} : \X \to \R$, many $y$ will lie in the subdifferentials of multiple $x_i$. This also leaves little choice, as we are forced to interpolate affinely between any such pair of points so that the value can be in the subgradient of both.  
Many points $x_i \in \supp \rho$ will send their mass split to one $y$, and so for any potential $\phi^*$, these $y$ should lie at the intersection of many cells rather than just inside a single cell. 
\end{remark}

\begin{remark}[Extension to cost $c$]
The extension of this to cost $c$ simply uses generalised Laguerre cells
\begin{equation*}
    \lag_{c}(\phi) = \left\{ y \in \R^d : c(x_i, y) - \phi_i \leq c(x_j, y) - \phi_j\;\; \forall j= 1,...M\right\}.
\end{equation*}
The compatibility condition demands $S_i^c \subset \lag_i^c(\phi)$ where $S_i^c$ is defined similarly to the quadratic case, using $\Gamma_c$ instead. The hyperplane constraints are given explicitly using $h_{ij}^c$ defined by
\begin{equation*}
    h_{ij}^c : (\R^d)^M \times (\R^d)^N \to \R; \quad X, Y \mapsto \sup_{y \in S_i^c(X,Y)} c(x_i, y) -c(x_j, y).
\end{equation*}
$\phi \in \R^M$ is optimal if an only if for all $i = 1,...M$ and $j=1,...N$,
\begin{equation*}
    \langle \phi | e_i - e_j \rangle \geq h_{ij}^c(X, Y).
\end{equation*}
\end{remark}

\begin{example}
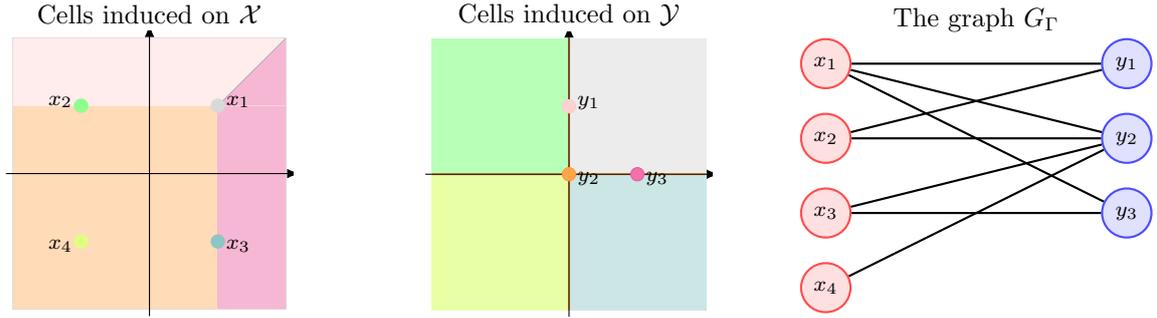
\begin{figure}
\centering
\begin{minipage}{0.31\linewidth}
\centering
\begin{tikzpicture}[>=Latex,scale=0.9]
  \tikzset{
    cellA/.style={fill=orange!28},
    cellB/.style={fill=magenta!35},
    cellC/.style={fill=pink!30},
    gridf/.style={very thin, gray!25},
    axisHL/.style={line width=0.8pt, orange!70},
    site1/.style={draw=lightgray!60, fill=lightgray!60, circle, minimum size=5pt, inner sep=0pt},
    site4/.style={draw=lime!50, fill=lime!50, circle, minimum size=5pt, inner sep=0pt},
    site3/.style={draw=teal!45, fill=teal!45, circle, minimum size=5pt, inner sep=0pt},
    site2/.style={draw=green!45, fill=green!45, circle, minimum size=5pt, inner sep=0pt},
    lab/.style={black, inner sep=1.5pt, font=\scriptsize}
  }

  \begin{scope}
    \clip (-2.1,-2.1) rectangle (2.1,2.1);

    \fill[cellA] (-2,-2) rectangle (1,1);

    \fill[cellB] (1,-2) rectangle (2,2);

    \fill[cellC] (-2,1) rectangle (2,2);

    \fill[cellB] (1,1) -- (2,2) -- (2,1) -- cycle;

    \draw[gray!60] (1,1) -- (2,2);

    \draw[gridf] (-2,-2) rectangle (2,2);


    \draw[->] (-2.1,0) -- (2.2,0) node[right] {$x_1$};
    \draw[->] (0,-2.1) -- (0,2.2) node[above] {$x_2$};
  \end{scope}

  \node[site1] (x1) at (1,1) {};
  \node[site2] (x2) at (-1,1) {};
  \node[site3] (x3) at (1,-1) {};
  \node[site4] (x4) at (-1,-1) {};

  \node[lab, anchor=west]  at ($(x1)+(0.06,0.06)$) {$x_1$};
  \node[lab, anchor=east]  at ($(x2)+(-0.06,0.06)$) {$x_2$};
  \node[lab, anchor=west]  at ($(x3)+(0.06,-0.06)$) {$x_3$};
  \node[lab, anchor=east]  at ($(x4)+(-0.06,-0.06)$) {$x_4$};

  \node[font=\small] at (0,2.35) {Cells induced on $\X$};
\end{tikzpicture}
\end{minipage}\hfill
\begin{minipage}{0.31\linewidth}
\centering
\begin{tikzpicture}[>=Latex,scale=0.9]
  \tikzset{
    cellG/.style={fill=green!28},
    cellY/.style={fill=lime!35},
    cellP/.style={fill=lightgray!30},
    cellB/.style={fill=teal!20},
    gridf/.style={very thin, gray!25},
    cut/.style={line width=0.8pt, orange!70},
    site1/.style={circle, draw=pink!80, fill=pink!70, minimum size=5pt, inner sep=1.2pt},
    site2/.style={circle, draw=orange!80, fill=orange!70, minimum size=5pt, inner sep=1.2pt},
    site3/.style={circle, draw=magenta!80, fill=magenta!70, minimum size=5pt, inner sep=1.2pt},
    lab/.style={black, inner sep=1.5pt, font=\scriptsize}
  }

  \begin{scope}
    \clip (-2.1,-2.1) rectangle (2.1,2.1);

    \fill[cellG] (-2,0) rectangle (0,2);  
    \fill[cellY] (-2,-2) rectangle (0,0); 
    \fill[cellP] (0,0) rectangle (2,2);   
    \fill[cellB] (0,-2) rectangle (2,0);  


    \draw[cut] (0,-2) -- (0,2);
    \draw[cut] (-2,0) -- (2,0);

    \draw[->] (-2.1,0) -- (2.2,0) node[right] {$y_1$};
    \draw[->] (0,-2.1) -- (0,2.2) node[above] {$y_2$};
  \end{scope}

  \node[site1] (y1) at (0,1)  {};
  \node[site2] (y2) at (0,0)  {};
  \node[site3] (y3) at (1,0)  {};

  \node[lab, anchor=west] at ($(y1)+(0.06,0.06)$) {$y_1$};
  \node[lab, anchor=west] at ($(y2)+(0.06,-0.06)$) {$y_2$};
  \node[lab, anchor=west] at ($(y3)+(0.06,-0.06)$) {$y_3$};

  \node[font=\small] at (0,2.35) {Cells induced on $\Y$};
\end{tikzpicture}
\end{minipage}\hfill
\begin{minipage}{0.31\linewidth}
\begin{tikzpicture}[>=Latex,scale=0.9]
\def\nodesize{6.5mm}
\def\xsep{2.2cm}
\def\yTop{1.7cm}
\def\yStep{1.1cm}

\tikzset{
  redcircle/.style  ={circle, draw=red!70, thick, fill=red!12,  minimum size=\nodesize, inner sep=0pt},
  bluecircle/.style ={circle, draw=blue!70, thick, fill=blue!12, minimum size=\nodesize, inner sep=0pt},
  lab/.style        ={font=\scriptsize, inner sep=0pt, outer sep=0pt},
  edge/.style       ={thick},
}
  \node[font=\small] at (0,2.35) {The graph $G_\Gamma$};
\node[redcircle] (A1) at (-\xsep, \yTop) {};
\node[lab] at (A1) {$x_1$}; 

\node[redcircle] (A2) at (-\xsep, \yTop-\yStep) {};
\node[lab] at (A2) {$x_2$}; 

\node[redcircle] (A3) at (-\xsep, \yTop-2*\yStep) {};
\node[lab] at (A3) {$x_3$}; 

\node[redcircle] (A4) at (-\xsep, \yTop-3*\yStep) {};
\node[lab] at (A4) {$x_4$}; 

\node[bluecircle] (B1) at (\xsep, \yTop) {};
\node[lab] at (B1) {$y_1$}; 

\node[bluecircle] (B2) at (\xsep, \yTop-\yStep) {};
\node[lab] at (B2) {$y_2$}; 

\node[bluecircle] (B3) at (\xsep, \yTop-2*\yStep) {};
\node[lab] at (B3) {$y_3$}; 

\draw[edge] (A1) -- (B1);
\draw[edge] (A1) -- (B2);
\draw[edge] (A1) -- (B3);

\draw[edge] (A2) -- (B1);
\draw[edge] (A2) -- (B2);

\draw[edge] (A3) -- (B2);
\draw[edge] (A3) -- (B3);

\draw[edge] (A4) -- (B2);

\end{tikzpicture}
\end{minipage}\hfill
\caption{Laguerre cell decompositions of the source (left) and target (middle) domains induced by an optimal dual vector, and corresponding graph $G_\Gamma$ (right).}
\label{fig:discrete-cell-example}
\end{figure}
In $\R^2$, let $\rho = \sum_{i=1}^4 \frac{1}{4}\delta_{x_i}$ and $\mu = \sum_{j=1}^3 \frac{1}{3} \delta_{y_j}$,  where $\{x_i\}_{i = 1}^4$ and $\{y_j\}_{j=1}^3$ are as in Figure \ref{fig:discrete-cell-example}. The points on the left picture of Figure \ref{fig:discrete-cell-example} correspond to the cells (or hyperplanes) in the centre picture, and vice versa. Non-uniqueness of the primal problem occurs precisely when there is a non-trivial cycle of the form point 1, hyperplane 1, point 2, hyperplane 2, $\cdots$ hyperplane $n$, point 1, where each point lies on the boundary of the hyperplanes it is next to in the cycle. Here, there exist three such cycles in $G_\Gamma$, given by
\begin{equation*}
    (x_1, y_2, x_2, y_1, x_1), \; (x_1, y_2, x_3, y_3, x_1), \; \text{ and } (x_1, y_3, x_3, y_2, x_2, y_1, x_1).
\end{equation*}
\end{example}
\chapter{Perturbation of support positions}
\label{ch: pertubations}

We now investigate how solutions to discrete optimal transport problems change when the support points $X$ and $Y$ are perturbed. This leads naturally to questions of stability for both primal and dual formulations.

\section{Discussion on how to use linear programming connections}

The following could lead to a better understanding of Problem \ref{problem: general plan stability}:

\begin{problem}
\label{problem: discrete plan stability}
    How does $\Gamma(X, Y)$ change as we perturb $X$ and $Y$? To what extent can we say that if $X, X'$ and $Y, Y'$ are close (and hence $W_2(\rho_X, \rho_{X'})$ and $W_2(\mu_Y, \mu_{Y'})$ are small), then $\Gamma_c(X, Y)$ and $\Gamma_c(X', Y')$ are quantitatively close with respect to an appropriate metric? We would like results of the form
\begin{equation}
\label{eq: stability of plans discrete goal statement}
    W_2(\gamma_{X_0,Y_0}, \gamma_{X_1, Y_1}) \leq C \Bigl(W_2(\rho_{X_0}, \rho_{X_1}) + W_2(\mu_{Y_0}, \mu_{Y_1}) \Bigr)^q + \mathcal{E}(X, Y, \alpha, \beta)
\end{equation}
for all $\gamma_{X_i, Y_i} \in \Gamma(X_i, Y_i)$, where $\mathcal{E}(X, Y, \alpha, \beta)$ is an error term to be understood depending on the structure of the measures.

If we fix $X$ and just vary $Y$ with all points staying inside the same compact set, can we arrive at an error $\E(X, \alpha)$ not depending on $Y$, which is, for example, very small when $\rho_X$ represents a very fine mesh discrete approximation of a regular absolutely continuous measure? Note here that in the spirit of known results in the absolutely continuous case, we should consider some constraint that all the coordinates of $Y$ live inside the same compact set $\Y$, and the error should be allowed to depend on $\Y$ also.\end{problem}

In general, quantitative stability cannot hold without some form of error term $\mathcal{E}$. Sometimes this error can be large, as the following example demonstrates.
\begin{example}
\label{ex: non uniqueness gives jump discontinuity}
Let \textcolor{sourcecolour}{$\rho_\varepsilon = \frac{1}{2} \delta_{(-1, \varepsilon)} + \frac{1}{2} \delta_{(1, -\varepsilon)}$} and \textcolor{targetcolour}{$\mu = \frac{1}{2} \delta_{(0, 1)} + \frac{1}{2} \delta_{(0 , -1)}$}, see Figure \ref{fig:stability-failure}. For $\varepsilon \neq 0$, there is a unique optimal plan $\gamma_\varepsilon$, but at $\varepsilon= 0$ we have non-uniqueness, with any convex combination of the two limiting cases giving an optimal plan.
\begin{figure}[h]
    \centering

\begin{tikzpicture}
    \draw (0, 1.5) node[above=3pt]{\scriptsize$\varepsilon>0$};
    \draw (-1.5, 0) -- (1.5, 0);
    \draw (0, -1.5) -- (0, 1.5);
    \draw[->, mixed, thick] (-0.85, 0.32) -- (-0.15, 0.88);
    \draw[->, mixed, thick] (0.85, -0.32) -- (0.15, -0.88);
    
    \draw (0,1) node[cross=3pt, targetcolour]{};
    \draw (0,1) node[anchor=west, targetcolour]{\scriptsize$(0, 1)$};
    
    \draw (0,-1) node[cross=3pt, targetcolour]{};
    \draw (0,-1) node[anchor=east, targetcolour]{\scriptsize$(0, -1)$};
    
    \draw (-1,0.2) node[cross=3pt, sourcecolour]{};
    \draw (-1, 0.2) node[below = 3pt, sourcecolour]{\scriptsize$(-1, \varepsilon)$};
    
    \draw (1,-0.2) node[cross=3pt, sourcecolour]{};
    \draw (1, -0.2) node[above=3pt, sourcecolour]{\scriptsize$(1, -\varepsilon)$};
\end{tikzpicture}
\hfill
\begin{tikzpicture}
    \draw (0, 1.5) node[above=3pt]{\scriptsize$\varepsilon=0$};
    \draw (-1.5, 0) -- (1.5, 0);
    \draw (0, -1.5) -- (0, 1.5);
    \draw[->, mixed, thick] (-0.85, -0.15) -- (-0.15, -0.85);
    \draw[->, mixed, thick] (-0.85, 0.15) -- (-0.15, 0.85);
    
    \draw[->, mixed, thick] (0.85, 0.15) -- (0.15, 0.85);
    \draw[->, mixed, thick] (0.85, -0.15) -- (0.15, -0.85);
    
    \draw (0,1) node[cross=3pt, targetcolour]{};
    
    \draw (0,-1) node[cross=3pt, targetcolour]{};
    
    \draw (-1,0) node[cross=3pt, sourcecolour]{};
    
    \draw (1,0) node[cross=3pt, sourcecolour]{};
\end{tikzpicture}
\hfill
\begin{tikzpicture}
    \draw (0, 1.5) node[above=3pt]{\scriptsize$\varepsilon<0$};
    \draw (-1.5, 0) -- (1.5, 0);
    \draw (0, -1.5) -- (0, 1.5);
    \draw[->, mixed, thick] (-0.85, -0.32) -- (-0.15, -0.88);
    \draw[->, mixed, thick] (0.85, 0.32) -- (0.15, 0.88);
    
    \draw (0,1) node[cross=3pt, targetcolour]{};
    \draw (0,1) node[anchor=east, targetcolour]{\scriptsize$(0, 1)$};
    
    \draw (0,-1) node[cross=3pt, targetcolour]{};
    \draw (0,-1) node[anchor=west, targetcolour]{\scriptsize$(0, -1)$};
    
    \draw (-1,-0.2) node[cross=3pt, sourcecolour]{};
    \draw (-1, -0.2) node[above = 3pt, sourcecolour]{\scriptsize$(-1, -\varepsilon)$};
    
    \draw (1,0.2) node[cross=3pt, sourcecolour]{};
    \draw (1, 0.2) node[below=3pt, sourcecolour]{\scriptsize$(1, \varepsilon)$};
\end{tikzpicture}

    \caption{ Optimal plans between \textcolor{sourcecolour}{$\rho_\varepsilon$} and \textcolor{targetcolour}{$\mu$}. }
    \label{fig:stability-failure}
\end{figure}
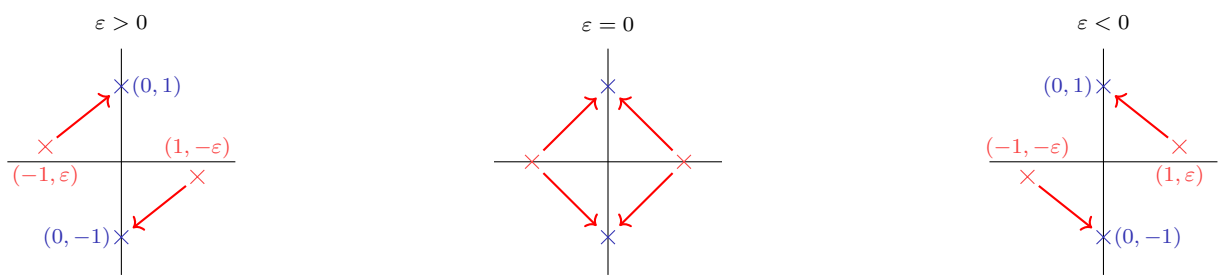
For $\varepsilon>0$, $W_2(\gamma_\varepsilon, \gamma_{-\varepsilon}) = 2$ while $W_2(\rho_\varepsilon, \rho_{-\varepsilon}) = 2\varepsilon$. The size of the jump discontinuity is $\lim_{\varepsilon \to 0} W_2(\gamma_\varepsilon, \gamma_{-\varepsilon}) = 2$, which is precisely the Wasserstein diameter of the set of optimal plans between $\rho_0$ and $\mu$. This is the worst it can be, as a consequence of qualitative stability.
\end{example}

Problem \ref{problem: discrete plan stability} is related to the stability of the linear programs \eqref{eq: C(X, Y) linear programming problem} and \eqref{eq: C(X, Y) dual linear programming problem}, with respect to perturbations of $C(X, Y)$. However, the geometry of linear programming; i.e. $\Pi(\alpha, \beta) \subset \R^{M \times N}$ equipped with some norm, and the geometry of $(\mathcal{P}_2(\R^d \times \R^d), W_2)$ are very different. In particular, in \eqref{eq: C(X, Y) linear programming problem}, all the geometric information $X$ and $Y$ concerning the ambient space $\R^d$ is embedded directly into the cost vector $C(X, Y)$, and bears no relation to the geometry of the polyhedron $\Pi(\alpha, \beta)$. Consider the mapping
\begin{equation*}
    \gamma : (\R^d)^M \times (\R^d)^N \times \Pi(\alpha, \beta) \to \mathcal{P}_2(\R^d \times \R^d); \quad (X, Y, \hat\gamma) \mapsto \gamma(X, Y),
\end{equation*}
where we denote
\begin{equation*}
    \gamma(X, Y)  := \sum_{i=1}^M \sum_{j=1}^N \hat\gamma_{ij} \delta_{x_i, y_j} \in \Pi(\rho_X, \mu_Y).
\end{equation*}
First, assume we are comparing two different transport plans in $\Pi(\rho_X, \mu_Y)$. Two distinct couplings $\hat\gamma_0, \hat\gamma_1 \in \Pi(\alpha, \beta)$ have a fixed Euclidean distance $\| \hat \gamma_0 - \hat \gamma_1\|_2>0$ in the polyhedron, whilst $W_2(\hat\gamma_0(X, Y), \hat\gamma_1(X, Y))$ can be made arbitrarily large or small by varying $X$ and $Y$. We note that the $l^1$ norm on $\Pi(\alpha, \beta) \subset \R^{M \times N}$ corresponds to the total variation:
\begin{equation*}
    \|\hat \gamma_0 - \hat \gamma_1\|_1 = \| \gamma_0(X, Y) - \gamma_1(X, Y)\|_{TV}
\end{equation*}

In general, however, we want to compare two plans, one in $\Pi(\rho_{X}, \mu_Y)$ and one in $\Pi(\rho_{X'}, \mu_{Y'})$, for different $X, X'$ and $Y, Y'$. In this case, neither $W_2(\hat\gamma_0(X, Y), \hat\gamma_1(X', Y'))$ nor  $\|\gamma_0(X, Y) - \gamma_1(X', Y')\|_{TV}$ are represented by the geometry of $\Pi(\alpha, \beta)$. It is unclear, then, that the quantitative stability results in linear programming can be directly applied to the primal problem to deduce those of a functional analysis flavour. Regardless, the linear programming formulation will give us some qualitative insight into how optimal plans deform as we vary $X$ and $Y$ continuously.

It is natural to ask an analogue of Problem \ref{problem: discrete plan stability} for the stability of the set of dual optimisers $\Phi$. This is relevant both for independent interest as well as that it may hopefully allow better understanding of Problem \ref{problem: discrete plan stability}.
\begin{problem}
    \label{problem: discrete potential stability}
    How do $\Phi_c(X, Y)$ and $\Psi_c(X, Y)$ change as we perturb $X$ and $Y$? If the supports are close so that $\|X - X'\|, \|Y - Y'\| \ll 1$ (and hence $W_2(\rho_X, \rho_{X'})$ and $W_2(\mu_Y, \mu_{Y'})$ are small), are $\Phi_c(X, Y)$ and $\Phi_c(X', Y')$ quantitatively close in some sense? We would like results of the form
\begin{equation}
\label{eq: stability of potentials discrete goal statement}
    d_{H, \|\cdot\|_\phi}(\Phi_c(X, Y), \Phi_c(X', Y')) \leq C \Bigl(W_2(\rho_{X_0}, \rho_{X_1}) + W_2(\mu_{Y_0}, \mu_{Y_1}) \Bigr)^q + \mathcal{E}
\end{equation}
where $d_{H, \|\cdot\|_\phi}$ is the Hausdorff distance between sets on $\R^N$ with respect to a choice of norm $\| \cdot \|_\phi$, and $\mathcal{E}$ is an error term to be understood based on the structure of the measures.
\end{problem}

A direct comparison between vectors $\phi \in \Phi(X, Y)$ and $\phi' \in \Phi(X', Y')$ for some norm on $\R^M$ has no direct interpretation if $X \neq X'$. If, however, $X = X'$, then norm distances do have a useful interpretation. For example
\begin{equation*}
    \| \phi \|_{\R^M, \alpha} := \sum_{i=1}^M \alpha_i |\phi_i|^2,
\end{equation*}
so that if $X = X'$, then for $\phi \in \Phi(X, Y)$ and $\phi' \in \Phi(X, Y')$ we have
\begin{equation*}
    \|\phi - \phi'\|_{\R^M, \alpha} = \|\phi^{**} - \phi^{'**}\|_{L^2(\rho)},
\end{equation*}
where here $\phi^{**}: \X \to \R$ corresponds to the double Legendre transform with respect to $X$ then $Y$, so that $\phi^{**}$ restricted to $X$ returns the vector $\phi$. (This is only true when $\phi$ was already assumed to be optimal; in general, the restriction of the double Legendre transform is just the restriction of the largest convex function which is less than $\phi_i$ at each $x_i$, which may not return $\phi_i$ itself.) Thus, a norm on $\R^N$ corresponds directly to an $L^2(\rho)$ distance between continuous potentials which is amenable to qualitative limits.

\begin{problem}
    Use the understanding gained from Problem \ref{problem: discrete plan stability} and \ref{problem: discrete potential stability} to take qualitative limits of \eqref{eq: stability of plans discrete goal statement} and \eqref{eq: stability of potentials discrete goal statement} to say things about general transport problems beyond the discrete case. In this way, a sequence of linear programs which may a priori seem unrelated can correspond to weakly converging marginal sequences of discrete measures for which \eqref{eq: stability of plans discrete goal statement} may hopefully be stable.
\end{problem}

Given that the norm structure on $D(C)$ is much more useful than that of $\Pi(\alpha, \beta)$, to approach Problem \ref{problem: discrete plan stability} it seems reasonable to first understand component wise effects on $\phi \in \Phi(X, Y)$ of perturbing $Y$ through the explicit characterisation given in Proposition \ref{prop: dual set characterisation}. The $\phi_i$ values correspond to the translations of hyperplanes with gradient $x_i$, casting this over the target domain we could hope to understand how the Laguerre cells induced on the target domain evolve. Since there should be the same amount of mass in each cell, we could use this information to construct a coupling between two different transport plans with the same first marginal, giving an estimate on $W_2(\gamma_{X, Y}, \gamma_{X, Y'})$.

\section{Structure of primal optimisers under deformations}

In this section, we provide a qualitative explanation of how the set of optimisers $\Gamma_c(X, Y)$ defined in \eqref{eq: discrete transport problem X and Y} deforms under Perturbations of $X$ and $Y$. Intuitively, the idea is as follows: if we perturb the positions of each of the masses continuously such that at each point in time, the transport problem has \textbf{a unique solution}, then the structure of the optimal plan remains the same. For the structure of the optimal plan to change, it must be that at some point, one way of transporting becomes better than another, and at the point of this change, the before plan and after plan must have the same cost, contradicting uniqueness.

We first present a direct proof in the language of measure-theoretic optimal transport using the $c$-cyclical monotonicity \eqref{intro: c cm definition statment}. We then describe how this connects to uniqueness in linear programming through the characterisation of extreme points of $\Pi(\alpha, \beta)$, Lemma \ref{lemma: characterisation of extreme points of primal polyhedra}.

\begin{definition}[Glue-composition of transport plans]
    Let $\rho, \mu_0, \mu_1$ be probability measures on $\R^d$ and consider transport plans $\gamma_0 \in \Pi(\rho, \mu_0)$ and $\pi \in \Pi(\mu_0, \mu_1)$. Then $\gamma_1 \in \Pi(\rho, \mu_1)$ is said to be a \textit{glue-composition} of $\gamma_0$ and $\pi$ if there exists a measure $\eta \in \Pi(\rho, \mu_0, \mu_1)$ with $p_{12\#}\eta = \gamma_0$, $p_{23\#}\eta = \pi$ and $p_{13\#}\eta = \gamma_1$. Here $p_{ij}$ represents the canonical projection onto the $i$th and $j$th coordinates. 
    We denote by
    \begin{equation*}
        \gamma_0 \circ \pi \subset \Pi(\rho, \mu_1)
    \end{equation*}
    the set of all possible glue-compositions of $\gamma_0$ and $\pi$. There always exists at least one way to glue transport plans with a common marginal, since one can always take $\gamma_{0, y} \otimes \pi_y \di \mu_0(y)$ where $\gamma_{0, y}$ and $\pi_y$ are disintegrations of $\gamma_0$ and $\pi$ in the second/first coordinates respectively, see \cite[Lemma 5.5]{santambrogio2015optimal} for more details. Unlike the composition of maps, glue-compositions need not be unique, see Example \ref{ex: glueing nonuniqueness}.
\end{definition}

\begin{example}[Non-uniqueness of glueing]
\label{ex: glueing nonuniqueness}
Consider the transport problem illustrated in Figure \ref{fig:glueing-nonuniqueness} between $\rho_X$ and $\mu_{Y(t)}$ where
\begin{equation*}
    \alpha = (3/4, 1/4),\quad \beta = (1/2, 1/4, 1/4),
\end{equation*}
\begin{equation*}
    X = (A, B),\; Y_0 = (C, D, D),\; Y_1 = (E, F, G),\; Y(t) = (1-t)Y_0 + tY_1.
\end{equation*}
There are multiple glue-compositions of the optimal plan $\gamma \in \Pi(\rho_X, \mu_{Y_0})$ with the optimal plan $\pi_t \in \Pi(\mu_{Y_0}, \mu_{Y(t)})$. Only one of them is the optimal plan for quadratic cost between $\rho_X$ and $\mu_{Y(t)}$ (the middle diagram).
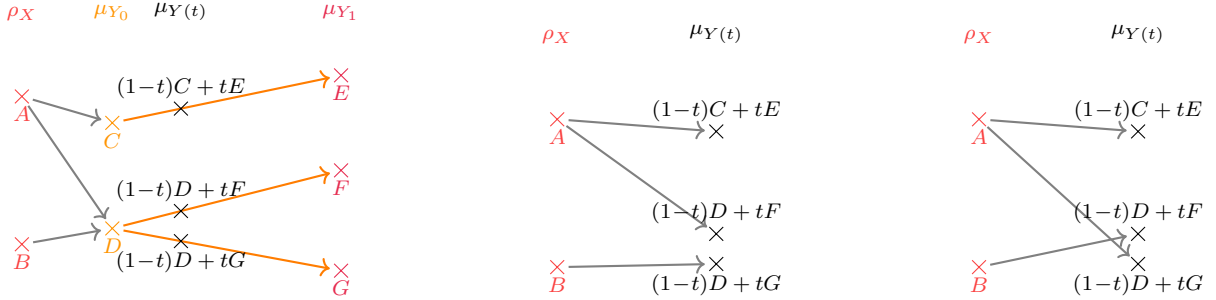
\begin{figure}[ht]
    \centering

\begin{tikzpicture}[scale=0.7]

    \draw (-3, 2.5) node[above=3pt]{\scriptsize\textcolor{sourcecolour}{$\rho_X$}};
    \draw (-3, 1.4) node[cross=3pt, sourcecolour]{};
    \draw (-3,1.4) node[anchor=north, sourcecolour]{\scriptsize$A$};
    \draw (-3,-1.4) node[cross=3pt, sourcecolour]{};
    \draw (-3,-1.4) node[anchor=north, sourcecolour]{\scriptsize$B$};

    \draw[->, mixed, thick, color=gray, shorten >=0.15cm, shorten <=0.15cm] (-3, 1.4) -- (-1.3,0.9);
    \draw[->, mixed, thick, color=gray, shorten >=0.15cm, shorten <=0.15cm] (-3,-1.4) -- (-1.3, -1.1);
    \draw[->, mixed, thick, color=gray, shorten >=0.15cm, shorten <=0.15cm] (-3, 1.4) -- (-1.3, -1.1);
    
    \draw (-1.3, 2.5) node[above=3pt]{\scriptsize\textcolor{amber}{$\mu_{Y_0}$}};
    \draw (-1.3,0.9) node[cross=3pt, amber]{};
    \draw (-1.3,0.9) node[anchor=north, amber]{\scriptsize$C$};
    \draw (-1.3, -1.1) node[cross=3pt, amber]{};
    \draw (-1.3, -1.1) node[anchor=north, amber]{\scriptsize$D$};

    \draw[->, mixed, thick, color=amber2, shorten >=0.15cm, shorten <=0.15cm] (-1.3,0.9) -- (3,1.8);
    \draw[->, mixed, thick, color=amber2, shorten >=0.15cm, shorten <=0.15cm] (-1.3, -1.1) -- (3, 0);
    \draw[->, mixed, thick, color=amber2, shorten >=0.15cm, shorten <=0.15cm] (-1.3, -1.1) -- (3, -1.9);

    \draw (-0.01, 2.5) node[above=3pt]{\scriptsize\textcolor{black}{$\mu_{Y(t)}$}};
    \draw (-0.01,1.17) node[cross=3pt, black]{};
    \draw (-0.01,1.17) node[anchor=south, black]{\scriptsize$(1\!-\!t)C + tE$};
    \draw (-0.01,-0.77) node[cross=3pt, black]{};
    \draw (-0.01,-0.77) node[anchor=south, black]{\scriptsize$(1\!-\!t)D + tF$};
    \draw (-0.01,-1.34) node[cross=3pt, black]{};
    \draw (-0.01,-1.34) node[anchor=north, black]{\scriptsize$(1\!-\!t)D + tG$};

    \draw (3, 2.5) node[above=3pt]{\scriptsize\textcolor{amaranth}{$\mu_{Y_1}$}};
    \draw (3,1.8) node[cross=3pt, amaranth]{};
    \draw (3,1.8) node[anchor=north, amaranth]{\scriptsize$E$};
    \draw (3, 0) node[cross=3pt, amaranth]{};
    \draw (3,0) node[anchor=north, amaranth]{\scriptsize$F$};
    \draw (3, -1.9) node[cross=3pt, amaranth]{};
    \draw (3,-1.9) node[anchor=north, amaranth]{\scriptsize$G$};
    \end{tikzpicture}
\hfill
\begin{tikzpicture}[scale=0.7]

    \draw (-3, 2.5) node[above=3pt]{\scriptsize\textcolor{sourcecolour}{$\rho_X$}};
    \draw (-3, 1.4) node[cross=3pt, sourcecolour]{};
    \draw (-3,1.4) node[anchor=north, sourcecolour]{\scriptsize$A$};
    \draw (-3,-1.4) node[cross=3pt, sourcecolour]{};
    \draw (-3,-1.4) node[anchor=north, sourcecolour]{\scriptsize$B$};

    \draw[->, mixed, thick, color=gray, shorten >=0.15cm, shorten <=0.15cm] (-3, 1.4) -- (-0.01, 1.17);
    \draw[->, mixed, thick, color=gray, shorten >=0.15cm, shorten <=0.15cm] (-3, 1.4) -- (-0.01, -0.77);
    \draw[->, mixed, thick, color=gray, shorten >=0.15cm, shorten <=0.15cm] (-3, -1.4) -- (-0.01, -1.34);

    \draw (-0.01, 2.5) node[above=3pt]{\scriptsize\textcolor{black}{$\mu_{Y(t)}$}};
    \draw (-0.01,1.17) node[cross=3pt, black]{};
    \draw (-0.01,1.17) node[anchor=south, black]{\scriptsize$(1\!-\!t)C + tE$};
    \draw (-0.01,-0.77) node[cross=3pt, black]{};
    \draw (-0.01,-0.77) node[anchor=south, black]{\scriptsize$(1\!-\!t)D + tF$};
    \draw (-0.01,-1.34) node[cross=3pt, black]{};
    \draw (-0.01,-1.34) node[anchor=north, black]{\scriptsize$(1\!-\!t)D + tG$};

\end{tikzpicture}
\hfill
\begin{tikzpicture}[scale=0.7]

    \draw (-3, 2.5) node[above=3pt]{\scriptsize\textcolor{sourcecolour}{$\rho_X$}};
    \draw (-3, 1.4) node[cross=3pt, sourcecolour]{};
    \draw (-3,1.4) node[anchor=north, sourcecolour]{\scriptsize$A$};
    \draw (-3,-1.4) node[cross=3pt, sourcecolour]{};
    \draw (-3,-1.4) node[anchor=north, sourcecolour]{\scriptsize$B$};


    \draw[->, mixed, thick, color=gray, shorten >=0.15cm, shorten <=0.15cm] (-3, 1.4) -- (-0.01, 1.17);
    \draw[->, mixed, thick, color=gray, shorten >=0.15cm, shorten <=0.15cm] (-3, 1.4) -- (-0.01, -1.34);
    \draw[->, mixed, thick, color=gray, shorten >=0.15cm, shorten <=0.15cm] (-3, -1.4) -- (-0.01, -0.77);

    \draw (-0.01, 2.5) node[above=3pt]{\scriptsize\textcolor{black}{$\mu_{Y(t)}$}};
    \draw (-0.01,1.17) node[cross=3pt, black]{};
    \draw (-0.01,1.17) node[anchor=south, black]{\scriptsize$(1\!-\!t)C + tE$};
    \draw (-0.01,-0.77) node[cross=3pt, black]{};
    \draw (-0.01,-0.77) node[anchor=south, black]{\scriptsize$(1\!-\!t)D + tF$};
    \draw (-0.01,-1.34) node[cross=3pt, black]{};
    \draw (-0.01,-1.34) node[anchor=north, black]{\scriptsize$(1\!-\!t)D + tG$};

\end{tikzpicture}

\caption{Measures $\rho_X$, $\mu_{Y_0}$, $\mu_{Y(t)}$ and $\mu_{Y_1}$ and couplings between them (left); and two potential glue-compositions of $\gamma_0$ and $\pi_t$ (middle and right).}
    \label{fig:glueing-nonuniqueness}
\end{figure}
\end{example}

\begin{theorem}[Structure of discrete plans along trajectories preserving uniqueness]
\label{thrm: structure of discrete plans new notation}
Let $c: \R^d \times \R^d \to \R_+$ be continuous. Let $X: [0, 1] \to (\R^d)^M$ and $Y : [0, 1] \to (\R^d)^N$ be continuous curves such that
\begin{enumerate}[label = (\roman*)]
    \item \label{enum: distinct points (0, 1)} For each $t \in (0, 1)$ and any distinct indices $i \neq j$, $x_i(t) \neq x_j(t)$ and $y_i(t) \neq y_j(t)$.
    \item \label{enum: unique solutions for each pb} For each $t \in [0, 1]$ solutions to the $c$-cost transport between $\rho_{X(t)}$ and $\mu_{Y(t)}$ are unique.
\end{enumerate}
Then there exists $\hat{\gamma} \in \Pi(\alpha, \beta)$ such that
\begin{equation*}
    \Gamma_c(X(t), Y(t)) = \left\{ \sum_{i = 1}^{M}\sum_{j=1}^N \hat\gamma_{ij} \delta_{x_i(t), y_j(t)} \right\} \text{ for all } t \in [0, 1],
\end{equation*}
so that the ``same" plan is the unique optimal plan for all $t \in [0, 1]$. In other words, given the optimal plan for some $s \in (0, 1)$, the unique optimal plan for each other $t \in [0, 1]$ is a glue-composition (note there can be multiple glue compositions here, but only one can be optimal since we assume uniqueness) of $ \pi_{ts}^\rho \circ \gamma(s) \circ \pi_{st}^\mu$ where
\begin{equation*}
    \pi_{ts}^\rho := \sum_{i = 1}^M \alpha_i \delta_{x_i(t), x_i(s)} \in \Pi(\rho_{X(t)}, \rho_{X(s)})\; \text{ and } \; \pi_{st}^\mu :=\sum_{i = 1}^M \beta_i \delta_{y_i(s), y_i(t)} \in \Pi(\mu_{Y(s)}, \mu_{Y(t)}).
\end{equation*}
\end{theorem}
\begin{proof}
Fix some $s \in (0, 1)$ and let $\gamma(s)$ be the unique optimal plan between $\rho_{X(s)}$ and $\mu_{Y(s)}$ for cost $c$, written in the form
\begin{equation}
\label{eq: gamma(s) definion}
    \gamma(s) = \sum_{i = 1}^{M}\sum_{j=1}^N \hat\gamma_{ij}(s) \delta_{x_i(s), y_j(s)}. \quad \text{ (Here some $\hat\gamma_{ij}(s)$ may be zero.)}
\end{equation}
Here a priori, the $\hat\gamma_{ij}(s)$ \textit{depend on the choice of $s$}. We want to show that they are independent of $s$, so that the same $\gamma(t)$ is optimal for all  $t \in [0, 1]$, rather than just $t = s$. The ways to choose a finite number of ways to choose distinct $\hat P \subset \hat\gamma(s)$ and distinct $P \subset \gamma(t)$ are in bijection, and there are a finite number of possibilities $r$. We denote these by $\{P_k(t)\}_{k=1}^r$, where each is of the form
\begin{equation*}
    P_k(t) = \{(x_{i_l}(t), y_{j_l}(t))\}_{l=1}^{L_k} \quad \text{ with }\quad x_{i_{l}}(t) \neq x_{i_{l'}}(t), y_{j_l}(t) \neq y_{j_{l'}}(t)\; \forall l \neq l', t \in (0, 1).
\end{equation*}
By Proposition \ref{prop: uniqueness iff strict cm}, $\Delta_{P_k(s)} >0$, and since $c$, $X(t)$ and $Y(t)$ are continuous, so is $t \mapsto \Delta_{P_k(t)}$. As there are a finite number of distinct support selections,
\begin{equation*}
    \underline{\Delta}(t) : = \inf_{k = 1,..., r} \Delta_{P_k(t)}
\end{equation*}
is continuous and $\underline{\Delta}(s)>0$. Set
\begin{equation*}
    t^+ = \inf \left\{ t \geq s : \underline{\Delta}(t) = 0 \right\} \quad \text{ and } t^- = \sup\left\{t \leq s : \underline{\Delta}(t) = 0\right\}.
\end{equation*}
Assume $t^- \in (0, s)$, then there exists $P \subset \supp \gamma(t^-)$ such that $\Delta_P = 0$, and by continuity and definition of $t^-$ as a supremum, $\Delta_{P'} \geq 0$ for all other $P' \subset \supp \gamma(t^-)$. Proposition \ref{prop: uniqueness iff strict cm} implies both that $\gamma(t^-)$ is optimal and that it is not the only optimiser, contradicting the non-uniqueness assumption. Hence, it must be that $\underline{\Delta}(t) > 0$ for all $t \in (0,s)$. Arguing symmetrically for $t^+$, we deduce that $\underline{\Delta}(t) > 0$ for all $t \in (0,1)$, and hence is the unique optimiser. This extends to $t \in [0, 1]$ by qualitative stability.

To explicit the glueing structure, for $s \in (0, 1)$ and $t \in [0, 1]$ consider the measure
\begin{equation}
\label{eq: G(s, t) definition glueing measure}
    \eta(s,t) := \sum_{i=1}^M \sum_{j = 1}^N \hat\gamma_{ij} \delta_{x_i(t), x_i(s), y_j(s), y_j(t)} \in \Pi(\rho_{X(t)}, \rho_{X(s)}, \mu_{Y(s)}, \mu_{Y(t)}).
\end{equation}
It follows that
\begin{equation*}
    p_{12\#}\eta(s, t) = \pi^\rho_{ts};\quad p_{23\#} \eta(s, t) = \gamma(s); \quad \text{ and } p_{34\#} \eta(s, t) = \pi^\mu_{st}
\end{equation*}
and
\begin{equation*}
    p_{14\#} \eta(s, t) = \gamma(t) \quad\forall t \in [0, 1].
\end{equation*}
Thus, we have shown that the optimal plan over $\Pi(\rho_t, \mu_t)$ is given by glue composition, supported by the measure $\eta(s, t)$.
\end{proof}

\begin{remark}
The connection between Theorem \ref{thrm: structure of discrete plans new notation} and the linear programming formulation is made clear by the characterisation of extreme points of $\Pi(\alpha, \beta)$, Lemma \ref{lemma: characterisation of extreme points of primal polyhedra}. Since the weight vectors $\alpha \in \R^N$ and $\beta \in \R^M$ are fixed, varying $X$ and $Y$ corresponds to varying the direction of the cost vector $C(X, Y)$ inside $\Pi(\alpha, \beta)$. So long as we move $C(X, Y)$ in a way which preserves non-orthogonality to the face of $\Pi(\alpha, \beta)$ in the direction $-C$, we retain uniqueness, and the same vertex (and hence ``same" plan) is optimal. If at some point $-C(X, Y)$ is orthogonal to a face, this entire face will be optimal (non-uniqueness). One of the vertices on this face must necessarily be the unique optimal vertex from which we arrived (this is qualitative stability/upper semicontinuity). When we are on the interior of this face, the graph $G(\hat\gamma)$ has cycles, which correspond to multiple possible permutations of the support being optimal (we can construct multiple optimal plans à la $c$-cyclical monotonicity).
One cannot perturb $C$ continuously such that we have uniqueness all along the trajectory, with a different vertex $\hat\gamma_1 \in \Pi(\alpha, \beta)$ optimal at the end compared to $\hat\gamma_0$ at the beginning: If
\begin{equation*}
    \langle \hat \gamma_0 | C(0) \rangle > \langle\hat\gamma_1 |C(0) \rangle \quad \text{ and } \quad \langle \hat \gamma_1 | C(1) \rangle > \langle\hat\gamma_0 |C(1) \rangle,
\end{equation*}
then by continuity of $C(t)$, the intermediate value theorem tells us there exists $t \in [0, 1]$ with
\begin{equation*}
    \langle \hat \gamma_0 | C(t) \rangle = \langle \hat \gamma_1 | C(t) \rangle \implies \langle \langle \hat \gamma_0 - \hat\gamma_1 | C(t) \rangle =0.
\end{equation*}
Thus $C(t)$ is orthogonal to the $1$ dimensional edge connecting $\hat\gamma_0$ and $\hat \gamma_1$.

There is, however, no direct equivalence between the uniqueness of the measure-theoretic and linear programming formulations. If for some $i_0, i_1$ or $j_0, j_1$ we have $x_{i_0} = x_{i_1}$ or $y_{j_0} = y_{j_1}$, then we have non-uniqueness of the Linear program \eqref{eq: C(X, Y) linear programming problem} whilst the measure theoretic problem \eqref{eq: discrete transport problem X and Y} could retain uniqueness. Thus, we can pass from one vertex to another being optimal in $\Pi(\alpha, \beta)$ whilst $\Gamma(t)$ remains a singleton. This is illustrated in the example below.
\end{remark}
\begin{example}[Intersecting particles]
For interpolations $X(t) = (1-t) X_0 + t X_1$ and $Y(t) = (1-t) Y_0 + t Y_1$ that correspond to quadratic geodesic interpolations in the sense of Remark \ref{rmk: quadratic interpolants particles dont cross}, the non-intersecting particle assumption \ref{enum: distinct points (0, 1)} is guaranteed. Without this assumption, consider
\begin{equation*}
    \alpha = \beta = (1/2, 1/2),
\end{equation*}
\begin{equation*}
    X= (A, B), Y_0 = (C, D), Y_1 = (F, E),
\end{equation*}
as illustrated in Figure \ref{fig:intersecting transport rays}. We retain uniqueness all along the trajectory, but after $t=1/2$ when $y_1(t) = y_2(t)$, the glueing no longer corresponds to the optimal coupling.

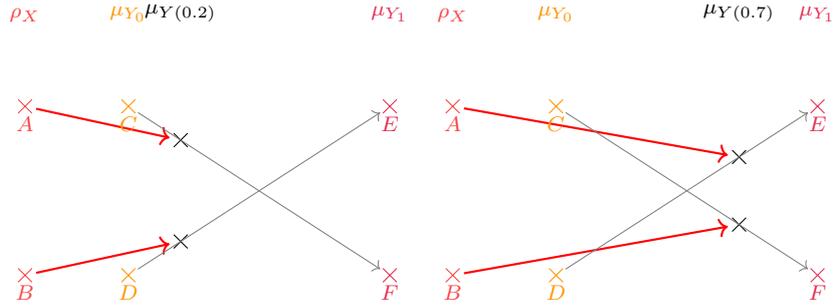
\begin{figure}[ht]
    \centering

\begin{tikzpicture}[scale=0.8]

    \draw (-3, 2.5) node[above=3pt]{\scriptsize\textcolor{sourcecolour}{$\rho_X$}};
    \draw (-3, 1.4) node[cross=3pt, sourcecolour]{};
    \draw (-3,1.4) node[anchor=north, sourcecolour]{\scriptsize$A$};
    \draw (-3,-1.4) node[cross=3pt, sourcecolour]{};
    \draw (-3,-1.4) node[anchor=north, sourcecolour]{\scriptsize$B$};
    
    \draw[->, mixed, thick, color=red, shorten >=0.15cm, shorten <=0.15cm] (-3,-1.4) -- (-0.44, -0.84);
    \draw[->, mixed, thick, color=red, shorten >=0.15cm, shorten <=0.15cm] (-3, 1.4) -- (-0.44,0.84);
    
    \draw (-0.44, 2.5) node[above=3pt]{\scriptsize\textcolor{black}{$\mu_{Y({0.2})}$}};
    \draw (-0.44,0.84) node[cross=3pt, black]{};
    \draw (-0.44, -0.84) node[cross=3pt, black]{};

    \draw (-1.3, 2.5) node[above=3pt]{\scriptsize\textcolor{amber}{$\mu_{Y_0}$}};
    \draw (-1.3,1.4) node[cross=3pt, amber]{};
    \draw (-1.3,1.4) node[anchor=north, amber]{\scriptsize$C$};
    \draw (-1.3, -1.4) node[cross=3pt, amber]{};
    \draw (-1.3, -1.4) node[anchor=north, amber]{\scriptsize$D$};

    
    \draw[->, mixed, thin, color=gray, shorten >=0.15cm, shorten <=0.15cm] (-1.3, 1.4) -- (3, -1.4); 
    \draw[->, mixed, thin, color=gray, shorten >=0.15cm, shorten <=0.15cm] (-1.3, -1.4) -- (3, 1.4); 

    \draw (3, 2.5) node[above=3pt]{\scriptsize\textcolor{amaranth}{$\mu_{Y_1}$}};
    \draw (3,1.4) node[cross=3pt, amaranth]{};
    \draw (3,1.4) node[anchor=north, amaranth]{\scriptsize$E$};
    \draw (3, -1.4) node[cross=3pt, amaranth]{};
    \draw (3,-1.4) node[anchor=north, amaranth]{\scriptsize$F$};
\end{tikzpicture}
\begin{tikzpicture}[scale=0.8]

    \draw (-3, 2.5) node[above=3pt]{\scriptsize\textcolor{sourcecolour}{$\rho_X$}};
    \draw (-3, 1.4) node[cross=3pt, sourcecolour]{};
    \draw (-3,1.4) node[anchor=north, sourcecolour]{\scriptsize$A$};
    \draw (-3,-1.4) node[cross=3pt, sourcecolour]{};
    \draw (-3,-1.4) node[anchor=north, sourcecolour]{\scriptsize$B$};

    \draw[->, mixed, thick, color=red, shorten >=0.15cm, shorten <=0.15cm] (-3,-1.4) -- (1.71, -0.56);
    \draw[->, mixed, thick, color=red, shorten >=0.15cm, shorten <=0.15cm] (-3, 1.4) -- (1.71,0.56);
    
    \draw (1.71, 2.5) node[above=3pt]{\scriptsize\textcolor{black}{$\mu_{Y({0.7})}$}};
    \draw (1.71,0.56) node[cross=3pt, black]{};
    \draw (1.71, -0.56) node[cross=3pt, black]{};

    \draw (-1.3, 2.5) node[above=3pt]{\scriptsize\textcolor{amber}{$\mu_{Y_0}$}};
    \draw (-1.3,1.4) node[cross=3pt, amber]{};
    \draw (-1.3,1.4) node[anchor=north, amber]{\scriptsize$C$};
    \draw (-1.3, -1.4) node[cross=3pt, amber]{};
    \draw (-1.3, -1.4) node[anchor=north, amber]{\scriptsize$D$};

    
    \draw[->, mixed, thin, color=gray, shorten >=0.15cm, shorten <=0.15cm] (-1.3, 1.4) -- (3, -1.4); 
    \draw[->, mixed, thin, color=gray, shorten >=0.15cm, shorten <=0.15cm] (-1.3, -1.4) -- (3, 1.4); 

    \draw (3, 2.5) node[above=3pt]{\scriptsize\textcolor{amaranth}{$\mu_{Y_1}$}};
    \draw (3,1.4) node[cross=3pt, amaranth]{};
    \draw (3,1.4) node[anchor=north, amaranth]{\scriptsize$E$};
    \draw (3, -1.4) node[cross=3pt, amaranth]{};
    \draw (3,-1.4) node[anchor=north, amaranth]{\scriptsize$F$};
\end{tikzpicture}

    \caption{Optimal couplings between $\rho_X$ and $\mu_{Y(t)}$ for $t=0.2$ and $t=0.7$.}
    \label{fig:intersecting transport rays}
\end{figure}
\end{example}

\begin{lemma}
    Under the hypotheses of Theorem \ref{thrm: structure of discrete plans new notation}, let $X(t)$ and $Y(t)$ be of the form $X(t) = (1-t)X_1 + t X_2$, $Y(t) = (1-t)Y_1 + t Y_2$ for $X_i \in (\R^d)^M$ and $Y_i \in (\R^d)^N$. Then there exists a glueing measure $\eta(0, 1) \in \Pi(\rho_{X_1}, \rho_{X_0}, \mu_{Y_0}, \mu_{Y_1})$ which describes all optimal plans between $\rho_{X(t)}$ and $\mu_{Y(t)}$ simultaneously.
\end{lemma}
\begin{proof}
    Recalling \eqref{eq: G(s, t) definition glueing measure} which defines $\eta(s, 1)$ for $s \in (0, 1)$, we set
    \begin{equation*}
        \eta(0, 1) : = \lim_{s \to 0^+} \eta(s, 1) \in \Pi(\rho_{X_1}, \rho_{X_0}, \mu_{Y_0}, \mu_{Y_1})
    \end{equation*}
    where the limit is in the sense of weak convergence since $X(s) \to X_0$ and $Y(s) \to Y_0$ are continuous. Explicitly, \eqref{eq: G(s, t) definition glueing measure} tells us that 
    \begin{equation*}
        \eta(0, 1) = \sum_{i=1}^M \sum_{j = 1}^N \hat\gamma_{ij} \delta_{x_i^1, x_i^0, y_j^0, y_j^1} \in \Pi(\rho_{X(t)}, \rho_{X(s)}, \mu_{Y(s)}, \mu_{Y(t)})
    \end{equation*}
    where $\hat \gamma \in \Pi(\alpha, \beta)$ is the vertex corresponding to optimality for $s \in (0, 1)$ in Theorem \ref{thrm: structure of discrete plans new notation}. Then since $x_i(t) = (1-t)x_i^0 + t x_i^1$ and similarly for $Y(t)$, the optimal plan $\gamma(t) \in \Gamma(X(t), Y(t))$ is
    \begin{equation*}
        \gamma(t) = ( tx^1 + (1-t)x^0, (1-t)y^0 + t y^1)_\# \eta(0, 1) \; \text{ for all } t \in [0, 1].
    \end{equation*}
\end{proof}

\begin{remark}
    Another way of formulating the above is as follows. For $X(t), Y(t) : [0, 1] \to (\R^d)^M \times (\R^d)^N$ continuous, the measures $\rho_{X(t)}, \mu_{Y(t)}$ are absolutely continuous curves in $(\mathcal{P}_p(\R^d),W_p)$ for $p>1$. By \cite[Chapter 5]{santambrogio2015optimal} there exist vector fields $v^X_t, v_t^Y$ convecting these curves, in that
    \begin{equation*}
        \partial_t \rho_{X(t)} + \nabla \cdot (v_t^X\rho_{X(t)}) = 0, \quad \partial_t \mu_{Y(t)} + \nabla \cdot (v_t^Y\mu_{Y(t)}) = 0.
    \end{equation*}
    (Here one can explicitly just take vector fields agreeing a.e. with the coordinates of $X'(t)$ and $Y'(t)$.)
    Under the Hypotheses of Theorem \ref{thrm: structure of discrete plans new notation}, for $\gamma_t$ the unique optimal plan, we have
    \begin{equation*}
        \partial_t \gamma_t + \nabla \cdot ((v_t^X, v_t^Y) \gamma_t) = 0.
    \end{equation*}
\end{remark}

\begin{remark}
\label{rmk: quadratic interpolants particles dont cross}
    The reason why we must define $\eta(0, 1)$ delicately is that in \ref{enum: distinct points (0, 1)} we do not assume the coordinates are distinct at $s = 0$. Thus, there may be multiple ways to glue, only one of which will give the optimal plan by our uniqueness assumption. Taking the limit down from $s>0$ ensures we select the right glueing by qualitative stability.  The reason we prove the result with these slightly more general hypotheses (particles can intersect at $t= 0, 1$) is explained below.

    Take couplings in $ \pi^\rho \in \Pi(\rho_{X_0}, \rho_{X_1})$ and $\pi^\mu \in \Pi(\rho_{Y_0}, \rho_{Y_1})$. Let $M = \# \supp \pi^\rho$ and $N = \#\supp \pi^\mu$, then let each coordinate of $X_0, X_1$ correspond to a point in in $\supp \pi^\rho$ with 
    \begin{equation*}
        \alpha_i = \pi^\rho((x^0_i, x^1_i)) \quad \text{ and } \quad \beta_j = \pi^\mu((y_j^0, y_j^1))
    \end{equation*}
    for each distinct pair of points $x^0, x^1 \in \supp \pi^\rho$ and $y^0, y^1 \in \supp \pi^\mu$.
    If $\pi^\rho$ and $\pi^\mu$ are chosen to be optimal for $p$ cost with $p>1$, then the interpolations $X(t) = (1-t) X_0 + t X_1$ and $Y(t) = (1-t) Y_0 + t Y_1$ correspond to $W_p$ geodesic interpolations, and in fact all $W_p$ geodesics are of this form, see \cite[Chapter 5]{santambrogio2015optimal}.
    \begin{equation*}
    \rho_t = ((1-t)x_0 + tx_1)_\#\pi^\rho \text{ and } \mu_t = ((1-t)y_0 + ty_1)_\#\pi^\mu.
    \end{equation*}
    Thus, we retain the slight generality of allowing particles to intersect at $t \in \{0, 1\}$ in the statement of Theorem \ref{thrm: structure of discrete plans new notation}: in general transport plans will split mass on departure and arrival, and we would like these interpolations to be covered by the theorem.
\end{remark}

\begin{lemma}
\label{lem: p cost geodesic gives non crossing}
    Let $c(x,y) = \|x-y\|^p$ with $p>1$.
    Assume that $X(t) = (1-t) X_0 + t X_1$ corresponds to a $W_p$ geodesic in the sense of Remark \ref{rmk: quadratic interpolants particles dont cross}. Then $X(t)$ satisfies the non-crossing particle assumption \ref{enum: distinct points (0, 1)}.
\end{lemma}
\begin{proof}
    This is a straightforward consequence of $ c$-monotonicity and strict convexity, following an argument very similar in flavour to \cite[Theorem 2.9]{santambrogio2015optimal}. Assume that $(x_0, x_1), (z_0, z_1) \in \supp \pi^\rho$, where $\pi^\rho \in \Pi(\rho_{X_0}, \rho_{X_1})$ is the $p$ cost optimal plan along which $X(t)$ is supported, in the sense of Remark \ref{rmk: quadratic interpolants particles dont cross}. Let 
    \begin{equation*}
        x_t = (1-t) x_0 + t x_1 = x_0 + t v_x \quad \text{ and } \quad y_t = (1-t) z_0 + t z_1 = z_0 + t v_z,
    \end{equation*}
    where $v_x = x_1 - x_0$ and $v_z = z_1 - z_0$. By $p$-cost cyclical monotonicity,
    \begin{equation}
    \label{eq: p monotony two points}
        \|x_0 - x_1\|^p + \|z_0 - z_1\|^p \leq \|z_1 - x_0\|^p + \|z_0 - x_1\|^p.
    \end{equation}
    Assume there exists $t \in (0, 1)$ such that
    \begin{equation}
    \label{eq: intersection at some t}
        x(t) = (1-t)x_0 + t x_1 = (1-t) z_0 + t z_1 = z(t).
    \end{equation}
    Then for this $t$,
    \begin{equation*}
        z_1 - x_0 = z_1 - z(t) + x(t) - x_0 = (1-t)v_z + t v_x.
    \end{equation*}
    and similarly $z_0 - x_1 = t v_z + (1-t) v_x$. If $v_z = v_x$, then \eqref{eq: intersection at some t} forces that $x_0 = z_0$ and $x_1 = z_1$. Thus, assuming $v_z \neq v_x$, By strict convexity of $\| \cdot ||^p$,
    \begin{align}
        \|(1-t)v_z + t v_x\|^p + \|t v_z + (1-t) v_x\|^p <& (1-t)\|v_z\|^p + t \|v_x\|^p + t\|v_z\|^p + (1-t) \|v_x\|^p\\
        =& \|x_0 - x_1\|^p + \|z_0 - z_1\|^p.
    \end{align}
    Combining the strict inequality above with \eqref{eq: p monotony two points} gives a contradiction, and so there can exist no such $t \in (0, 1)$ intersection point.
\end{proof}

The qualitative structural result Theorem \ref{thrm: structure of discrete plans new notation} allows for the following quantitative result.

\begin{theorem}[Stability of plans under uniqueness]
\label{thrm: stability of plans under uniqueness}
For the $p$ cost with $1<p<\infty$, let $\rho_0, \rho_1, \mu_0$ and $\mu_1$ be discrete measures in $\R^d$. Let $\gamma_0 \in \Pi(\rho_0, \mu_0)$ and $\gamma_1 \in \Pi(\rho_1, \mu_1)$ be optimal plans. Set $(\rho_t)_{t \in [0, 1]}$ and $(\mu_t)_{t \in [0, 1]}$ as $W_p$ geodesics connecting $\rho_0$ to $\rho_1$ and $\mu_0$ to $\mu_1$. Assume for each $t \in [0, 1]$ that the $W_p$ transport problem over $\Pi(\rho_t, \mu_t)$ has a unique solution. Then
\begin{equation*}
    W_p^p(\gamma_0, \gamma_1) \leq C_p (W_p^p(\rho_0, \rho_1) + W_p^p(\mu_0, \mu_1)).
\end{equation*}
where $C_p = \max(1, 2^{\frac{p}{2}-1})$.
\end{theorem}
\begin{proof}
Let $\pi^\rho$, $\pi^\mu$ be the plans supporting the geodesics $\rho_t$ and $\mu_t$, in the sense of Remark \ref{rmk: quadratic interpolants particles dont cross}. Since $\pi^\rho, \pi^\mu$ are optimal, the non-intersecting particle assumption \ref{enum: distinct points (0, 1)} is satisfied due to Lemma \ref{lem: p cost geodesic gives non crossing}, so apply Theorem \ref{thrm: structure of discrete plans new notation}. Hence $\gamma_1$ has the structure
\begin{equation*}
    \gamma_1 = \pi^\rho \circ \gamma_0 \circ \pi^\mu.
\end{equation*}
Let $\eta \in \Pi(\rho_1, \rho_0, \mu_0, \mu_1)$ be the corresponding glueing measure, so that $p_{12\#}\eta = \pi^\rho$, $p_{23\#}\eta = \gamma_0$, $p_{34\#}\eta = \pi^\mu$ and $p_{14\#}= \gamma_1$. Then $\eta$ defines a coupling between $\gamma_0$ and $\gamma_1$ so that
\begin{align*}
    W_p^p(\gamma_0, \gamma_1) \leq& \int_{\R^{4d}} \|(x_0, y_0)-(x_1, y_1)\|^p \di \eta(x_1, x_0, y_0, y_1)\\
    \leq& C_p \int_{\R^{4d}} \|x_0 -x_1\|^p + \|y_0 - y_1\|^p \di \eta(x_1, x_0, y_0, y_1)\\
    =& C_p\left(\int_{\R^{2d}} \|x_1 - x_0\|^p \di \pi^\rho(x_1, x_0) + \int_{\R^{2d}} \|y_0 - y_1\|^p \di \pi^\mu(y_0, y_1)\right)\\
    =& C_p(W^p_p(\rho_1, \rho_0) + W^p_p(\mu_0, \mu_1)).
\end{align*}
\end{proof}
\begin{remark}
The distances between plans and between marginals are actually comparable under these assumptions, since we always have the reverse inequality as a consequence of Proposition \ref{prop:reverse stability}. In the case $p=2$, $c_p = C_p = 1$ so that $W_2^2(\gamma_0, \gamma_1) = W_2^2(\rho_0, \rho_1) + W_2^2(\mu_0, \mu_1)$.
\end{remark}

\chapter{Uniqueness of continuous Kantorovich potentials}
\label{ch: uniqueness of dual problem}

In this chapter, we study the uniqueness properties of the continuous dual problem (\ref{intro: c dual problem}). In general, optimisers to (\ref{intro: c dual problem})  or (\ref{eq: C(X, Y) dual linear programming problem}) are at best unique up to a constant, since given any $(\phi, \psi)$ admissible, $(\phi + \lambda, \psi - \lambda)$ is also admissible for any $\lambda \in \R$ and has the same cost. We provide what is, to our knowledge, the first dual uniqueness criterion applicable when both measures are concentrated on lower-dimensional subsets of $\R^d$. We then use this characterisation to give a combinatorial, graph-based characterisation of uniqueness akin to Lemma \ref{lemma: balinski dual extreme vertices characterisation} and Proposition \ref{prop: acciaio discrete dual uniqueness}.

\section{Literature review}

We begin by reviewing the literature on the uniqueness of the dual problem. As discussed in Chapter \ref{ch: discrete transport problem}, the uniqueness of the discrete primal and dual problems is understood by characterising the extreme points of the respective polyhedra. Lemma \ref{lemma: characterisation of extreme points of primal polyhedra} and Lemma \ref{lemma: balinski dual extreme vertices characterisation} describe these in terms of a bipartite graph between source and target support points due to \cite{balinski1993signature, acciaio2025characterization}. Uniqueness of the primal problem occurs if and only if the corresponding graph $G_\Gamma$ has no cycles, and uniqueness of the dual occurs if and only if $G_\Gamma$ is connected. We also mention \cite{nutz2025quadratically}, which arrives at an essentially equivalent graph theoretic criterion for dual potentials in the quadratically regularised discrete problem, as well as giving a quantitative description akin to \cite[Theorem 3.9]{acciaio2025characterization}. 

We now turn to the uniqueness of the continuous dual problem (\ref{intro: c dual problem}). The following is the classically dual uniqueness criterion.
\begin{proposition}\textnormal{\cite[Proposition 7.18]{santambrogio2015optimal}}
\label{prop: int closure uniqueness santamb}
    Let $\rho$ and $\mu$ be probability measures supported on compact subsets $\X$ and $\Y$ of $\R^d$, let $c \in C^1(\X \times \Y)$. Assume that $\supp \rho$ is the closure of a connected open set. Then optimisers to the dual problem (\ref{intro: c dual problem}) are unique up to a constant.
\end{proposition}
Two recent papers \cite{staudt2025uniqueness, yang2023optimal} have approached the uniqueness of the continuous dual problem. Neither explicitly states a bipartite graph theoretic formulation, but both give results which are of this flavour. We also mention \cite{nutz2025quadratically}, which establishes an analogue of Proposition \ref{prop: int closure uniqueness santamb} in the quadratically regularised case. In \cite[Theorem 4.5]{yang2023optimal}, the authors characterise uniqueness under the assumption that $\supp \rho$ is the closure of a connected open set. Here they use what amounts to a bipartite graph-like connectedness condition à la Proposition \ref{prop: acciaio discrete dual uniqueness}, which is asymmetric in $\rho$ and $\mu$. One set of nodes is chosen to be the interiors of the connected components of $\rho$, while the other is simply the connected components of $\mu$. In \cite{staudt2025uniqueness}, the authors also consider the decomposition of $\supp \rho$ and $\supp \mu$ into connected components, presenting results under the assumption that the restriction of Kantorovich potentials to each component is unique, as well as presenting some results under the \textit{``closure of a connected open set"} support hypothesis.

\section{Uniqueness}

We begin by extending Proposition \ref{prop: int closure uniqueness santamb}. The assumption \textit{``closure of a connected open set"} is inherently ``full dimensional" in that each point in the interior should contain an open neighbourhood, so $\supp \rho$ cannot have any lower dimensional structure. Our result provides what is to our knowledge, the first dual uniqueness result where \textit{both} measures are supported on sets of dimension strictly smaller than the ambient space. We will require the following assumptions on the cost
\begin{assumption}
\hfill
\label{assumption: cost hypotheses}
\begin{enumerate}[label=(\roman*)]
    \item $c$ is continuous and bounded below.
    \item For fixed $y \in \Y$, $x \mapsto c(x, y)$ is differentiable for all $x \in \X$.
    \item The $y$ parametrised family of functions $x \mapsto c(x, y)$ are equilocally Lipschitz - given a compact $K \subset \Y$ we can choose a Lipschitz constant $L_K$ which works for all $y \in K$.
\end{enumerate} 
\end{assumption}
In particular, $c(x, y) = \|x - y\|^p$ satisfies Assumption \ref{assumption: cost hypotheses} for all $p >1$ (for $p \in (0, 1]$, differentiability of $x \mapsto \| x- y\|^p$ fails at $x = y$). Continuity of $c$ forces that the maps $x \mapsto c(x, y)$ are locally bounded in $x$ and $y$. These hypotheses induced the following regularity on $c$-concave (and hence optimal) potentials. We leave the proof to the Appendix, Lemma \ref{app: potential regularity}.
\begin{lemma}
\label{lem: potential regularity}
    Assume $c$ satisfies Assumption \ref{assumption: cost hypotheses}, and $\Y$ is compact. Then given any $\psi \in C_b(\Y)$, $\psi^c$ defined as in \eqref{intro: c transf definition} is Locally Lipschitz on $\X$, and does not take the value $-\infty$ anywhere.
\end{lemma}

\begin{definition}
    We say a set $\Omega$ is Lipschitz-path connected if for any $x_0, x_1 \in \Omega$, there exists a finite length curve $\omega : [0, 1] \to \Omega$ such that $\omega(0)=x_0$ and $\omega(1)=x_1$. Without loss of generality, we take $\omega$ to be parametrised with constant speed, hence Lipschitz.
\end{definition}

\begin{figure}[ht]
    \centering
    \begin{tikzpicture}[scale=1.2]

  \draw[fill=gray!20, draw=black, line width=0.9pt] (0,0) circle (1);
  \draw[fill=gray!20, draw=black, line width=0.9pt] (4,0) circle (1);

  \draw[draw=black, line width=1.2pt] (1,0) -- (3, 0);

  \fill (0,0.3) circle (2pt) node[above] {$x_0$};
  \fill (4,-0.3) circle (2pt) node[below] {$x_1$};

  \draw[red, dotted, thick, line width=1.5pt] (0,0.3) -- (1, 0);
  \draw[red, dotted, thick, line width=1.5pt] (1,0) -- (3, 0);
  \draw[red, dotted, thick, line width=1.5pt] (4,-0.3) -- (3, 0);

\end{tikzpicture}
    \hfill
    \begin{tikzpicture}[scale=1.2]

  \draw[fill=gray!20, draw=black, line width=0.9pt] (0,0) circle (1);
  \draw[fill=gray!20, draw=black, line width=0.9pt] (4,0) circle (1);

  \draw[draw=black, line width=1.2pt] (1,0) .. controls (2,1.5) .. (3,0);

  \fill (0,0.3) circle (2pt) node[above] {$x_0$};
  \fill (4,-0.3) circle (2pt) node[below] {$x_1$};

  \draw[red, dotted, thick, line width=1.5pt] (0,0.3) -- (1,0);
  \draw[red, dotted, thick, line width=1.5pt] (1,0) .. controls (2,1.5) .. (3,0);
  \draw[red, dotted, thick, line width=1.5pt] (3,0) -- (4,-0.3);

\end{tikzpicture}

    \caption{Two Lipschitz path-connected sets. Their interiors are not connected, so known uniqueness results do not apply to measures with this support. Uniqueness with bounded target, Theorem \ref{thrm: uniqueness Lipschitz connected bounded} applies to both domains, whilst for unbounded target, Theorem \ref{thrm: unbounded target uniqueness} only applies to the left picture.}
    \label{fig:Lipschitz Path connected domain}
\end{figure}
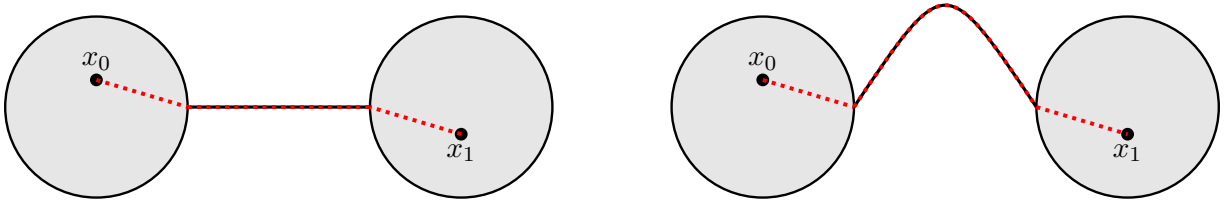

In particular, Lipschitz-path connected spaces include many lower-dimensional subspaces of $\R^d$ which are not covered by the closure of the interior style assumption, which is inherently $d$-dimensional. Measures with such support, include, for example, measures with densities with respect to some Hausdorff measure $\mathcal{H}^k$, restricted to some submanifold of $\R^d$.
\begin{theorem}
\label{thrm: uniqueness Lipschitz connected bounded}
Let $c: \R^d \times \R^d \to \R_+$ satisfy Assumption \ref{assumption: cost hypotheses}, let $\rho$ and $\mu$ be probability measures on $\R^d$.
Assume that
\begin{enumerate}[label=(\roman*)]
    \item $\supp \rho$ is Lipschitz-path connected,
    \item $\supp \mu$ is bounded.
\end{enumerate}
Then optimisers of the dual problem (\ref{intro: c dual problem}) are unique up to a constant.
\end{theorem}
\begin{proof}
Let $\phi_0, \phi_1 \in L^1(\rho)$ be two optimal dual potentials for \eqref{intro: c dual problem}. Take any $x_0, x_1 \in \supp \rho$, let $\omega : [0, 1] \to \supp \rho$ be a Lipschitz curve with $\omega(0) = x_0$ and $\omega(1) = x_1$. We set
 \begin{equation*}
     \tilde\phi_i : [0, 1] \to \R; \quad\tilde\phi_i(t) := \phi_i(\omega(t)).
 \end{equation*}
By Lemma \ref{lem: potential regularity}, each $\phi_i$ is locally Lipschitz and hence each $\tilde\phi_i$ is Lipschitz as a composition of Lipschitz functions. By Rademacher's theorem, both $\tilde \phi_i$ and $\omega$ are differentiable $t$-a.e., so the intersection of these points has full Lebesgue measure on $[0, 1]$. Let $t$ be such a differentiability point.
Fix any optimal $\gamma \in \Gamma(\rho, \mu)$. Since $\supp \mu$ is bounded $p_X( \supp \gamma) = \supp \rho$ where $p_X$ denotes projection onto the $x$ coordinate, see Appendix, Lemma \ref{app: support projection lemma}. Hence for each $t \in [0, 1]$, there exists at least one $y$ such that $(\omega(t), y) \in \supp \gamma$. By the compatibility condition (\ref{intro: compatibility condition}), $y \in \partial^c \phi_i(\omega(t))$. Applying a $c$-superdifferential inequality \eqref{eq: c superdiff} at $\omega(t)$ for $t \in (0, 1)$ gives
\begin{equation*}
    c(\omega(t), y) - \phi_i(\omega(t)) \leq c(\omega(t+h), y) - \phi_i(\omega(t+h)).
\end{equation*}
Considering small $h>0$ and $h<0$, dividing by $h$ and passing to the limit $h \to 0$ which exists by differentiability, we deduce
\begin{equation*}
    \tilde\phi_i'(t) = \langle \nabla_x c(\omega(t), y) | \omega'(t)\rangle.
\end{equation*}
The above shows that for a.e. $t$, the projection of $\nabla_x c(\omega(t), \partial^c \phi_i )$ in the direction $\omega'(t)$ is unique, and specified to be the same for any dual potential. It follows that
\begin{equation*}
    \tilde\phi_1'(t) - \tilde\phi_0'(t) = 0 \text{ for a.e. } t \in [0, 1],
\end{equation*}
and consequently
\begin{equation*}
    \phi_0(x_1) - \phi_0(x_0) = \int_0^1 \tilde\phi_0'(t) \di t = \int_0^1 \tilde\phi_1'(t) \di t = \phi_1(x_1) - \phi_1(x_0)
\end{equation*}
so that $\phi_1(x_0) - \phi_0(x_0) = \phi_1(x_1) - \phi_0(x_1)$ for all $x_0, x_1 \in \supp \rho$. In other words, $\phi_0$ and $\phi_1$ are equal up to a constant.
\end{proof}

\begin{example}
For quadratic cost $c(x, y) = - \langle x | y \rangle$, it is the projection of $\partial\phi_i$ in the direction $\omega'(t)$ which is a.e. unique. In general, however, the set of $t \in [0, 1]$ along a given path for which there are multiple $y \in \partial \phi (\omega(t))$ can have non-zero Lebesgue measure. 
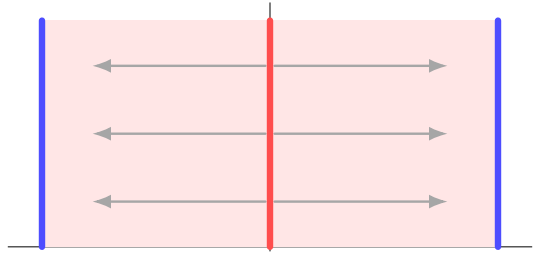
\begin{figure}[ht]
        \centering
\begin{tikzpicture}[>=Latex,scale=3]

  \tikzset{
    rho/.style   ={line width=2.4pt, red!70, line cap=round},
    mu/.style    ={line width=2.4pt, blue!70, line cap=round},
    plan/.style  ={->, line width=0.9pt, gray!70, line cap=round},
  }

  \draw[very thin] (-1.15,0) -- (1.15,0);
  \draw[very thin] (0,-0.02) -- (0,1.08);

  \fill[red!10] (0,0) rectangle (-1,1);
  \fill[red!10] (0,0) rectangle ( 1,1);

  \foreach \y in {0.20,0.50,0.80}{
    \draw[plan] (-0.02,\y) -- (-0.78,\y);
    \draw[plan] ( 0.02,\y) -- ( 0.78,\y);
  }

  \draw[mu]  (-1,0) -- (-1,1);    
  \draw[mu]  ( 1,0) -- ( 1,1);    
  \draw[rho] ( 0,0) -- ( 0,1);    

\end{tikzpicture}

        \caption{Transport between $\rho = \mathcal{H}^1|_{\{0\} \times [0, 1]}$ and $\mu= \frac{1}{2}(\mathcal{H}^1|_{\{-1\} \times [0, 1]} + \mathcal{H}^1|_{\{1\} \times [0, 1]})$. }
        \label{fig:uniqueness of gradient projection}
\end{figure}
The above proof establishes that the set of possible projections of all such $y$ in the direction $\omega'(t)$ is a.e. a singleton. Consider the following example given in \cite[Section 1.4]{santambrogio2015optimal} and illustrated in Figure \ref{fig:uniqueness of gradient projection}. On $\R^2$, let
    \begin{equation*}
        \rho = \mathcal{H}^1|_{\{0\} \times [0, 1]} \quad \text{ and } \quad \mu = \frac{1}{2}\mathcal{H}^1|_{\{-1\} \times [0, 1]} + \frac{1}{2}\mathcal{H}^1|_{\{1\} \times [0, 1]}.
    \end{equation*}
Here, Theorem \ref{thrm: uniqueness Lipschitz connected bounded} applies so we have uniqueness. Despite this, at every point in $\supp\rho$, any optimal potential's subgradient is multivalued; it must contain at least two points corresponding to the horizontal splitting of the mass of the optimal plan (and by convexity it will also contain the full line between them). But the projection of these gradient values vertically (directions which stay in $\supp \rho$) gives a single value.
\end{example}

\begin{example}
    There is no reason why we cannot have uniqueness of the dual problem when the primal problem has non-uniqueness. On $\R^2$, let
    \begin{equation*}
        \rho = \mathcal{H}^1|_{\{0\} \times [-1, 1]} \quad \text{ and } \quad \mu = \mathcal{H}^1|_{[1, 3] \times \{0\}}.
    \end{equation*}
    Here every transport plan $\gamma \in \Pi(\rho, \mu)$ is optimal, since $\langle x_0 - x_1 | y \rangle = 0$ for any $x_0, x_1 \in \supp \rho$ and $y \in \supp \mu$, and so all plans have cyclically monotone support.

\begin{figure}[ht]
\centering
\begin{tikzpicture}[>=Latex,scale=2]

  \tikzset{
    rho/.style   ={line width=2.4pt, red!70, line cap=round},
    mu/.style    ={line width=2.4pt, blue!70, line cap=round},
  }

  \fill[red!10] (0,-1) -- (0, 1) -- (2.5,0) -- cycle;
  \fill[red!10] (0,-1) -- (0, 1) -- (0.5,0) -- cycle;

  \draw[very thin] (-0.2,0) -- (2.7,0);
  \draw[very thin] (0,-1.1) -- (0,1.1);

  \draw[rho] (0,-1) -- (0,1);        
  \draw[mu]  (0.5,0) -- (2.5,0);     

\end{tikzpicture}

\caption{Transport between $\mathcal{H}^1|_{\{0\} \times [-1, 1]}$ and $\mu = \mathcal{H}^1|_{[1, 3] \times \{0\}}$.}
\label{fig:primal-nonunique-dual-unique}
\end{figure}
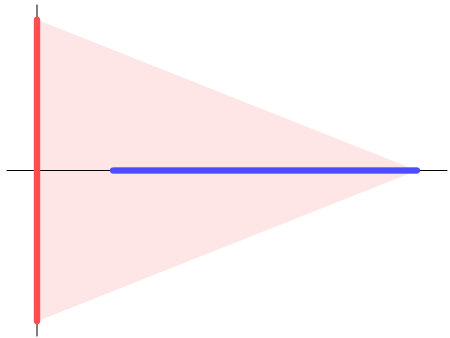
Any plan is optimal, so the primal has non-uniqueness while the dual solution is unique up to a constant.

\end{example}

\begin{remark}
\label{rmk: decomposition into lip path components}
    Assume we have two Lipschitz-path connected sets $A_0, A_1$ in $\supp \rho$, and consider a ``distance" between each defined by
    \begin{equation*}
        d(A_0, A_1) = \inf_{\substack{x_0 \in A_0\\ x_1 \in A_1}} \|x_0 - x_1\|.
    \end{equation*}
    (Of course, this is not a proper distance, as it satisfies neither positive definiteness nor the triangle inequality.)
    If $d(A_0, A_1) = 0$, then in the quadratic case, (still assuming $\supp \mu$ is bounded so that $\phi$ is Lipschitz) $\phi$ is unique up to a constant on $A_0 \cup A_1$. To see this, let $\phi_0$ and $\phi_1$ be two optimal potentials. Theorem \ref{thrm: uniqueness Lipschitz connected bounded} tells us that when restricted to each $A_i$, $\phi_1 = \phi_0 + a_i$ for some $a_i \in \R$.
    Take $\{x_i^n\}_{k=1}^\infty \subset A_i$ points such that $\|x_0^n - x_1^n\| \to 0$ as $n \to \infty$. Then
    \begin{equation*}
        |a_0 - a_1| = |\phi_1(x_0^n) - \phi_0(x_0^n) - (\phi_1(x_1^n) - \phi_0(x_1^n))| \leq 2L\|x_0^n - x_1^n\| \to 0.
    \end{equation*}
    Where $L$ is the Lipschitz constant of the $\phi_i$ (which can be taken to be $R_{\supp\mu}$, the radius of the smallest ball containing $\supp \mu$). Hence, $a_0 = a_1$ and potentials are specified up to a constant on $A_0\cup A_1$ as required. More generally, the above can easily be extended to $\rho$ with support which is connected and is the union of finitely many Lipschitz-path connected components. Here, one can decompose $\rho$ into its Lipschitz-path components, observe that the graph with vertices as these components and an edge if and only if the two components are distance zero apart is connected, then argue as above to deduce that the potential is unique up to a constant across the entire support. The limit of applicability for this argument seems to be to any $\supp \rho$ such that for any $x_0, x_1 \in \supp \rho$,
    \begin{equation*}
        \inf_{\substack{\omega: [0, 1] \to \supp \rho\\ \omega(0) = x_0 \; \omega(1) = x_1}} \left\{ \sum_{k=1}^\infty \|\omega(t^+_k)- \omega(t_k^-)\| : \omega \text{ piecewise Lipschitz with jumps at } t_0, t_1, ... \right\} = 0.
    \end{equation*}
    However, this technique does not seem amenable to a more general connected $\supp\rho$, or even just a path-connected $\supp \rho$ (without the hypothesis that the path should be Lipschitz). In particular, assume $f: [0, 1] \to [0, 1]$ is a sample path of a Brownian motion. Consider $\rho = (\id, f)_\#\text{Leb}|_{[0,1]}$. Then $\supp \rho$ is path connected, but every path connecting two points is locally of infinite length and nowhere differentiable, so the above argument cannot propagate uniqueness along such a curve.
\end{remark}
For $c(x, y) = \|x-y\|^2$, in certain cases, the above argument also allows the removal of the assumption that $\supp \mu$ is bounded.
\begin{definition}
 The relative interior of a set, denoted $\ri$ is the interior taken with respect to the affine hull of the set rather than the full space. Here, the affine hull of $\X \subset \R^d$ is defined as
 \begin{equation*}
     \aff(\X) :=\left\{ \sum_{i=1}^k \lambda_i x_i | k>0, x_i \in \X, \lambda_i \in \R \text{ with } \sum_{i=1}^k \lambda_i = 1\right\}.
 \end{equation*}
\end{definition}

\begin{theorem}\textnormal{\cite[Theorem 10.4]{rockafellar2015convex}.}
\label{thrm: rockafellar local lip}
    Let $f$ be a proper convex function, then $f$ is locally Lipschitz on its restriction to $\ri (\dom f)$.
\end{theorem}

\begin{theorem}
\label{thrm: unbounded target uniqueness}
Let $c(x, y) = \|x-y\|^2$, let $\rho$ and $\mu$ be probability measures on $\R^d$. Assume that $\supp \rho$ is Lipschitz-path connected such that for any $x_0, x_1 \in \supp \rho$ the path can be chosen with
\begin{equation}
\label{eq: interior of convex hull curve condition}
    \omega(t) \in \Int (\conv \supp \rho) \text{ for all } t \in [0, 1].
\end{equation}
Then optimisers of the dual problem (\ref{intro: c dual problem}) are unique up to a constant.
\end{theorem}
\begin{proof}
    This follows almost directly from the proof of Theorem \ref{thrm: uniqueness Lipschitz connected bounded}. Fix a $t$ where \eqref{eq: interior of convex hull curve condition} holds,
    and take $\varepsilon>0$ with $B_\varepsilon(\omega(t)) \subset \Int \conv (\supp \rho)$. By definition of the support, $\rho(B_\varepsilon(\omega(t)))>0$, and so given some optimal plan $\gamma \in \Pi(\rho, \mu)$, we define
    \begin{equation*}
        \tilde\rho = \frac{1}{\rho(B_\varepsilon(\omega(t)))} \rho|_{B_\varepsilon(\omega(t))}
    \end{equation*}
    and its image measure under $\gamma$ denoted $\tilde\mu_\gamma$, defined by
    \begin{equation*}
        \tilde\mu_\gamma(A) = \frac{1}{\rho(B_\varepsilon(\omega(t)))} \gamma(B_\varepsilon(\omega(t))) \times A) \text{ for all measurable } A \subset \R^d.
    \end{equation*}
    By Theorem \ref{thrm: rockafellar local lip}, any optimal potential $\phi$ is Lipschitz when restricted to $B_\varepsilon(\omega(t))$, and so it follows that $\supp \tilde\mu_\gamma$ is bounded. Consequently, we can apply Theorem \ref{thrm: uniqueness Lipschitz connected bounded} to the dual problem between $\tilde\rho$ and $\tilde\mu_\gamma$, which gives that the restrictions of all potentials for the original problem are unique up to a constant on $\supp \rho \cap B_\varepsilon(\omega(t))$. We conclude that potentials are unique globally up to a constant on $\supp \rho$ as required, by choosing overlapping neighbourhoods covering the path $\omega$.
\end{proof}
\begin{remark}
 One could attempt to extend the above to only requiring 
\begin{equation*}
    \omega(t) \in \ri (\conv \supp \rho) \text{ for all } t \in [0, 1],
\end{equation*}
following the same proof only instead using balls $B^\text{ri}_\varepsilon$ with respect to $\aff\supp \rho$ rather than all of $\R^d$. There are two places in the proof of Theorem \ref{thrm: uniqueness Lipschitz connected bounded} uses that $\supp \mu$ is bounded. Once is for the potentials to be Lipschitz, this extends directly to a relative interior assumption by Theorem \ref{thrm: rockafellar local lip}. The second is to apply Lemma \ref{app: support projection lemma} and deduce $\supp \rho \subset p_X(\supp \gamma)$. This requires some work, and one should prove that for an optimal quadratic cost $\gamma$,
\begin{equation*}
    (\ri \conv \supp \rho) \cap \supp \rho \subset p_X(\supp \gamma).
\end{equation*}
\end{remark}

We now present an analogue of the bipartite graph condition Proposition \ref{prop: acciaio discrete dual uniqueness} for the continuous problem.

\begin{theorem}
\label{thrm: uniqueness graph continuous case}
    Let $\rho$ and $\mu$ be probability measures on compact $\X$ and $\Y$ subsets of $\R^d$. Let $c \in C^1(\X \times \Y)$. Assume that each support is a finite union of Lipschitz-path connected components $(A_i)_{i \in I}$ and $(B_j)_{j \in J}$ respectively. Define the bipartite graph $G = (V, E)$ with $V = (A_i)_{i \in I} \cup (B_j)_{j \in J}$ and 
    \begin{equation*}
        (A_i, B_j) \in E \iff \text{ there exists an optimal } \gamma \in \Pi(\rho, \mu) \text{ with } \gamma(A_i \times B_j) > 0.
    \end{equation*}
    If $G$ is connected, then solutions to the dual problem (\ref{intro: c dual problem}) are unique up to a constant. 
\end{theorem}
\begin{proof}
Let $\phi_0, \phi_1 \in C_b(\X)$ be two Brenier potentials, let $\psi_0, \psi_1 \in C_b(\Y)$ be their corresponding dual potentials. In general, if $\gamma \in \Pi(\rho, \mu)$ is optimal, then any $\tilde\gamma \leq \gamma$ is also optimal for the $c$ cost transport between its marginals, see Appendix, Lemma \ref{app: subplan is optimal}. Consequently, the compatibility condition (\ref{intro: compatibility condition}) implies that if $\phi \in L^1(\rho)$ is optimal for the dual problem over $\Pi(\rho, \mu)$, then it is optimal for the dual problem over $\Pi(\tilde\rho, \tilde\mu)$, since $\supp \tilde \gamma \subset \supp \gamma \subset \partial \phi$. 

Both $\X$ and $\Y$ are compact and $c\in C^1$ so $c$ is also Lipschitz, and hence satisfies Assumption \ref{assumption: cost hypotheses} in both variables. Thus by Theorem \ref{thrm: uniqueness Lipschitz connected bounded}, there exist $a_i, b_j \in \R$ such that on each component $A_i$, $\phi_0 = \phi_1 + a_i$ and on each component $B_j$, $\psi_0 = \psi_1 + b_j$. We want to show that all $a_i$ are equal to some $a \in \R$, and consequently each $b_j = -a$ also. 

Fix some $i, i'$, let $x \in A_i$ and $x' \in A_{i'}$. We will show that $a_i = \phi_0(x) - \phi_1(x) = \phi_0(x') - \phi_1(x') = a_{i'}$. Let $i=i_1, j_1, i_2, j_2,..., j_{n-1}, i_n = i'$ be a path between $A_i$ and $A_{i'}$ in $G$. Then there exist
\begin{equation*}
    x=x_{1}, \hat x_{1}, y_{1}, \hat y_{1}, x_{2}, \hat x_{2},..., y_{{n-1}}, \hat y_{{n-1}}, x_{n}, \hat x_{n} = x'
\end{equation*}
with $x_{k}, \hat x_{k} \in A_{i_k}$ and $y_{k}, \hat y_{k} \in B_{i_k}$ such that for each $k=1,...,n-1$ there exists $\gamma \in \Pi(\rho, \mu)$ optimal with $(\hat x_{k}, y_{k}) \in \supp \gamma \cap (A_{i_k} \times B_{j_k})$ and there exists $\gamma \in \Pi(\rho, \mu)$ optimal with $(x_{{k+1}}, \hat y_{k}) \in \supp \gamma \cap (A_{i_{k+1}} \times B_{j_k})$. Consequently, $(\hat x_{k}, y_{k}), (x_{{k+1}}, \hat y_{k}) \in \partial \phi_i$ and so for $l = 0, 1$,
\begin{equation}
    \phi_l(\hat x_{k}) + \psi_l(y_k) = \langle\hat{x}_k | y_k \rangle \quad \text{ and }\quad \phi_l(x_{k+1}) + \psi_l(\hat y_k) = \langle \hat{x}_{k+1} | y_k \rangle.
\end{equation}
It follows that
\begin{align*}
    a_i &= \phi_0(x) - \phi_1(x) = \phi_0(\hat x_1) - \phi_1(\hat x_1) = \psi_1(y_1) - \psi_0(y_1)\\ &= \psi_1(\hat y_1) - \psi_0( \hat y_1) = \phi_0(x_2) - \phi_1( x_2) = \cdots = \phi_0(\hat x_n) - \phi_1(\hat x_n) = a_{i'},
\end{align*}
and so $a_i = a_{i'} = -b_{j} = -b_{j'}$ for all $i, j \in I \times J$. In other words, the dual potentials are unique up to a constant as required.
\end{proof}

\begin{remark}
    In \cite{staudt2025uniqueness}, the authors establish uniqueness under a hypothesis they refer to as \textit{non-degeneracy}; that there exist no subsets of indices $I' \subset I$ and $J' \subset J$ such that
\begin{equation}
    \label{intro:eq: german nondegeneracy}
    0<\rho( \cup_{i \in I'} A_i) = \mu( \cup_{j \in J'} B_j)<1. 
\end{equation}
This hypothesis is much stronger; it forces that \textit{any} $\gamma \in \Pi(\rho, \mu)$, its support graph $G_\gamma$ is connected. To see this, take any $\gamma \in \Pi(\rho, \mu)$ and assume there exist at least two distinct connected components of $G_\gamma$, choose one, and set $I'$ and $J'$ to be all the indices of vertices corresponding to this component. By definition of $G_\gamma$, for any $i \in I'$ and $j \notin J'$, $\gamma(A_i \times B_j) = 0 $ and similarly for $i \notin I'$ and $j \in J'$. Consequently
\begin{equation*}
    \rho( \cup_{i \in I'} A_i) = \gamma( \cup_{i \in I'} A_i \times \Y) = \gamma( \cup_{i \in I'} A_i \times \cup_{j \in J'} B_j) = \gamma( \X \times \cup_{j \in J'} B_j) = \mu( \cup_{j \in J'} B_j).
\end{equation*}
By the assumption that $G_\gamma$ has more than one component, each of $I'$ and $J'$ correspond to a proper subset of the sets of support components of $\rho$ and $\mu$ respectively, and hence the $0< \rho( \cup_{i \in I'} A_i) = \mu( \cup_{j \in J'} B_j) < 1$.
\end{remark}

\begin{remark}
    It would be interesting to consider if the above can be extended to the multi-marginal transport problem, considering the connectivity of an $n$-partite hypergraph.
\end{remark}

When there are an infinite number of components, there can exist $A_i$ with $\rho(A_i)= 0$ and $B_j$ with $\mu(B_j) = 0$. This can occur when these components are limit sets of points from other components, see Example \ref{ex: uniqueness despite disconnected graph}.
\begin{example}[Uniqueness of potentials despite disconnected graph]
\label{ex: uniqueness despite disconnected graph}
    In $\R^2$, set $\rho = \frac{1}{2}\mathcal{H}^1_{\{-1\} \times [0, 1]} + \frac{1}{2}\mathcal{H}^1_{\{1\} \times [0, 1]}$. Define the sequence of points $\{x_n\}_{n \in \N} \subset \R^2$ by
    \begin{equation*}
        x_{2n -1} := \left(-\frac{1}{2n}, \frac{1}{n}\right); \quad  x_{2n} := \left(\frac{1}{2n}, \frac{1}{n}\right),
    \end{equation*}
    then set
    \begin{equation*}
        \mu = \sum_{n = 1}^\infty 2^{-n}(\delta_{x_{2n-1}} + \delta_{ x_{2n}}),
    \end{equation*}
see Figure \ref{fig: disconnected graph preserves uniqueness of potentials}.
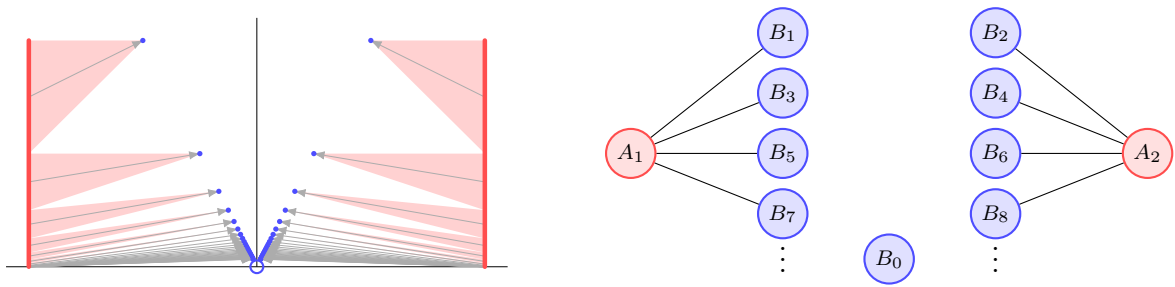
\begin{figure}[h]
    \centering
    
    \begin{minipage}{0.48\linewidth}
      \centering
    \begin{tikzpicture}[>=Latex,scale=3]
    
    \tikzset{
      rho/.style   ={line width=1.7pt, red!70, line cap=round},
      mupt/.style  ={circle, fill=blue!70, inner sep=0.7pt},
      plan/.style  ={->, very thin, gray!65},
    }
    
    \draw[very thin] (-1.1,0) -- (1.1,0);
    \draw[very thin] (0,0) -- (0,1.1);
    
    \foreach \n in {1,...,7}{
      \pgfmathsetmacro{\y}{1/(\n)}
      \pgfmathsetmacro{\xL}{-1/(2*\n)}
      \pgfmathsetmacro{\xR}{ 1/(2*\n)}
      \pgfmathsetmacro{\ylow}{pow(2,-\n)}
      \pgfmathsetmacro{\yhigh}{pow(2,-(\n-1))}
      \pgfmathsetmacro{\ymid}{(\ylow+\yhigh)/2}
    
      \fill[red!18] (-1,\ylow) -- (-1,\yhigh) -- (\xL,\y) -- cycle;
      \fill[red!18] ( 1,\ylow) -- ( 1,\yhigh) -- (\xR,\y) -- cycle;
    
      \draw[plan] (-1,\ymid) -- (\xL,\y);
      \draw[plan] ( 1,\ymid) -- (\xR,\y);
    
      \node[mupt] at (\xL,\y) {};
      \node[mupt] at (\xR,\y) {};
    }
    
    \foreach \n in {8,...,30}{
      \pgfmathsetmacro{\y}{1/(\n)}
      \pgfmathsetmacro{\xL}{-1/(2*\n)}
      \pgfmathsetmacro{\xR}{ 1/(2*\n)}
      \pgfmathsetmacro{\ylow}{pow(2,-\n)}
      \pgfmathsetmacro{\yhigh}{pow(2,-(\n-1))}
      \pgfmathsetmacro{\ymid}{(\ylow+\yhigh)/2}
    
      \draw[plan] (-1,\ymid) -- (\xL,\y);
      \draw[plan] ( 1,\ymid) -- (\xR,\y);
    
      \node[mupt] at (\xL,\y) {};
      \node[mupt] at (\xR,\y) {};
    }
    
    \draw[blue!70, thick] (0,0) circle (0.8pt);
    
    \draw[rho] (-1,0) -- (-1,1);
    \draw[rho] ( 1,0) -- ( 1,1);
    \end{tikzpicture}
    \end{minipage}\hfill
    \begin{minipage}{0.48\linewidth}
    \centering
    \begin{tikzpicture}[baseline=(current bounding box.center), node distance=7mm]

\def\nodesize{6.5mm}   
\def\xsep{1.4cm}
\def\outersep{2cm}
\def\yTop{1.55cm}
\def\yStep{0.8cm}
\def\drop{0.6cm}

\tikzset{
  redcircle/.style  ={circle, draw=red!70, thick, fill=red!12,  minimum size=\nodesize, inner sep=0pt},
  bluecircle/.style ={circle, draw=blue!70, thick, fill=blue!12, minimum size=\nodesize, inner sep=0pt},
  lab/.style        ={font=\scriptsize, inner sep=0pt, outer sep=0pt},
  edge/.style       ={thin}
}

\newcommand{\fixnode}[4]{
  \node[#2] (#1) at #3 {};
  \node[lab] at (#1) {#4};
}

\fixnode{x1}{bluecircle}{(-\xsep,\yTop)}{$B_1$}
\fixnode{x2}{bluecircle}{($(x1)-(0,\yStep)$)}{$B_3$}
\fixnode{x3}{bluecircle}{($(x2)-(0,\yStep)$)}{$B_5$}
\fixnode{x4}{bluecircle}{($(x3)-(0,\yStep)$)}{$B_7$}
\node at ($(x4)-(0,0.5cm)$) {$\vdots$};

\fixnode{xt1}{bluecircle}{(\xsep,\yTop)}{$B_2$}
\fixnode{xt2}{bluecircle}{($(xt1)-(0,\yStep)$)}{$ B_4$}
\fixnode{xt3}{bluecircle}{($(xt2)-(0,\yStep)$)}{$ B_6$}
\fixnode{xt4}{bluecircle}{($(xt3)-(0,\yStep)$)}{$ B_8$}
\node at ($(xt4)-(0,0.5cm)$) {$\vdots$};

\fixnode{A1}{redcircle}{($(x3)+(-\outersep,0)$)}{$A_1$}
\fixnode{A2}{redcircle}{($(xt3)+(\outersep,0)$)}{$A_2$}

\foreach \k in {1,2,3,4}{
  \draw[edge] (A1) -- (x\k);
  \draw[edge] (A2) -- (xt\k);
}

\fixnode{O}{bluecircle}{($ (x4)!0.5!(xt4) + (0,-\drop) $)}{$B_0$}

\end{tikzpicture}
    \end{minipage}

    \caption{Transport between $\rho$ and $\mu$ (Left), and the Bipartite support graph (Right). The graph is disconnected, but the uniqueness of potentials holds as the components are ``asymptotically connected" via the point $(0, 0)$.}
    \label{fig: disconnected graph preserves uniqueness of potentials}
\end{figure}
The connected components of $\supp \rho$ are $A_1 = \{-1\} \times [0, 1]$ and $A_2 = \{1\} \times [0, 1]$. The connected components of $\supp \mu$ are $B_0 = \{(0, 0)\}$ and $B_n = \{x_n\}$ for $n \in \N$. Here, the component $B_0$ receives no mass from $\mu$, but belongs to the support nonetheless as a limit point of points from the other components. The graph for the transport has three disjoint components, one from the left pairings of points and one from the right. Theorem \ref{thrm: uniqueness graph continuous case} gives that the potentials are defined up to a constant on each component of the graph. Then continuity of any Brenier potential at $(0, 0)$ forces uniqueness up to a constant globally, as a consequence of the arguments presented in Remark \ref{rmk: decomposition into lip path components}.
\end{example}

One can also use calculations similar to those in the proof of Lemma \ref{lemma: balinski dual extreme vertices characterisation} to ascertain quantitative uniqueness results, i.e. bounds on the diameter of the optimal set. In particular, we have the following
\begin{proposition}
    Let $\rho$ and $\mu$ be probability measures on $\R^d$, let $c(x, y) = \|x-y\|^2$. Assume that each support is a finite union of Lipschitz-path connected components $(A_i)_{i \in I}$ and $(B_j)_{j \in J}$ respectively. Assume that the graph $G$ defined in Theorem \ref{thrm: uniqueness graph continuous case} consists of two connected components $G = G_0 \sqcup G_1$. Then for any two Brenier potentials $\phi_0, \phi_1$, we have
    \begin{equation}
    \label{eq: basic quantitative uniqueness statement}
        \inf_{\lambda \in \R} \| \phi_0 - \phi_1 - \lambda\|_{\infty} \leq \frac{1}{2}\inf_{\substack{(x_0, y_0) \in \supp \Gamma \cap G_0\\ (x_1, y_1) \in \supp \Gamma \cap G_1}} \langle x_1 - x_0 | y_1 - y_0 \rangle.
    \end{equation}
\end{proposition}
\begin{proof}
Assume we know the values of a Brenier potential $\phi$ on the $G_0$ component of $G$. For any $(x_0, y_0) \in \supp \Gamma \cap G_0$ and $(x_1, y_1) \in \supp \Gamma \cap G_1$, we have $y_0 \in \partial \phi(x_0)$ and $y_1 \in \partial \phi(x_1)$. By the subdifferential inequality applied at each $(x_i, y_i)$, the smallest and largest $\phi$ could be at $x_1$, in terms of $\phi(x_0)$, is
\begin{equation}
    \phi(x_1) \in \Big[\phi(x_0) + \langle y_0 | x_1 - x_0 \rangle,\, \phi(x_0) + \langle y_1 | x_1 - x_0 \rangle\Big].
\end{equation}
Hence for two potentials $\phi_0, \phi_1$ with $\phi_0(x_0) = \phi_1(x_0)$, rearranging the above gives
\begin{equation*}
    |\phi_0(x_1) - \phi_1(x_1)| \leq \langle x_1 - x_0 | y_1 - y_0 \rangle.
\end{equation*}
Since on each component of $G$ the potentials are unique up to a constant,
\begin{equation*}
    |\phi_0(x_1) - \phi_1(x_1) | = |\phi_0(x_1') - \phi_1(x_1') |
\end{equation*}
for all $x_1, x_1' \in G_1$, and so if the difference between two potentials is bounded at one place in $G_1$, the same bound holds everywhere. (Note here by $x \in G_1$ we mean $x$ belongs to the union of the $A_i$ vertices in the component $G_1$.) Thus, all differences globally on $G_1$ are bounded by the tightest possible bound at any one point in the component, so we can pass to the infimum. Hence (\ref{eq: basic quantitative uniqueness statement}) follows after dividing by $2$ since we can optimise $\lambda$ so the displacement is centred, rather than $\phi_0 = \phi_1$ on $G_0$.
\end{proof}
\begin{remark}
    The right-hand side of (\ref{eq: basic quantitative uniqueness statement}) can be bounded using Cauchy-Schwarz to give
    \begin{equation*}
        \inf_{\lambda \in \R} \| \phi_0 - \phi_1 - \lambda\|_{\infty} \leq \frac{1}{2} \min( \varepsilon_\X R_\Y, \varepsilon_\Y R_\X),
    \end{equation*}
    where $R_\X$ and $R_\Y$ are the smallest radii of balls centred at the origin containing $\supp \rho$ and $\supp \mu$ respectively, and $\varepsilon_\X$ is the largest pairwise distance attainable by partitioning $(A_i)$ into two components,
    \begin{equation*}
        \varepsilon_\X = \inf_{\aleph_1\sqcup \aleph_2 = \{ A_i\}_{i \in I}} d\Big(\bigcup_{A \in \aleph_1} A,\, \bigcup_{A \in \aleph_2}A\Big)
    \end{equation*}
    with $d$ denoting the closest distance between sets defined in Remark \ref{rmk: decomposition into lip path components}, and $\varepsilon_\Y$ is defined symmetrically.

    For example, suppose $\supp \rho$ has two Lipschitz-path connected components $A_0$ and $A_1$, then we have the following bound on the uniform norm diameter of the optimal set
    \begin{equation*}
        \sup_{\substack{\phi_0, \phi_1\\\text{optimal}}}\inf_{\lambda \in \R} \| \phi_0 - \phi_1 - \lambda\|_{\infty} \leq \frac{1}{2} d(A_0, A_1) \diam(\supp \mu).
    \end{equation*}
    Where here we replaced $R_\Y$ with $\diam \Y$ due to the translation invariance of quadratic optimal transport, as in Chapter \ref{ch: ac stability}.

    When there are more than two connected components of $G$, one should instead consider the graph of connected components of $G$ given the above distance quantity, and the diameter of the optimal set becomes less clear. One can see parallels between this calculation and those of Proposition \ref{prop: dual set characterisation}. A better understanding of this type of graph structure is the subject of ongoing work.
\end{remark}

\chapter{Conclusions}

The contributions of this work are:

\textbf{Chapter \ref{ch: intro}:} The observations and preliminary results on the stability of optimal transport plans.

\textbf{Chapter \ref{ch: ac stability}:} The extension of the strong convexity inequality to a larger class of measures using glueing, which implies stability of Wasserstein barycentres for a large class of measures.

\textbf{Chapter \ref{ch: discrete transport problem}:} Unifying known literature on the discrete primal and dual problems. Explicit characterisation of the set of dual optimisers for the discrete transport problem. (Some unfinished calculations have been omitted here, which seem to suggest this formulation could be very helpful for studying the quantitative stability of $\Phi$ and $\Psi$.)

\textbf{Chapter \ref{ch: pertubations}:} A qualitative understanding of how fully discrete transport plans behave under perturbation of source support points - quantitatively stable as long as uniqueness is preserved, at which case the optimal plans may have jump discontinuities corresponding to changing vertex on the polyhedron $\Pi(\alpha, \beta)$. This is the first step towards a resolution of Problem \ref{problem: general plan stability} under the discrete hypothesis on $\rho$.

\textbf{Chapter \ref{ch: uniqueness of dual problem}:} Significant weakening of source hypotheses needed for uniqueness of optimisers for the dual problem. Lipschitz-path connected spaces are a huge class of spaces, including any connected differentiable submanifold of $\R^d$. Formalised graph-based uniqueness criteria in the continuous setting. First calculations in the direction of Quantitative uniqueness results (bounds on the diameter of the optimal set).

The main takeaway of this report is that the connections between linear programming and optimal transport should not be overlooked. After a period of intense interaction in the 20th century, modern optimal transport has, in recent decades, moved closer to analysis and PDEs. As shown in Chapter \ref{ch: discrete transport problem}, the combinatorial characterisations of extreme points provide valuable intuition for the discrete transport problem, motivating the novel results of Chapters \ref{ch: pertubations} and \ref{ch: uniqueness of dual problem}. This demonstrates that linear programming and combinatorial perspectives remain relevant for advancing our theoretical understanding of optimal transport. Problems \ref{problem: general plan stability}, \ref{problem: discrete plan stability}, and \ref{problem: discrete potential stability} remain wide open, and I look forward to tackling them during the course of my PhD...

\bibliographystyle{plain}
\bibliography{ref2}

\appendix
\chapter{Appendix}

\begin{lemma}
\label{app: potential regularity}
    Assume $c$ satisfies Assumption \ref{assumption: cost hypotheses}, and $\Y$ is compact. Then given any $\psi \in C_b(\Y)$, $\psi^c$ defined as in \eqref{intro: c transf definition} is Locally Lipschitz on $\X$, and does not take the value $-\infty$ anywhere.
\end{lemma}
\begin{proof}
    Since $\Y$ is compact, the functions
    \begin{equation*}
        x \mapsto c(x, y)
    \end{equation*}
    are locally uniformly Lipschitz by assumption. To see this, Fix some $K_\X \subset \X$ compact, and some $x_0 \in K_\X$, then by continuity of $c$ in $y$, 
    \begin{equation*}
        y \mapsto c(x_0, y)
    \end{equation*}
    Attains its maximum and minimum over $\Y$ by compactness. Since all $x \mapsto c(x, y)$ share the same Lipschitz constant on $K_\X$, it follows that the  $x \mapsto c(x, y)$ are uniformly bounded on $K_\X$. Consequently,
    \begin{equation*}
        |\psi^c(x)| = |\inf_{y \in \Y} c(x, y) - \psi(y)| \leq  \|c(x_0, \cdot)\|_\infty + L_{K_\X}\diam \Y + \|\psi\|_\infty \text{ for all } x \in K_\X,
    \end{equation*}
    where $L_{K_\X}$ is the Lipschitz constant of the equicontinuous family $c( \cdot, y)$ on $K_\X$.
    It follows that $\psi^c$ is bounded on each compact set in $\X$, and $L_{K_\X}$ Lipschitz by classical envelope arguments as an infimum of uniformly Lipschitz functions, see \cite[Box 1.8]{santambrogio2015optimal}.
\end{proof}

\begin{lemma}
\label{app: support projection lemma}
    Let $\rho$ and $\mu$ be probabilities on $\R^d$, and let $\gamma\in \Pi(\rho, \mu)$ with $\supp \mu$ bounded. Then $p_X(\supp \gamma) = \supp \rho$, so that for each $x \in \supp \rho$ there exists $y$ with $(x, y) \in \supp \gamma$. 
\end{lemma}
\begin{proof}
    One inclusion always holds without any hypotheses: Since $\supp \rho \times \R^d$ is a closed set of full measure, $\supp \gamma \subset \supp \rho \times \R^d$, and hence $p_X(\supp \gamma)  \subset \supp \rho$.
    The inverse inclusion does not always hold in general, as it could be that $p_X(\supp \gamma)$ is not closed. For example, $\X=\Y = \R$ and $\gamma$ supported on the graph of $\{ (x, 1/x) : x >0\}$, then $p_X(\supp \gamma) = (0, \infty)$, whereas the support of the $\X$ marginal is $[0, \infty)$.

    Assume $\supp \mu$ is bounded. Necessarily $p_X(\supp \gamma)$ has full $\rho$ mass, and so it holds that
    \begin{equation*}
        \supp \rho \subset \overline{p_X(\supp \gamma)}.
    \end{equation*}
    Thus, it suffices to show that $p_X(\supp \gamma)$ is closed. Take $x_n \in p_X(\supp \gamma)$, with $x_n \to x$ then there exist $y_n$ with $(x_n, y_n) \in \supp \gamma$, by compactness of $\supp \mu$ we can assume that $y_n$ converges to some $y$, and hence $(x, y) \in \supp \gamma$ by closedness of $\supp \gamma$. In other words, $x \in p_X(\supp \gamma)$, establishing the result.
\end{proof}

\begin{lemma}
\label{app: subplan is optimal}
If $\gamma \in \Pi(\rho, \mu)$ is an optimal transport plan for cost $c: \R^d \times \R^d \to \R_+$ integrable and $\pi\in \mathcal{M}^+(\R^d \times \R^d)$ such that $\pi \leq \gamma$, then
\begin{equation*}
    \hat\pi := \frac{\pi}{\pi(\R^d \times \R^d)}
\end{equation*}
is an optimal transport plan between its marginals $\hat\rho = p_{X\#}\hat\pi$ and $\hat\mu = p_{Y\#}\hat\pi$.
\end{lemma}
\begin{proof}
Assume $\hat\pi$ is not optimal, then there exists a transport plan $w$ with marginals $\hat\rho, \hat\mu$ such that $\int c \di w  < \int c \di \hat\pi$. Thus the transport plan between marginals $\rho$ and $\mu$ defined by
\begin{equation*}
    \hat\gamma = \hat\pi(\R^d \times \R^d)w + (\gamma - \pi)
\end{equation*}
has cost
\begin{align*}
    \int c(x, y) \di \hat\gamma =& \hat\pi(\R^d \times \R^d)\int c(x, y) \di w + \int c(x, y) \di (\gamma - \pi)\\
    <&\hat\pi(\R^d \times \R^d)\int c(x, y) \di \hat\pi + \int c(x, y) \di (\gamma - \pi)\\
    =&\int c(x, y) \di \gamma,
\end{align*}
contradicting the optimality of $\gamma$. (For the well definition of $\hat\gamma$ as a transport plan it was crucial $\pi \leq \gamma$ to ensure $\hat\gamma \geq 0$.)
\end{proof}

\end{document}